\documentclass[12pt,a4paper]{amsart}
\usepackage[latin1]{inputenc}
\usepackage{amsmath}
\usepackage{amsfonts}
\usepackage{amssymb}
\usepackage{amsthm}
\usepackage{anysize}
\usepackage{enumitem}
\usepackage{amscd}
\usepackage{hyperref}
\usepackage[square, comma, numbers, sort&compress]{natbib}
\usepackage{indentfirst}
\marginsize{3.0cm}{3.0cm}{3.0cm}{3.0cm}
\setenumerate{label={\normalfont(\arabic*)}}
\newcommand{\coloneqq}{\mathrel{\mathop:}=}
\newcommand{\form}[1]{{\langle #1 \rangle }}
\newcommand{\pfister}[1]{{\langle \! \langle #1 \rangle \! \rangle}}
\newcommand{\mydim}[1]{{\mathrm{dim}\!\; #1}}
\newcommand{\witti}[2]{{\mathfrak{i}_{#1}(#2)}}
\newcommand{\wittj}[2]{{\mathfrak{j}_{#1}(#2)}}
\newcommand{\anispart}[1]{{#1_{\mathrm{an}}}}
\newcommand{\qp}[1]{{#1_{\mathrm{qp}}}}
\newcommand{\simform}[1]{{#1_{\mathrm{sim}}}}
\newcommand{\normform}[1]{{#1_{\mathrm{nor}}}}

\newtheorem{theorem}{Theorem}[section]

\newtheorem{lemma}[theorem]{Lemma}
\newtheorem{proposition}[theorem]{Proposition}
\newtheorem{corollary}[theorem]{Corollary}
\newtheorem{conjecture}[theorem]{Conjecture}

\theoremstyle{definition}

\newtheorem{example}[theorem]{Example}
\newtheorem{examples}[theorem]{Examples}

\newtheorem{question}[theorem]{Question}

\theoremstyle{remark}
\newtheorem{remark}[theorem]{Remark}
\newtheorem{remarks}[theorem]{Remarks}

\numberwithin{equation}{section}
\begin{document}

\title{Hoffmann's conjecture for totally singular forms of prime degree}
\author{Stephen Scully}
\address{Department of Mathematical and Statistical Sciences, University of Alberta, Edmonton AB T6G 2G1, Canada}
\email{stephenjscully@gmail.com}


\begin{abstract} One of the most significant discrete invariants of a quadratic form $\phi$ over a field $k$ is its (\emph{full}) \emph{splitting pattern}, a finite sequence of integers which describes the possible isotropy behaviour of $\phi$ under scalar extension to \emph{arbitrary} overfields of $k$. A similarly important, but more accessible variant of this notion is that of the \emph{Knebusch splitting pattern} of $\phi$, which captures the isotropy behaviour of $\phi$ as one passes over a certain \emph{prescribed tower} of $k$-overfields. In this paper, we determine all possible values of this latter invariant in the case where $\phi$ is \emph{totally singular}. This includes an extension of Karpenko's theorem (formerly Hoffmann's conjecture) on the possible values of the first Witt index to the totally singular case. Contrary to the existing approaches to this problem (in the nonsingular case), our results are achieved by means of a new structural result on the higher anisotropic kernels of totally singular quadratic forms. Moreover, the methods used here readily generalise to give analogous results for arbitrary Fermat-type forms of degree $p$ over fields of characteristic $p>0$. Related problems concerning the Knebusch splitting of symmetric bilinear forms over fields of characteristic 2 and the \emph{full} splitting of totally singular quadratic forms are also considered. \end{abstract}

\maketitle

\section{Introduction} \label{Introduction}

Let $k$ be a field, let $\phi$ be a nonzero quadratic form on a (nonzero) $k$-vector space $V$ of finite dimension and let $X_\phi \subseteq \mathbb{P}(V)$ denote the projective $k$-scheme defined by the vanishing of $\phi$. Given a $k$-linear subspace $W$ of $V$, we write $\phi|_W$ for the restriction of $\phi$ to $W$. In the case where $\phi|_W$ is the zero form, $W$ is said to be \emph{totally isotropic} (with respect to $\phi$). The largest integer among the dimensions of all totally isotropic subspaces of $V$ is called the \emph{isotropy index} of $\phi$, and is denoted by $\witti{0}{\phi}$. In the special case where $\phi$ is \emph{nonsingular} (i.e., where $X_\phi$ is smooth), $\witti{0}{\phi}$ is more commonly known as the \emph{Witt index} of $\phi$, and is bounded from above by the integer part of $\mydim{\phi}/2$ (where $\mydim{\phi}$ denotes the dimension of the $k$-vector space $V$). By contrast, in the opposite extreme where $\phi$ is \emph{totally singular} (i.e., where $X_\phi$ has no smooth points at all), $\witti{0}{\phi}$ may take any value between $0$ and $\mydim{\phi} - 1$.

Assume now that $\phi$ is anisotropic (i.e., that $\witti{0}{\phi}=0
$). A simple, yet fundamentally important invariant of $\phi$ is its (\emph{full}) \emph{splitting pattern}, which may be defined as the increasing sequence of nonzero isotropy indices attained by $\phi$ under scalar extension to \emph{every} overfield of $k$\footnote{\label{Footnote1}This terminology is not standard -- see, e.g., \cite{HoffmannLaghribi1}, where refined splitting pattern invariants are considered. Our definition does, however, agree (in content if not presentation) with those found in the literature when $\phi$ is nonsingular or totally singular (the only relevant cases here).}. Although this sequence appears to be somewhat intractable in general, its first entry (assuming the sequence is nonempty) may be computed explicitly as the isotropy index $\witti{0}{\phi_L}$ of $\phi$ extended to the function field $L = k(X_{\phi})$ of the (integral) quadric $X_\phi$. This (almost tautological) observation is the basic motivation underlying the following construction, originally due to Knebusch (cf. \cite{Knebusch1}): Let $k_0 = k$, $\phi_0 = \phi$, and inductively define $k_r = k_{r-1}(X_{\phi_{r-1}})$, $\phi_r = \anispart{(\phi_{k_r})}$\footnote{For any form $\psi$ over a field $K$, $\anispart{\psi}$ denotes the \emph{anisotropic kernel} of $\psi$, an anisotropic $K$-form uniquely determined up to isomorphism by the (refined) Witt decomposition of $\psi$ - see \cite[\S 2]{HoffmannLaghribi1}.}, with the understanding that this (finite) process stops when we reach the first nonnegative integer $h(\phi)$ such that $\mydim{\phi_{h(\phi)}} \leq 1$. The integer $h(\phi)$ and the tower of fields $k = k_0 \subset k_1 \subset \hdots \subset k_{h(\phi)}$ are known as the \emph{height} and \emph{Knebusch splitting tower} of $\phi$, respectively. For $1 \leq r \leq h(\phi)$, the anisotropic form $\phi_r$ is called the \emph{r-th higher anisotropic kernel} of $\phi$. The \emph{r-th higher isotropy index of $\phi$}, denoted $\witti{r}{\phi}$, is defined as the difference $\witti{0}{\phi_{k_r}} - \witti{0}{\phi_{k_{r-1}}}$. The sequence $\mathfrak{i}(\phi) = \big(\witti{1}{\phi},\hdots,\witti{h(\phi)}{\phi}\big)$ is called the \emph{Knebusch splitting pattern} of $\phi$\footnote{\label{Footnote2}The term \emph{standard splitting pattern} is also used in the literature. Footnote \ref{Footnote1} again applies here.}. Note that we have $\witti{r}{\phi} = \witti{1}{\phi_{r-1}}$ for every $r \geq 2$ by the inductive nature of the construction.

If $\phi$ is nonsingular, then its full and Knebusch splitting patterns are easily seen to determine one another (cf. \cite[Prop. 25.1]{EKM}). This need not be the case for (totally) singular forms (see Example \ref{Exnongeneric} below), but Knebusch's construction nevertheless offers a meaningful and practical way to pre-classify quadratic forms according to some notion of ``algebraic complexity''. By virtue of its definition, the Knebusch splitting pattern thus embodies a fundamental link between intrinsic algebraic properties of quadratic forms and the geometry of the algebraic varieties which are naturally associated to them. In recent years, the advent of effective new tools with which to study algebraic cycles on projective homogeneous varieties has, in this way, led to dramatic progress on many long-standing problems within the algebraic theory of quadratic forms. The impact of these developments has been felt most deeply in the characteristic $\neq 2$ theory, where (1) anisotropic forms of dimension $\geq 2$ are necessarily nonsingular, and (2) the geometric methods are better developed, even if we restrict our considerations to nonsingular forms only -- see \cite{EKM} for a thorough exposition of much of the work which has been done in this area in the last two decades. One of the central problems in the investigation of splitting properties of quadratic forms over general fields is the following:

\begin{question} \label{Quevalues} Let $\phi$ be an anisotropic quadratic form of dimension $\geq 2$ over a field. What are the possible values of the sequence $\mathfrak{i}(\phi)$? \end{question}

Since $\witti{r}{\phi} = \witti{1}{\phi_{r-1}}$ for all $2 \leq r \leq h(\phi)$, the natural first approximation to this problem is to determine the possible values of the invariant $\mathfrak{i}_1$ among all forms of a given dimension. To this end, we have the following general conjecture:

\begin{conjecture}[Hoffmann\footnote{This conjecture was originally stated by Hoffmann under the additional hypothesis that $\mathrm{char}(k) \neq 2$, but we expect that the assertion is also valid in characteristic 2.}] \label{ConjHoffmanns} Let $\phi$ be an anisotropic quadratic form of dimension $\geq 2$ over a field. Then $\witti{1}{\phi} - 1$ is the remainder modulo $2^s$ of $\mydim{\phi} -1$ for some $s < \mathrm{log}_2(\mydim{\phi})$\footnote{After passing to a purely transcendental extension of $k$ if necessary, all values of $\witti{1}{\phi}$ which are not excluded by the conjecture can be realised by making an appropriate choice of $\phi$. In all cases, $\phi$ may, in fact, be chosen to be either nonsingular or (in the case where $\mathrm{char}(k) = 2$) totally singular -- see \cite[\S 7.4]{Vishik4} and Proposition \ref{Propexistenceofvalues} below.}. \end{conjecture}

Over fields of characteristic $\neq 2$, the first major result in the direction of Conjecture \ref{ConjHoffmanns} was established by Hoffmann himself, who showed that if $\mydim{\phi} = 2^n + m$ for nonnegative integers $n$ and $1 \leq m \leq 2^n$, then $\witti{1}{\phi} \leq m$ (\cite[Cor. 1]{Hoffmann1}). A few years later, Izhboldin (\cite[Cor. 5.12]{Izhboldin1}) proved that one cannot have $\witti{1}{\phi} = m-1$ here unless $m=2$. Izhboldin's paper combined an elaboration of the algebraic methods conceived in \cite{Hoffmann1} with emerging work of Vishik (\cite{Vishik1},\cite{Vishik2}), who developed a systematic approach to the study of the splitting pattern using the (integral) motive of the given quadric (see \cite{Vishik4}, where motivic methods were used to verify Hoffmann's conjecture in all dimensions  $\leq 22$ in this setting). Going further, Vishik later formulated a very general conjecture concerning the complete motivic decomposition a smooth anisotropic quadric (the \emph{Excellent Connections} conjecture) which subsumed the nonsingular case of Conjecture \ref{ConjHoffmanns} in a conceptual way. Not long after this, Karpenko (\cite{Karpenko1}) used a similar approach to prove the characteristic $\neq 2$ case of the conjecture in its entirety, the key new ingredient being the application of (reduced power) Steenrod operations on modulo-2 Chow groups. More recently, Vishik (\cite[Thm. 1.3]{Vishik5}) proved the Excellent Connections conjecture in characteristic $\neq 2$, thus yielding another proof of Hoffmann's conjecture in this setting. This work also makes essential use of Steenrod squares in Chow theory.

In characteristic 2, the general picture is more complicated. In the nonsingular case, Karpenko's (as well as Vishik's) approach to Conjecture \ref{ConjHoffmanns} is still valid, and progress is only hindered here by the fact that the total mod-2 Steenrod operation is not yet available when $2$ is not invertible. To this end, weak forms of the first three Steenrod squares have been constructed by Haution, and these are sufficient to prove nontrivial partial results towards Conjecture \ref{ConjHoffmanns} (see \cite[Thm. 6.2]{Haution1},\cite[Thm. 5.8]{Haution2}). Haution's results are further supplemented by earlier work of Hoffmann and Laghribi, who extended Hoffmann's upper bound on $\mathfrak{i}_1(\phi)$ to the characteristic-2 setting, irrespective of whether $\phi$ is nonsingular or not (\cite[Lem. 4.1]{HoffmannLaghribi2}). For singular forms, the situation is different, and almost nothing is known in the direction of Conjecture \ref{ConjHoffmanns} beyond Hoffmann and Laghribi's bound. In fact, the one exception lies in the extreme case where $\phi$ is \emph{totally singular}. Here, it was shown in \cite[Thm. 9.4]{Scully2} that if $\mydim{\phi} = 2^n +m$ for nonnegative integers $n$ and $1 \leq m \leq 2^n$, and if $\witti{1}{\phi} \neq m$ (that is, if Hoffmann and Laghribi's bound is not met), then $\witti{1}{\phi} \leq \frac{m}{2}$. In the present article, we will settle this case completely by proving:

\begin{theorem} \label{ThmHoffmannsconjecture} Conjecture \ref{ConjHoffmanns} is true in the case where $\phi$ is totally singular. \end{theorem}

Contrary to the existing approaches to the nonsingular case of Conjecture \ref{ConjHoffmanns}, our proof of Theorem \ref{ThmHoffmannsconjecture} does not involve the study of Chow correspondences on the quadric $X_\phi$. Indeed, although we also make essential use of the computation of the canonical dimension\footnote{Here, the \emph{canonical dimension} of an algebraic variety $X$ over a field $k$ should be understood as the minimal dimension of the image of a rational self-map $X \dashrightarrow X$.} of $X_\phi$ (see \cite{KarpenkoMerkurjev,Totaro1}), it is exploited here in a rather more direct and algebraic way. This point of view begins with the following observation: Let $\phi$ be an anisotropic totally singular quadratic form of dimension $\geq 2$ over a field $k$ of characteristic 2 and let $\psi \subset \phi$ be a subform of codimension $\witti{1}{\phi}$.

\begin{proposition}[{see Proposition \ref{Propwarmup} below}] \label{PropMotivation} Suppose, in the above situation, that $h(\psi) < h(\phi)$\footnote{A weaker condition will suffice -- see the statement of Proposition \ref{Propwarmup}.}. Then there exist a quasi-Pfister form $\pi$\footnote{That  is, $\pi$ is the diagonal part of a bilinear Pfister form over $k$ -- see \S \ref{QPforms} below.}, a subform $\sigma \subset \pi$, an element $\lambda \in k^*$ and a form $\tau$ over $k$ such that $\psi \simeq \pi \otimes \tau$ and $\phi \simeq \psi \perp \lambda\sigma$. \end{proposition}

In the situation of Proposition \ref{PropMotivation}, Hoffmann's conjecture is immediately verified. Indeed, (since $\sigma \subset \pi$) the integer $\mydim{\pi}$ is a power of 2 strictly greater than $\witti{1}{\phi} - 1 = \mydim{\sigma} -1$, and (since $\psi$ is divisible by $\pi$) we have $\mydim{\phi} - 1 = \mydim{\psi} + \witti{1}{\phi} -1 \equiv \witti{1}{\phi} - 1 \pmod{\mydim{\pi}}$. It is not always possible, however, to decompose the form $\phi$ in the manner intimated by the proposition. Indeed, Vishik (see \cite[Lem. 7.1]{Totaro2}) has given examples of 16-dimensional anisotropic quadratic forms in characteristic different from 2 which have first higher isotropy index equal to 2, but which do not decompose in this way, and the same examples carry over into the totally singular setting (see Lemma \ref{LemVishiksexample} below). Thus, the picture is, in general, more complicated than that suggested by Proposition \ref{PropMotivation}. Perhaps surprisingly, however, the main result of this paper shows that the next best thing happens here:

\begin{theorem} \label{Maintheorem} Let $\phi$ be an anisotropic totally singular quadratic form of dimension $\geq 2$ over a field of characteristic 2 and let $s$ be the smallest nonnegative integer such that $2^s \geq \witti{1}{\phi}$. Then $\phi_1$ is divisible by an $s$-fold quasi-Pfister form. \end{theorem}

As remarked above, the key ingredient needed for the proof of this theorem is Totaro's computation of the canonical dimension of a totally singular quadric (see \cite[Thm. 5.1]{Totaro1}). In \cite{Scully1}, this computation was extended to the wider class of Fermat-type hypersurfaces of degree $p$ over fields of characteristic $p>0$, and this extension enables us to also prove a direct analogue of Theorem \ref{Maintheorem} for totally singular forms of any prime degree $p>2$ (known here as \emph{quasilinear $p$-forms}) - see Theorem \ref{ThmMainthm} below. Subsequently, we also get an analogue of Theorem \ref{ThmHoffmannsconjecture} in higher degrees. This may be interpreted as a complete solution to the problem of determining the possible values of the canonical dimension of a degree-$p$ Fermat-type hypersurface in characteristic $p>0$ (it is worth noting here that no such result is known for Fermat-type hypersurfaces of prime degree $p>2$ over fields of characteristic \emph{different from $p$}). Returning to the case where $p=2$, let us explain more precisely how Theorem \ref{Maintheorem} implies totally singular case of Hoffmann's conjecture:

\begin{proof}[Proof of Theorem \ref{ThmHoffmannsconjecture}] Since $\phi$ is totally singular, we have $\mydim{\phi_1} = \mydim{\phi} - \witti{1}{\phi}$ (see Remark \ref{RemsKnebusch} (2) below). Thus, if $s$ is as in Theorem \ref{Maintheorem}, then
\begin{equation*} \mydim{\phi} - 1 = \mydim{\phi_1} + \witti{1}{\phi} - 1 \equiv \witti{1}{\phi} - 1 \pmod{2^s}. \end{equation*}
Since $\witti{1}{\phi} - 1 < 2^s$, the result follows. \end{proof}

Of course, our main result goes somewhat deeper than this. In fact, Theorem \ref{Maintheorem} (resp. Theorem \ref{ThmMainthm} below) yields a complete answer to Question \ref{Quevalues} in the totally singular case (resp. its analogue for arbitrary quasilinear $p$-forms). In other words, all restrictions on the possible values of the Knebusch splitting pattern of $\phi$ are explained here by the presence of certain divisibilities among its higher anisotropic kernels -- more precise statements will be given in \S \ref{Possiblevalues} below (see Theorem \ref{Thmallvalues}, Proposition \ref{Propexistenceofvalues}). With this result in hand, it is then natural to ask for a precise description of the general discrepancy which exists between the Knebusch and \emph{full} splitting patterns in the totally singular case. An immediate challenge here concerns the determination of all non-trivial restrictions which the former invariant imposes on the latter. In \S \ref{Excellentconnections} below, we initiate this process by conjecturing that the gaps in the full splitting pattern established in characteristic $\neq 2$ by Vishik as a consequence of his Excellent Connections theorem (see \cite[Prop. 2.6]{Vishik5} or Theorem \ref{ThmVishikstheorem} below) are also present in the totally singular theory. A particular case of this conjecture was already proved in \cite[Thm. 9.2]{Scully2}, and we provide some further evidence for its general veracity here\footnote{To the author's knowledge, there is no conjectural description of the possible values of the full splitting pattern, even in the nonsingular case (in any characteristic).}.

The rest of this paper is organised as follows: In sections \ref{Quasilinearpformsandhypersurfaces} and \ref{Incompressibilityresults}, we recall the basic theory of quasilinear $p$-forms and introduce the key notions and results which will be needed in the main part of the text. As a warm-up for the proof of our main result, we prove in \S \ref{Warmup} (a stronger version of) Proposition \ref{PropMotivation} above and consider some situations in which it may be applied. The proof of Theorem \ref{Maintheorem} (and its generalisation to higher degrees) is then given in \S \ref{ProofofMainTheorem}, and, in \S \ref{Applications}, we apply this result to determine all possible Knebusch splitting patterns of quasilinear $p$-forms and settle another conjecture of Hoffmann concerning quasilinear $p$-forms with ``maximal splitting''. In \S \ref{Bilinearforms}, we discuss some consequences of these results for the Knebusch splitting theory of symmetric bilinear forms in characteristic 2, as initiated by Laghribi in \cite{Laghribi6,Laghribi7}. Finally, in \S \ref{Excellentconnections}, we consider the aforementioned problem of establishing totally singular analogues of the results obtained by Vishik in \cite{Vishik5}.\\

\noindent {\bf Notation and Terminology.} Unless stated otherwise, $p$ will denote an arbitrary prime integer and $F$ will denote an arbitrary field of characteristic $p$. If $L$ is a field of characteristic $p$ and $a_1,\hdots,a_n$ are elements of $L$, then $L_{a_1,\hdots,a_n}$ will denote the field $L(\sqrt[p]{a_1},\hdots,\sqrt[p]{a_n})$. Finally, if $k$ is a field and $T = (T_1,\hdots,T_m)$ is a tuple of algebraically independent variables over $k$, then we will write $k[T]$ for the polynomial ring $k[T_1,\hdots,T_n]$ and $k(T)$ for its fraction field.

\section{Quasilinear $p$-forms and quasilinear $p$-hypersurfaces} \label{Quasilinearpformsandhypersurfaces}

The basic material presented in this section was originally developed in a series of papers by Hoffmann and Laghribi (see \cite{HoffmannLaghribi1,Laghribi2,Laghribi3,Laghribi5} and especially \cite{Hoffmann2}). Additional elementary results which will be needed in the sequel are also included here. For any details which are omitted from our exposition of the basic theory, we refer the reader to \cite{Hoffmann2}.

\subsection{Basic notions} \label{Basicnotions} Let $\phi \colon V \rightarrow F$ be a nonzero form on a (nonzero) finite-dimensional $F$-vector space $V$. We say that $\phi$ is a \emph{quasilinear $p$-form} (on $V$) if $\phi$ is homogeneous of degree $p$ and the equation $\phi(v+w) = \phi(v) + \phi(w)$ holds for all $(v,w) \in V \times V$. By a \emph{quasilinear $p$-form over $F$} (or sometimes simply a \emph{form over $F$} or \emph{$F$-form}), we will mean a quasilinear $p$-form on some (nonzero) finite-dimensional $F$-vector space. By a \emph{quasilinear $p$-hypersurface over $F$} we will mean a projective hypersurface defined by the vanishing of a quasilinear $p$-form on some $F$-vector space of dimension $\geq 2$. In the special case where $p=2$, we will speak of \emph{quasilinear quadratic forms} rather than quasilinear 2-forms.

\begin{remark} \label{Remnowheresmooth} Let $T = (T_1,\hdots,T_n)$ be a tuple of $n \geq 2$ algebraically independent variables over $F$, and let $f \in F[T]\setminus \lbrace 0 \rbrace$. Then the projective hypersurface $X_f = \lbrace f = 0 \rbrace \subset \mathbb{P}^{n-1}$ is nowhere smooth if and only if $f \in F[T_1^p,\hdots,T_n^p]$. In particular, if $p=2$, then a nonzero quadratic form of dimension $\geq 2$ over $F$ is totally singular (in the sense of \S \ref{Introduction}) if and only if it is quasilinear. \end{remark}

Let $\phi$ be a quasilinear $p$-form over $F$. The underlying $F$-vector space of $\phi$ will be denoted by $V_\phi$. Its dimension will be called the \emph{dimension} of $\phi$ and will be denoted by $\mydim{\phi}$. If $\mydim{\phi} \geq 2$, then the quasilinear $p$-hypersurface $\lbrace \phi =0 \rbrace \subset \mathbb{P}(V_\phi)$ (which is nowhere smooth by Remark \ref{Remnowheresmooth}) will be denoted by $X_{\phi}$. The set $\lbrace \phi(v)\;|\;v \in V_{\phi} \rbrace$ of elements of $F$ represented by $\phi$ will be denoted by $D(\phi)$. Given a field extension $L$ of $F$, we will write $\phi_L$ for the unique quasilinear $p$-form on the $L$-vector space $V_{\phi} \otimes_F L$ such that $\phi_L(v \otimes 1) = \phi(v)$ for all $v \in V_\phi$. If $R$ is a subring of $L$ containing $F$, then $D(\phi_R)$ will denote the subset $\lbrace \phi(w) \;|\; w \in V_\phi \otimes_F R \rbrace$ of $D(\phi_L)$ (which lies in $R$). Given $a \in F^*$, we will write $a\phi$ for the form $v \mapsto a\phi(v)$ on the vector space $V_\phi$. 

Let $\psi$ be another quasilinear $p$-form over $F$. If there exists an injective (resp. bijective) $F$-linear map $f \colon V_\psi \rightarrow V_\phi$ such that $\phi(f(v)) = \psi(v)$ for all $v \in V_\psi$, then we will say that $\psi$ is a \emph{subform} of (resp. is \emph{isomorphic} to) $\phi$ and write $\psi \subset \phi$ (resp. $\psi \simeq \phi$). If $\psi \simeq a\phi$ for some $a \in F^*$, then we will say that $\psi$ and $\phi$ are \emph{similar}. The \emph{sum} $\psi \oplus \phi$ (resp. \emph{product} $\psi \otimes \phi$) is defined as the unique quasilinear $p$-form on $V_\psi \oplus V_\phi$ (resp. $V_\psi \otimes_F V_\phi$) such that $(\psi \oplus \phi)\big((v,w)\big) = \psi(v) + \phi(w)$ (resp. $(\psi \otimes \phi)(v\otimes w) = \psi(v)\phi(w)$) for all $(v,w) \in V_\psi \times V_\phi$. Given a positive integer $n$, $n \cdot \phi$ will denote the sum of $n$ copies of $\phi$ (note that $n \cdot \phi \neq n\phi$). If there exists a form $\tau$ over $F$ such that $\phi \simeq \psi \otimes \tau$, then we will say that $\phi$ is \emph{divisible} by $\psi$.

Given elements $a_1,\hdots,a_n \in F$, at least one of which is nonzero, we will write $\form{a_1,\hdots,a_n}$ for the quasilinear $p$-form $(\lambda_1,\hdots,\lambda_n) \mapsto \sum_{i=1}^n a_i \lambda_i^p$ on the $F$-vector space $F^{\oplus n}$. By definition, every quasilinear $p$-form of dimension $n$ over $F$ is isomorphic to $\form{a_1,\hdots,a_n}$ for some $a_i \in F$.

A vector $v \in V_{\phi}$ is said to be \emph{isotropic} if $\phi(v) = 0$. We will say that $\phi$ is \emph{isotropic} if $V_{\phi}$ contains a nonzero isotropic vector, and \emph{anisotropic} otherwise. By the additivity of $\phi$, the subset $V_\phi^0$ of all isotropic vectors in $V_\phi$ is, in fact, an $F$-linear subspace of $V_\phi$. Its dimension will be called the \emph{isotropy index} of $\phi$, and will be denoted by $\witti{0}{\phi}$ (note that, in the special case where $p=2$, this agrees with the definition given in \S \ref{Introduction}). The additivity of $\phi$ also implies that $D(\phi)$ is a (nonzero finite-dimensional) $F^p$-linear subspace of $F$ (where $F$ is equipped with its natural $F^p$-vector space structure). Conversely, if $U$ is a nonzero finite-dimensional $F^p$-linear subspace of $F$, and if $a_1,\hdots,a_n$ is a basis of $U$, then $\sigma = \form{a_1,\hdots,a_n}$ is a quasilinear $p$-form over $F$ satisfying $D(\sigma) = U$. In fact, it is easy to see that $\sigma$ is, up to isomorphism, the unique \emph{anisotropic} quasilinear $p$-form with this property:

\begin{lemma}[{cf. \cite[Prop. 2.12]{Hoffmann2}}] \label{Lemexistenceofforms} Let $U$ be a nonzero finite-dimensional $F^p$-linear subspace of $F$. Then, up to isomorphism, there exists a unique anisotropic quasilinear $p$-form $\phi$ over $F$ such that $D(\phi) = U$. \end{lemma}

In particular, Pfister's quadratic subform theorem (\cite[Thm. 17.12]{EKM}) takes the following simplified form in this setting:

\begin{proposition}[{cf. \cite[Prop. 2.6]{Hoffmann2}}] \label{Propanisclassification} Let $\psi$ and $\phi$ be anisotropic quasilinear $p$-forms over $F$. Then $\psi \subset \phi$ if and only if $D(\psi) \subseteq D(\phi)$. In particular, $\psi \simeq \phi$ if and only if $D(\psi) = D(\phi)$. \end{proposition}

In view of these observations, we can define (up to isomorphism) the \emph{anisotropic kernel} of $\phi$ as the unique quasilinear $p$-form $\anispart{\phi}$ over $F$ such that $D(\anispart{\phi}) = D(\phi)$. If we view $D(\phi)$ as an $F$-vector space via the Frobenius $F \mapsto F^p$, then $\phi \colon V_\phi \rightarrow D(\phi)$ is a surjective $F$-linear map with kernel $V_\phi^0$. We thus obtain:

\begin{proposition}[{cf. \cite[Lem 2.10]{Hoffmann2}}] \label{PropWittdecomposition} Let $\phi$ be a quasilinear $p$-form over $F$. Then $\phi$ is anisotropic if and only if $\phi \simeq \anispart{\phi}$. If $\phi$ is isotropic, then $\phi \simeq \anispart{\phi} \oplus \witti{0}{\phi} \cdot \form{0}$. In particular, $\mydim{\anispart{\phi}} = \mydim{\phi} - \witti{0}{\phi}$. \end{proposition}

In summary, we see that $\phi$ is determined up to isomorphism by the set $D(\phi)$ and the integer $\witti{0}{\phi}$. Since we only consider nonzero forms here, we always have $\mydim{\anispart{\phi}} \geq 1$. In the case where $\mydim{\anispart{\phi}} = 1$, we will say that $\phi$ is \emph{split}. Given another form $\psi$ over $F$, we will write $\psi \sim \phi$ whenever $\anispart{\psi} \simeq \anispart{\phi}$. If $F$ is perfect (that is, if $F = F^p$), then every form over $F$ is split and the theory is vacuous. As such, we are essentially only interested in the case where $F$ is imperfect.

\begin{remark} \label{Remanispart} Let $\phi$ be a quasilinear $p$-form over $F$. Choose a basis $v_1,\hdots,v_n$ of $V_\phi$ and let $a_i = \phi(v_i)$ for all $1 \leq i \leq n$, so that $\phi \simeq \form{a_1,\hdots,a_n}$. Suppose that $v \in V_\phi$ is an isotropic vector, and write $v = \sum_{i=1}^n\lambda_i v_i$. If $\lambda_j \neq 0$ for $1 \leq j \leq n$, then the $F^p$-vector space $D(\phi)$ is spanned by the elements $a_1,
\hdots,a_{j-1},a_{j+1},\hdots,a_n$. In particular, we have $\phi \sim \form{a_1,\hdots,a_{j-1},a_{j+1},\hdots,a_n}$. \end{remark}

\subsection{Function fields of quasilinear $p$-hypersurfaces and their products} \label{Functionfields} Let $\phi$ be a quasilinear $p$-form of dimension $\geq 2$ over $F$. If $\phi$ is not split, then the quasilinear $p$-hypersurface $X_\phi$ is an integral scheme (cf. \cite[Lem. 7.1]{Hoffmann2}), as is its affine cone $\lbrace \phi = 0 \rbrace \subset \mathbb{A}(V_\phi)$. In this case, we will write $F(\phi)$ for the function field of the former and $F[\phi]$ for that of the latter. If $L$ is a field extension of $F$, then we will simply write $L(\phi)$ instead of $L(\phi_L)$ whenever it is defined. In general, given a finite collection $\phi_1,\hdots,\phi_n$ of quasilinear $p$-forms dimension $\geq 2$ over $F$, we will write $F(\phi_1 \times \hdots \times \phi_n)$ for the function field of the scheme $X_{\phi_1} \times \hdots \times X_{\phi_n}$, provided that it is integral. This notation will be further simplified where possible; for example, if $\phi_1 = \hdots = \phi_n = \phi$, then we will simply write $F(\phi^{\times n})$ instead of $F(\phi_1 \times \hdots \times \phi_n)$.

\begin{remarks} \label{Remsff} Let $\phi$ be a quasilinear $p$-form over $F$. Assume that $\phi$ is not split. We make the following basic observations:
\begin{enumerate}[leftmargin=*] \item Let $a_0,\hdots,a_n \in F$ be such that $\phi \simeq \form{a_0,\hdots,a_n}$ and $a_0,a_1 \neq 0$. Then we have $F$-isomorphisms
\begin{equation*} F(\phi) \simeq F(S)\Big(\sqrt[p]{a_1^{-1}\big(a_0 + a_2S_2^p + \hdots + a_nS_n^p\big)}\Big) \end{equation*}
and
\begin{equation*} F[\phi] \simeq \mathrm{Frac}\big(F[T]/(a_0T_0^p + \hdots + a_nT_n^p)\big) \simeq F(U)\Big(\sqrt[p]{a_0^{-1}\big(a_1U_1^p + \hdots + a_nU_n^p\big)}\Big), \end{equation*}
where $S = (S_2,\hdots,S_n)$, $T = (T_0,\hdots,T_n)$ and $U = (U_1,\hdots,U_n)$ are tuples of algebraically independent variables over $F$.
\item $F[\phi]$ is $F$-isomorphic to a degree-1 purely transcendental extension of $F(\phi)$.
\item $F(\phi)$ is $F$-isomorphic to a degree-$\witti{0}{\phi}$ purely transcendental extension of $F(\anispart{\phi})$ -- see Proposition \ref{PropWittdecomposition}.
\item The form $\phi_{F(\phi)}$ is evidently isotropic. Furthermore, if $a_1,\hdots,a_n\in F$ are such that $\phi \simeq \form{a_1,\hdots,a_n}$, then consideration of the generic point in $X_\phi\big(F(\phi)\big)$ shows that $\phi_{F(\phi)} \sim \form{a_1,\hdots,a_{i-1},a_{i+1},\hdots,a_n}$ for every $1\leq i\leq n$ -- see Remark \ref{Remanispart}. \end{enumerate}
 \end{remarks}

\subsection{Quasi-Pfister $p$-forms} \label{QPforms} Let $\phi$ be a quasilinear $p$-form over $F$ and let $n$ be a positive integer. We say that $\phi$ is an \emph{n-fold quasi-Pfister $p$-form} if there exist $a_1,\hdots,a_n \in F$ such that $\phi \simeq \pfister{a_1,\hdots,a_n} \coloneqq \otimes_{i=1}^n\form{1,a_i,a_i^2,\hdots,a_i^{p-1}}$. For convenience, we also say that $\phi$ is a \emph{$0$-fold quasi-Pfister $p$-form} if $\phi \simeq \form{1}$. Note, in particular, that if $\phi$ is an $n$-fold quasi-Pfister $p$-form for some $n\geq 0$, then we have $\mydim{\phi} = p^n$. The basic observation concerning quasi-Pfister $p$-forms is found in the following proposition (which follows easily from Lemma \ref{Lemexistenceofforms}):

\begin{proposition}[{cf. \cite[Prop. 4.6]{Hoffmann2}}] \label{PropQP=field} Let $\phi$ be a quasilinear $p$-form over $F$. Then $\anispart{\phi}$ is a quasi-Pfister $p$-form if 
and only if $D(\phi)$ is a subfield of $F$. \end{proposition}

Since the set of elements represented by an arbitrary quasi-Pfister $p$-form is, by definition, a subfield of the base field, we obtain:

\begin{corollary}[{cf. \cite[Prop. 4.6]{Hoffmann2}}] \label{CoranispartQP} Let $\phi$ be a quasi-Pfister $p$-form over $F$. Then $\anispart{\phi}$ is a quasi-Pfister $p$-form. In particular, if $\phi$ is isotropic, then $\mydim{\anispart{\phi}} = \frac{1}{p^k}\mydim{\phi}$ for some $k \geq 1$. \end{corollary}

\begin{remark} \label{RemExplicitQP} More explicitly, let $\phi = \pfister{a_1,\hdots,a_n}$ for some $n \geq 1$ and $a_i \in F$. Then $D(\phi) = F^p(a_1,\hdots,a_n)$. Let $m$ be such that $[F^p(a_1,\hdots,a_n):F^p] = p^m$. If $m = 0$ (i.e., $D(\phi) = F^p$), then $\anispart{\phi} \simeq \form{1}$. If $m \geq 1$, then $\anispart{\phi} \simeq \pfister{b_1,\hdots,b_m}$ for any $m$ elements $b_1,\hdots,b_m \in F$ such that $F^p(b_1,\hdots,b_m) = F^p(a_1,\hdots,a_n)$. \end{remark}

Quasi-Pfister $p$-forms have a central role to play in the general theory of quasilinear $p$-forms. As shown by Hoffmann (\cite{Hoffmann2}), these forms are distinguished here by the very same properties which distinguish the classical Pfister forms among nonsingular quadratic forms. For this reason, it will be useful to define the \emph{divisibility index} of a given form $\phi$, denoted $\mathfrak{d}_0(\phi)$, as the largest nonnegative integer $s$ such that $\anispart{\phi}$ is divisible by an $s$-fold quasi-Pfister $p$-form. Clearly we have $\mathfrak{d}_0(\phi) \leq \mathrm{log}_p(\mydim{\anispart{\phi}})$, with equality holding if and only if $\anispart{\phi}$ is similar to a quasi-Pfister $p$-form. An alternative description of this invariant will be given in Corollary \ref{Cordivisibilityindexsimilarity} below.

\subsection{The norm form} \label{Normform} Let $\phi$ be a quasilinear $p$-form over $F$. The \emph{norm field} of $\phi$, denoted $N(\phi)$, is defined (cf. \cite[Def. 4.1]{Hoffmann2}) as the smallest subfield of $F$ which contains all ratios of nonzero elements of $D(\phi)$. Note, in particular, that we have $N(a\phi) = N(\phi) = N(\anispart{\phi})$ for all $a \in F^*$. In spite of its simple nature, this invariant has an important role to play in the whole theory. A more explicit description of the norm field may be given as follows:

\begin{remark} \label{Remexplicitnormfield} If $a_1,\hdots,a_n \in F$ are such that $\phi \simeq \form{a_1,\hdots,a_n}$ and $a_1 \neq 0$, then we have $N(\phi) = F^p(\frac{a_2}{a_1},\hdots,\frac{a_n}{a_1})$. \end{remark}

In particular, we see that $N(\phi)$ is a nonzero finite-dimensional $F^p$-linear subspace of $F$. By Lemma \ref{Lemexistenceofforms}, it follows that, up to isomorphism, there exists a unique anisotropic quasilinear $p$-form $\normform{\phi}$ over $F$ such that $D(\normform{\phi}) = N(\phi)$. The form $\normform{\phi}$ is called the \emph{norm form} of $\phi$ (cf. \cite[Def. 4.9]{Hoffmann2}). By Proposition \ref{PropQP=field}, $\normform{\phi}$ is a quasi-Pfister $p$-form. Its dimension (which is necessarily equal to a power of $p$) is called the \emph{norm degree} of $\phi$, and is denoted by $\mathrm{ndeg}(\phi)$ (cf. \cite[Def. 4.1]{Hoffmann2}). The following lemma characterises the norm form as the smallest anisotropic quasi-Pfister $p$-form which contains $\anispart{\phi}$ as a subform up to multiplication by a scalar (again, this is a simple consequence of Proposition \ref{Propanisclassification}):

\begin{lemma}[{cf. \cite[Lem. 2.10]{Scully2}}] \label{Lempropertyofnormform} Let $\phi$ be a quasilinear $p$-form (resp. a quasilinear $p$-form such that $1 \in D(\phi)$) and $\pi$ an anisotropic quasi-Pfister $p$-form over $F$. Then $\anispart{\phi}$ is similar to a subform of $\pi$ (resp. $\anispart{\phi} \subset \pi$) if and only if $\normform{\phi} \subset \pi$. In particular, $\anispart{\phi}$ is similar to a subform of $\normform{\phi}$ (resp. $\anispart{\phi} \subset \normform{\phi}$).\end{lemma}

\begin{example} \label{Exnormform} Let $\phi$ be a quasilinear $p$-form over $F$. The following are equivalent:
\begin{enumerate}\item $\anispart{\phi}$ is similar (resp. isomorphic) to a quasi-Pfister $p$-form.
\item $\normform{\phi} \simeq a\anispart{\phi}$ for some $a \in F^*$ (resp. $\normform{\phi} \simeq \anispart{\phi}$). 
\item $N(\phi) = aD(\phi)$ for some $a \in F^*$ (resp. $N(\phi) = D(\phi)$).\end{enumerate} \end{example}

\subsection{Similarity factors} \label{Similarityfactors} Let $\phi$ be a quasilinear $p$-form over $F$. By a \emph{similarity factor} of $\phi$, we mean an element $a \in F^*$ such that $a\phi \simeq \phi$. The set of all similarity factors of $\phi$ will be denoted by $G(\phi)^*$, and we will write $G(\phi)$ for the set $G(\phi)^* \cup \lbrace 0 \rbrace$. Note that $G(a\phi) = G(\phi) = G(\anispart{\phi})$ for all $a \in F^*$, the second equality being an obvious consequence of Proposition \ref{PropWittdecomposition}. Thus, in view of Proposition \ref{Propanisclassification}, we have:

\begin{lemma}[{cf. \cite[Lem. 6.3]{Hoffmann2}}] \label{Lemsimilarityfactor} Let $\phi$ be a quasilinear $p$-form over $F$ and let $a \in F^*$. Then $a \in G(\phi)^*$ if and only if $aD(\phi) \subseteq D(\phi)$. \end{lemma}

\begin{example} \label{ExroundPfister} Let $\phi$ be a quasi-Pfister $p$-form over $F$. Then, since $D(\phi)$ is a subfield of $F$, we have $G(\phi) = D(\phi)$. \end{example}

More generally, Lemma \ref{Lemsimilarityfactor} immediately implies the following:

\begin{corollary}[{cf. \cite[Prop. 6.4]{Hoffmann2}}] \label{Corsimilarityfield} Let $\phi$ be a quasilinear $p$-form over $F$. Then $G(\phi)$ is a subfield of $N(\phi)$ containing $F^p$. \end{corollary}

In particular, $G(\phi)$ is a nonzero finite-dimensional $F^p$-linear subspace of $F$. By Lemma \ref{Lemexistenceofforms}, it follows that, up to isomorphism, there exists a unique anisotropic quasilinear $p$-form $\simform{\phi}$ over $F$ such that $D(\simform{\phi}) = G(\phi)$. The form $\simform{\phi}$ is called the \emph{similarity form} of $\phi$ (cf. \cite[Def. 6.5]{Hoffmann2}). By Proposition \ref{PropQP=field}, $\simform{\phi}$ is a quasi-Pfister $p$-form. Taken together, Examples \ref{Exnormform} and \ref{ExroundPfister} yield:

\begin{example} \label{Exsimform} Let $\phi$ be a quasilinear $p$-form over $F$. The following are equivalent:
\begin{enumerate} \item $\anispart{\phi}$ is similar (resp. isomorphic) to a quasi-Pfister $p$-form.
\item $\simform{\phi} \simeq \normform{\phi} \simeq a\anispart{\phi}$ for some $a \in F^*$ (resp. $\simform{\phi} \simeq \normform{\phi} \simeq \anispart{\phi}$).
\item $G(\phi) = N(\phi) = aD(\phi)$ for some $a \in F^*$ (resp. $G(\phi) = N(\phi) = D(\phi)$). \end{enumerate} \end{example}

The basic observation concerning similarity factors is the following:

\begin{proposition} \label{Propdivisibilitybysimform} Let $\phi$ and $\psi$ be quasilinear $p$-forms over $F$. Then $G(\psi) \subseteq G(\phi)$ if and only if $\anispart{\phi}$ is divisible by $\simform{\psi}$. 
\begin{proof} We may assume that $1 \in D(\phi)$. Suppose that $G(\psi) \subseteq G(\phi)$. By Corollary \ref{Corsimilarityfield}, $G(\psi)$ and $G(\phi)$ are subfields of $F$. By Lemma \ref{Lemsimilarityfactor}, $D(\phi)$ is naturally a (finite-dimensional) vector space over $G(\phi)$, and hence over $G(\psi)$. If $a_1,\hdots,a_m$ is a basis of $D(\phi)$ over $G(\psi)$, then (since $D(\simform{\psi}) = G(\psi)$), Lemma \ref{Lemexistenceofforms} implies that $\anispart{\phi} \simeq \simform{\psi} \otimes \form{a_1,\hdots,a_m}$. Conversely, if $\anispart{\phi}$ is divisible by $\simform{\psi}$, then it is clear that $G(\psi) \subseteq G(\phi)$, since $G(\psi) = D(\simform{\psi}) = G(\simform{\psi})$ by Example \ref{Exsimform}. \end{proof} \end{proposition}

We thus obtain the following characterisation of the similarity form:

\begin{corollary}[{cf. \cite[Prop. 6.4]{Hoffmann2}}] \label{Corpropertyofsimform} Let $\phi$ be a quasilinear $p$-form and $\pi$ an anisotropic quasi-Pfister $p$-form over $F$. Then $\anispart{\phi}$ is divisible by $\pi$ if and only if $\simform{\phi}$ is divisible by $\pi$. In particular, $\anispart{\phi}$ is divisible by $\simform{\phi}$. \end{corollary}

This enables us to reinterpret the divisibility index $\mathfrak{d}_0(\phi)$ (see \S \ref{QPforms}) as follows:

\begin{corollary} \label{Cordivisibilityindexsimilarity} Let $\phi$ be a quasilinear $p$-form over $F$. Then $\mathfrak{d}_0(\phi) = \mathrm{log}_p(\mydim{\simform{\phi}}) = [G(\phi):F^p]$. \end{corollary}

We also get the following:

\begin{corollary} \label{CordivbyQP} Let $\phi$ and $\psi$ be quasilinear $p$-forms over $F$. Then $\anispart{\phi}$ is divisible by $\normform{\psi}$ if and only if $N(\psi) \subseteq G(\phi)$. If, additionally, $1 \in D(\psi)$, then the latter condition may be replaced by $D(\psi) \subseteq G(\phi)$.
\begin{proof} For the second statement, we simply recall that $N(\phi)$ is the smallest subfield of $F$ containing all ratios of nonzero elements of $D(\phi)$ and that $G(\phi)$ is a subfield of $F$ (Corollary \ref{Corsimilarityfield}). For the first, we can replace $\psi$ by its norm form to arrive at the case where $N(\psi) = G(\psi)$ and $\normform{\psi} \simeq \simform{\psi}$ (see Example \ref{Exsimform}). The result is therefore a particular case of Proposition \ref{Propdivisibilitybysimform}. \end{proof} \end{corollary}

\subsection{A criterion for a quasilinear $p$-form to be quasi-Pfister} \label{Criterion} Let $\phi$ be a quasilinear $p$-form over $F$ such that $1 \in D(\phi)$, let $L$ be a field extension of $F$ and let $\alpha \in D(\phi_L)\setminus \lbrace 0 \rbrace$. Consider the set $S_{\alpha} = \lbrace a \in F\;|\;\alpha a \in D(\phi_L) \rbrace$. Since $D(\phi_L)$ is an $L^p$-linear subspace of $L$, we have the following observation:
\begin{equation} \label{eq2.1} \sum \lambda_i^pa_i\in S_{\alpha} \text{ for all } \lambda_i \in F \text{ and all } a_i \in S_{\alpha}. \end{equation}

\begin{lemma} \label{Lemmultipliers} In the above situation, let $P \in F^p[T]$ be a polynomial of degree $<p$ in a single variable $T$ such that $P(b) \in S_{\alpha}$ for all $b \in D(\phi)$. Then $b^n \in S_{\alpha}$ for all $b \in D(\phi)$ and all $n \leq \mathrm{deg}(P)$.
\begin{proof} We proceed by induction on $d = \mathrm{deg}(P)$. Since $\alpha \in D(\phi_L)$, the case where $d=0$ is trivial. Suppose now that $d>0$, and let $\lambda \in F$ be such that $P(T + \lambda^p) = P(T) + Q(T)$ for some $Q \in F^p[T]$ of degree $d-1$. Since $F^p \subseteq D(\phi)$ by hypothesis, our assumption and \eqref{eq2.1} imply that $Q(b) = P(b + \lambda^p) - P(b) \in S_{\alpha}$ for all $b \in D(\phi)$. By the induction hypothesis, it follows that $b^n \in S_{\alpha}$ for all $b \in D(\phi)$ and all $n<d$. Finally, since $P(b) = \sum_{i=0}^d\lambda_i^pb^i$ for some $\lambda_i \in F$ with $\lambda_d \neq 0$, \eqref{eq2.1} implies that, for any $b \in D(\phi)$, we also have $b^d \in S_{\alpha}$. This proves the lemma. \end{proof} \end{lemma}

Suppose now that there exists a polynomial $P \in F^p[T]$ as in the statement of Lemma \ref{Lemmultipliers} with $\mathrm{deg}(P) \geq 2$ (in particular, we necessarily have $p>2$). A first application of the lemma shows that we have $D(\phi) \subseteq S_{\alpha}$. Since $D(\phi_L)$ is spanned by $D(\phi)$ as an $L^p$-vector space, this implies that $\alpha D(\phi_L) \subseteq D(\phi_L)$, and hence (for dimension reasons) that $\alpha D(\phi_L) = D(\phi_L)$. Another application of Lemma \ref{Lemmultipliers} then shows that $b^n \in D(\phi_L)$ for all $b \in D(\phi)$ and all $n \leq \mathrm{deg}(P)$. In particular, since $\mathrm{deg}(P) \geq 2$, we have $2bc = (b + c)^2 - b^2 - c^2 \in D(\phi_L)$ for all $b,c \in D(\phi)$. Since $p>2$, and since $D(\phi_L)$ is spanned by $D(\phi)$ as an $L^p$-vector space, this implies that $D(\phi_L)$ is closed under multiplication, i.e., that $D(\phi_L) = N(\phi_L)$ (see Remark \ref{Remexplicitnormfield}). By Example \ref{Exnormform}, this means that $\anispart{(\phi_L)} \simeq \normform{(\phi_L)}$. We have thus proved:

\begin{lemma} \label{LemQPcriterion} Assume that $p>2$. Let $\phi$ be a quasilinear $p$-form over $F$ such that $1 \in D(\phi)$ and let $L$ be a field extension of $F$. Suppose that there exists a polynomial $P \in F^p[T]$ in a single variable $T$, and an element $\alpha \in D(\phi_L)\setminus \lbrace 0 \rbrace$ such that $2 \leq \mathrm{deg}(P) < p$ and $\alpha  P(b) \in D(\phi_L)$ for all $b \in D(\phi)$. Then $\anispart{(\phi_L)}$ is a quasi-Pfister $p$-form. \end{lemma}

\subsection{The Cassels-Pfister representation theorem} \label{CasselsPfister} Let $\phi$ be a quadratic form over a field $k$ and let $f \in k[T]$ be a polynomial in a single variable $T$ which is represented by the form $\phi_{k(T)}$. One of the foundational results of the classical algebraic theory of quadratic forms is the Cassels-Pfister representation theorem, which asserts that, in this case, $\phi$ already represents $f$ over the polynomial ring $k[T]$ (see \cite[Thm. 17.3]{EKM}). In the present setting, the original argument of Cassels may be readily adapted to prove the analogous statement for quasilinear $p$-forms. However, as pointed out by Hoffmann (\cite{Hoffmann2}), the additivity property of these forms enables one to prove a stronger multi-variable statement taking the following form: 

\begin{theorem}[{cf. \cite[Cor. 3.4]{Hoffmann2}}] \label{ThmCasselsPfister} Let $\phi$ be a quasilinear $p$-form over $F$, let $T = (T_1,\hdots,T_m)$ be a tuple of algebraically independent variables over $F$ and let $f \in F[T]$. Then $f \in D(\phi_{F(T)})$ if and only if $f \in D(\phi_{F[T]})$, if and only if $f \in D(\phi)[T_1^p,\hdots,T_m^p]$. \end{theorem}

Now, in the situation of Theorem \ref{ThmCasselsPfister}, the $F(T)^p$-vector space $D(\phi_{F(T)})$ is (evidently) spanned by elements of $D(\phi)$. Thus, in view of Lemma \ref{Lemsimilarityfactor}, we immediately obtain the following result concerning rational similarity factors:

\begin{corollary}[{cf.\cite[Prop. 6.7]{Hoffmann2}}] \label{CorCPsimilarity} Let $\phi$ be a quasilinear $p$-form over $F$, let $T = (T_1,\hdots,T_m)$ be a tuple of algebraically independent variables over $F$ and let $f \in F[T]$. Then $f \in G(\phi_{F(T)})$ if and only if $f \in G(\phi)[T_1^p,\hdots,T_m^p]$. \end{corollary}

\subsection{Isotropy of quasilinear $p$-forms under scalar extension} \label{Isotropy} We now collect some basic facts regarding the isotropy of quasilinear $p$-forms under scalar extension. 

Let $K$ and $L$ be extensions of a field $k$. Recall that a \emph{$k$-place} $K \dashrightarrow L$ is a pair $(R,f)$ consisting of a valuation subring $k \subseteq R \subseteq K$ and a local $k$-algebra homomorphism $f \colon R \rightarrow L$. For example, given an inclusion $i \colon K \hookrightarrow L$, the pair $(K,i)$ defines a $k$-place $K \dashrightarrow L$. If there exist $k$-places $K \dashrightarrow L$ and $L \dashrightarrow K$, then we say that $K$ and $L$ are \emph{equivalent over $k$}, and write $K \sim_k L$. For instance, this easily seen to be the case whenever $L$ (resp. $K$) is a purely transcendental extension of $K$ (resp. $L$) (see \cite[\S 103]{EKM} for further details). We have here the following basic lemma, which is a consequence of the completeness of $X_{\phi}$ (see \cite[(7.3.8)]{EGAII}):

\begin{lemma}[{cf. \cite[Lem. 3.4]{Scully2}}] \label{LemPlaceisotropy} Let $\phi$ be a quasilinear $p$-form over $F$ and let $K$ and $L$ be field extensions of $F$ such that there exists an $F$-place $K \dashrightarrow L$. Then $\witti{0}{\phi_L} \geq \witti{0}{\phi_K}$. In particular, if $K \sim_F L$, then $\witti{0}{\phi_K} = \witti{0}{\phi_L}$. \end{lemma}

Note, in particular, that passage to rational extensions of the base field does not affect the isotropy index of a quasilinear $p$-form. By MacLane's theorem (\cite[Prop. VIII.4]{Lang}), the same is, in fact, true of arbitrary \emph{separable} extensions\footnote{Recall that an extension of fields $k \subseteq L$ is called \emph{separable} if, for any algebraic closure $\overline{k}$ of $k$, the ring $L \otimes_k \overline{k}$ has no nontrivial nilpotent elements.}:

\begin{lemma}[{cf. \cite[Prop. 5.3]{Hoffmann2}}] \label{Lemseparableextensions} Let $\phi$ be an anisotropic quasilinear $p$-form over $F$ and let $L$ be a field extension of $F$. If $L$ is separable over $F$, then $\phi_L$ is anisotropic and $\mathrm{ndeg}(\phi_L) = \mathrm{ndeg}(\phi)$. \end{lemma}

Thus, in order to study the isotropy behaviour of quasilinear $p$-forms under scalar extension, we are effectively reduced to considering the case of purely inseparable algebraic extensions. In degree $p$, we have the following basic observations, all of which can be easily verified using the results which have been discussed thus far (recall here that, given $a_1,\hdots,a_n \in F$, $F_{a_1,\hdots,a_n}$ denotes the field $F(\sqrt[p]{a_1},\hdots,\sqrt[p]{a_n})$):

\begin{lemma}[{see \cite[\S 5]{Hoffmann2},\cite[Lem. 3.8]{Scully2}}] \label{LemIsotropypurelyinseparable} Let $\phi$ be a quasilinear $p$-form over $F$ and let $a \in F \setminus F^p$. Then:
\begin{enumerate} \item $D(\phi_{F_a}) = D(\pfister{a} \otimes \phi) = \sum_{i=0}^{p-1}a^iD(\phi)$.
\item $\witti{0}{\phi_{F_a}} = \frac{1}{p}\witti{0}{\pfister{a} \otimes \phi}$.
\item $\mathrm{ndeg}(\phi_{F_a}) = \begin{cases} \frac{1}{p}\mathrm{ndeg}(\phi) & \text{if } a \in N(\phi), \\ \mathrm{ndeg}(\phi) & \text{if } a \notin N(\phi). \end{cases}$.
\item If $\phi$ is anisotropic and $a \notin N(\phi)$, then $\phi_{F_a}$ is anisotropic.
\item $\mydim{\anispart{(\phi_{F_a})}} \geq \frac{1}{p}\mydim{\anispart{\phi}}$.
\item Equality holds in $(5)$ if and only if $\anispart{\phi}$ is divisible by $\pfister{a}$, if and only if $a \in G(\phi)$. \end{enumerate} \end{lemma}

\begin{remark} The second equivalence in (6) holds by Corollary \ref{CordivbyQP}. \end{remark}

As an application of the first part of the lemma, we have:

\begin{corollary} \label{Corsimilaritydegreep} Let $\phi$ be a quasilinear $p$-form over $F$ and let $a \in F \setminus F^p$. Then $G(\phi_{F_a}) = G(\pfister{a} \otimes \phi)$. In particular, $\mathfrak{d}_0(\pfister{a}) = \mathfrak{d}_0(\phi_{F_a}) + 1$.
\begin{proof} More specifically, the first statement is an immediate consequence of Lemma \ref{LemIsotropypurelyinseparable} (1) and Lemma \ref{Lemsimilarityfactor}. The second then follows from Corollary \ref{Cordivisibilityindexsimilarity}. \end{proof} \end{corollary}

Suppose now that $\pi = \pfister{a_1,\hdots,a_n}$ is an \emph{anisotropic} quasi-Pfister $p$-form over $F$. By Remark \ref{RemExplicitQP}, we have $[F^p(a_1,\hdots,a_n):F^p] = p^n$, which means that $a_i \notin F_{a_1,\hdots,a_{i-1}}$ for every $1 \leq i \leq n$. Repeated applications of Lemma \ref{LemIsotropypurelyinseparable} (2) and \ref{Corsimilaritydegreep} therefore yield the following proposition:

\begin{proposition} \label{PropindexofmultiplesofQP} Let $\phi$ be a quasilinear $p$-form over $F$, let $\pi$ be as above and let $\psi = \pi \otimes \phi$. Then $\witti{0}{\psi} = p^n\witti{0}{\phi_{F_{a_1,\hdots,a_n}}}$ and $\mathfrak{d}_0(\psi) = \mathfrak{d}_0(\phi_{F_{a_1,\hdots,a_n}}) + n$. \end{proposition}

Now, in view of Remark \ref{Remsff} (1), one may combine the above results in order to study the isotropy behaviour of quasilinear $p$-forms under scalar extension to function fields of quasilinear $p$-hypersurfaces. More specifically, let $\psi$ be a quasilinear $p$-form over $F$ which is not split, and let $a_0,\hdots,a_n \in F$ be such that $\psi \simeq \form{a_0,\hdots,a_n}$, with $a_0,a_1 \neq 0$. Then, by Remark \ref{Remsff} (1), we have an $F$-isomorphism of fields 
\begin{equation*} F(\psi) \simeq F(T)\Big(\sqrt[p]{a_1^{-1}\big(a_0 + a_2T_2^p + \hdots + a_nT_n^p\big)}\Big), \end{equation*}
where $T = (T_2,\hdots,T_n)$ is an $(n-1)$-tuple of algebraically independent variables over $F$. Thus, putting Lemmas \ref{Lemseparableextensions} and \ref{LemIsotropypurelyinseparable} and together, we obtain:

\begin{lemma}[{see \cite[\S\S 7.3,7.4]{Hoffmann2}}] \label{Lemisotropyff} Let $\phi$ be an anisotropic quasilinear $p$-form over $F$, and let $\psi$ be as above. Then:
\begin{enumerate} \item $\mydim{\anispart{(\phi_{F(\psi)})}} \geq \frac{1}{p}\mydim{\phi}$.
\item Equality holds in $(1)$ if and only if $a_1^{-1}(a_0 + a_2T_2^p + \hdots + a_nT_n^p) \in G(\phi_{F(T)})$.
\item $\mathrm{ndeg}(\phi_{F(\psi)}) \geq \frac{1}{p}\mathrm{ndeg}(\phi)$.
\item Equality holds in $(3)$ if and only if $a_1^{-1}(a_0 + a_2T_2^p + \hdots + a_nT_n^p) \in N(\phi_{F(T)})$.
\item The equivalent conditions of $(4)$ are satisfied if $\phi_{F(\psi)}$ is isotropic. \end{enumerate} \end{lemma}

As a basic application, we have:

\begin{corollary}[{see \cite[\S\S 7.3,7.4]{Hoffmann2}}] \label{Corisotropyff} Let $\phi$ and $\psi$ be quasilinear $p$-forms over $F$ such that $\phi$ is anisotropic and $\psi$ is not split. Then: 
\begin{enumerate} \item $\mydim{\anispart{(\phi_{F(\psi)})}} \geq \frac{1}{p}\mydim{\phi}$, with equality holding if and only if $N(\psi) \subseteq G(\phi)$.
\item If $\phi_{F(\psi)}$ is isotropic, then $N(\psi) \subseteq N(\phi)$. In particular, $\mathrm{ndeg}(\psi) \leq \mathrm{ndeg}(\phi)$. \end{enumerate}
\begin{proof} We may assume that $\psi$ is as in Lemma \ref{Lemisotropyff}. In this case, we have $N(\psi) = F^p(\frac{a_0}{a_1},\hdots,\frac{a_n}{a_1})$ (see Remark \ref{RemExplicitQP}), and so (1) follows from the first two parts of the former lemma and Corollary \ref{CorCPsimilarity}. Similarly, since $N(\phi_L) = D((\normform{\phi})_{L})$ for any field extension $L$ of $F$, (2) follows from Lemma \ref{Lemisotropyff} (4,5) and Theorem \ref{ThmCasselsPfister}. \end{proof} \end{corollary}

Finally, it will be useful to record in this section another basic application of the Cassels-Pfister theorem. To state it, let $T = (T_1,\hdots,T_m)$ be a tuple of algebraically independent variables over $F$, let $g \in F[T]$ be an irreducible polynomial and let $F[g]$ denoted the field $\mathrm{Frac}\big(F[T]/(g)\big)$ (i.e. the function field of the integral hypersurface $\lbrace g = 0 \rbrace \subset \mathbb{A}_F^m$). Given $f \in F[T]$, we write $\mathrm{mult}_g(f)$ for the multiplicity of $g$ in $f$, i.e., the largest nonnegative integer $s$ such that $f = g^sh$ for some $h \in F[T]$\footnote{With the added convention that $\mathrm{mult}_g(0) = +\infty$.}.

\begin{proposition} \label{PropCPapp} In the above situation, let $\phi$ be a quasilinear $p$-form over $F$ and let $f \in F[T]$. Suppose that $f \in D(\phi_{F(T)})$ and that $\phi_{F[g]}$ is anisotropic. Then $\mathrm{mult}_g(f) \equiv 0 \pmod{p}$. 
\begin{proof} Let $s = \mathrm{mult}_g(f)$. After replacing $f$ by $f/g^{kp} \in D(\phi_{F(T)})$ for a suitable integer $k \geq 0$, we may assume that $s<p$. Our goal is then to prove that $s = 0$. To see this, note first that there exists $v \in V_{\phi} \otimes_F F[T]$ such that $\phi_{F(T)}(v) = f$ by Theorem \ref{ThmCasselsPfister}. If $s \neq 0$, then the image $\overline{v}$ of $v$ in $V_{\phi} \otimes_F F[g]$ is an isotropic vector for $\phi_{F[g]}$. By hypothesis, it follows that $\overline{v} = 0$, which means that $v = gw$ for some $w \in V_\phi \otimes_F F[T]$. But this implies that $f = g^p\phi_{F(T)}(w)$, which contradicts the fact that $s < p$. We conclude that $s = 0$, and so the proposition is proved. \end{proof} \end{proposition}

\subsection{The divisibility index and scalar extension} \label{Divisibilityscalarextension} Let $\phi$ be a quasilinear $p$-form over $F$. We make some brief remarks concerning the behaviour of the divisibility index $\mathfrak{d}_0(\phi)$ (see \S \ref{QPforms}) under scalar extension.

\begin{lemma} \label{Lemdivisibilityindexextension} Let $\phi$ be a quasilinear $p$-form over $F$ and let $L$ be a field extension of $F$. If $(\simform{\phi})_L$ is anisotropic, then $\mathfrak{d}_0(\phi_L) \geq \mathfrak{d}_0(\phi)$.
\begin{proof} As an $L^p$-vector space, $D\big((\simform{\phi})_L\big)$ is spanned by $D(\simform{\phi}) = G(\phi)$. Since we evidently have $G(\phi) \subseteq G(\phi_L) = D\big(\simform{(\phi_L)}\big)$, and since $(\simform{\phi})_L$ is anisotropic by hypothesis, Proposition \ref{Propanisclassification} implies that $(\simform{\phi})_L \subset \simform{(\phi_L)}$. The desired assertion now follows from Corollary \ref{Cordivisibilityindexsimilarity}.
\end{proof} \end{lemma}

In particular, this applies in the case where $L$ is a separable extension of $F$ (see Lemma \ref{Lemseparableextensions}). In the case where $L$ is purely transcendental over $F$, we can say more:

\begin{lemma} \label{Lemdivisibilityindexstability} Let $\phi$ be a quasilinear $p$-form over $F$ and let $L$ be a purely transcendental extension of $F$. Then $\simform{(\phi_L)} \simeq (\simform{\phi})_L$ and $\mathfrak{d}_0(\phi_L) = \mathfrak{d}_0(\phi)$. 
\begin{proof} Continuing with the proof of Lemma \ref{Lemdivisibilityindexextension}, it is sufficient to show that in this case $G(\phi_L)$ is generated by $G(\phi)$ over $L^p$. If $L$ is finitely generated over $F$, then this follows from Corollary \ref{CorCPsimilarity}. On the other hand, the general case reduces easily to the finitely generated case in view of Lemma \ref{Lemsimilarityfactor}, so the lemma is proved. \end{proof} \end{lemma}

\subsection{The Knebusch splitting pattern} \label{Splittingpattern} Let $\phi$ be a quasilinear $p$-form over $F$. Following the construction outlined in \S \ref{Introduction} (see also \cite[\S 7.5]{Hoffmann2}), set $F_0  = F$, $\phi_0 = \anispart{\phi}$, and recursively define
\begin{itemize} \item $F_r = F_{r-1}(\phi_{r-1})$ (provided $\phi_{r-1}$ is not split), and
\item $\phi_r = \anispart{(\phi_{F_r})}$ (provided $F_r$ is defined). \end{itemize}
Note here that if $\phi_r$ is defined, then we have $\mydim{\phi_r} < \mydim{\phi_{r-1}}$ by Remark \ref{Remsff} (4). As such, the whole process is finite, terminating at the first nonnegative integer $h(\phi)$ for which $\mydim{\phi_{h(\phi)}} \leq 1$. The integer $h(\phi)$ will be called the \emph{height} of $\phi$, and the tower of fields $F_0 \subset F_1 \subset \hdots \subset F_{h(\phi)}$ will be called the \emph{Knebusch splitting tower} of $\phi$. For each $0 \leq r \leq h(\phi)$, we set $\wittj{r}{\phi} = \witti{0}{\phi_{F_r}}$. If $\phi$ is not split and $r\geq 1$, then the difference $\wittj{r}{\phi} - \wittj{r-1}{\phi}$ will be called the \emph{r-th higher isotropy index} of $\phi$, and will be denoted by $\witti{r}{\phi}$. In this case, the form $\phi_r$ will be called the \emph{r-th higher anisotropic kernel} of $\phi$. Finally, the sequence $\mathfrak{i}(\phi) = \big(\witti{1}{\phi},\hdots,\witti{h(\phi)}{\phi}\big)$ (understood to be empty if $\phi$ is split) will be called the \emph{Knebusch splitting pattern} of $\phi$\footnote{\label{Footnote3}See Footnote \ref{Footnote2}. We omit the term $\witti{0}{\phi}$ from the sequence because we are ultimately interested in the case where $\phi$ is anisotropic (i.e., where $\witti{0}{\phi} = 0$).}.

\begin{remarks} \label{RemsKnebusch} Let $\phi$ be a quasilinear $p$-form over $F$.
\begin{enumerate}[leftmargin=*] \item By the recursive nature of the above construction, we have $\witti{r}{\phi} = \witti{1}{\phi_{r-1}}$ for every $1 \leq r \leq h(\phi)$.
\item By Proposition \ref{PropWittdecomposition}, we have $\witti{r}{\phi} = \mydim{\phi_{r-1}} - \mydim{\phi_r}$ for all $1 \leq r \leq h(\phi)$.
\item Let $L$ be a field extension of $F$. As already remarked in \S \ref{Introduction}, it is not true in general that $\witti{0}{\phi_L} = \wittj{r}{\phi}$ for some $0 \leq r \leq h(\phi)$ - see Example \ref{Exnongeneric} below. \end{enumerate}
\end{remarks}

Note that by Remark \ref{Remsff} (3), we have the following:

\begin{lemma} \label{Lemsstequivalence} Let $\phi$ be a quasilinear $p$-form form of dimension $\geq 2$ and let $(F_r)$ denote its Knebusch splitting tower. Then $F_r \sim_F F(\phi^{\times r})$ for every $0 \leq r \leq h(\phi)$.\end{lemma}

In light of Lemma \ref{Lemsstequivalence}, we therefore have:

\begin{corollary} \label{Corcomputationssp} Let $\phi$ be a quasilinear $p$-form over $F$. Then, for every $0 \leq r \leq h(\phi)$, we have $\wittj{r}{\phi} = \witti{0}{\phi_{F(\phi^{\times r})}}$. \end{corollary}

Given the results of \S \ref{Isotropy}, we are now in a position to prove the following characterisation of anisotropic quasi-Pfister $p$-forms:

\begin{proposition}[{cf. \cite[Thm. 8.11]{HoffmannLaghribi1}}] \label{Propi1bound} Let $\phi$ be an anisotropic quasilinear $p$-form of dimension $\geq 2$ over $F$. Then $\mydim{\phi_1} \geq \frac{1}{p}\mydim{\phi}$, and the following conditions are equivalent:
\begin{enumerate} \item $\mydim{\phi_1} = \frac{1}{p}\mydim{\phi}$.
\item $\mathfrak{i}(\phi) = (p^{h(\phi)}-p^{h(\phi) -1},p^{h(\phi)-1}-p^{h(\phi) -2},\hdots,p^2 - p,p-1)$.
\item $\phi$ is similar to a quasi-Pfister $p$-form. \end{enumerate}
\begin{proof} The inequality $\mydim{\phi_1} \geq \frac{1}{p}\mydim{\phi}$ holds by Corollary \ref{Corisotropyff}. The same result shows that equality holds here if and only if $N(\phi) \subseteq G(\phi)$. By Corollary \ref{CordivbyQP}, the latter condition holds if and only if $\phi$ is divisible by $\normform{\phi}$. In view of Lemma \ref{Lempropertyofnormform}, this proves the equivalence of (1) and (3), as well as the implication $(2) \Rightarrow (3)$. On the other hand, if $\phi$ is similar to a quasi-Pfister $p$-form of dimension $p^n$, then $\phi_1$ is similar to a quasi-Pfister $p$-form of dimension $p^{n-1}$ by Corollary \ref{CoranispartQP} and (1). Since $\witti{1}{\phi} = \mydim{\phi} - \mydim{\phi_1}$ (see Remark \ref{RemsKnebusch} (2)), an easy induction on $h(\phi)$ then shows that (3) implies (2). \end{proof} \end{proposition}

Finally, a repeated application of Lemma \ref{Lemisotropyff} (3-5) (with $\psi = \phi$) yields the following computation of the height $h(\phi)$: 

\begin{corollary}[{cf. \cite[Thm. 7.25 (ii)]{Hoffmann2}}] \label{Corheight=ndegree} Let $\phi$ be a quasilinear $p$-form over $F$. Then $h(\phi) = \mathrm{log}_p\big(\mathrm{ndeg}(\phi)\big)$. \end{corollary}

Together with Corollary \ref{Corisotropyff} (2), this implies the following useful result:

\begin{corollary}[{cf. \cite[Prop. 4.12]{Scully2}}] \label{Corhfunctoriality} Let $\phi$ and $\psi$ be quasilinear $p$-forms over $F$ such that $\phi$ is anisotropic and $\psi$ is not split. If $\phi_{F(\psi)}$ is isotropic, then $h(\psi) \leq h(\phi)$. \end{corollary}

\subsection{The quasi-Pfister height and higher divisibility indices} \label{Higherdivindices} Let $\phi$ be a quasilinear $p$-form over $F$. As in \cite[\S 4.2]{Scully2}, we define the \emph{quasi-Pfister height} of $\phi$, denoted $\qp{h}(\phi)$, to be the smallest nonnegative integer $l$ such that $\phi_l$ is similar to a quasi-Pfister $p$-form (this is well defined, since $\phi_{h(\phi)}$, being of dimension 1, is similar to a 0-fold quasi-Pfister $p$-form). We have:

\begin{lemma} \label{Lemtruncationsplitting} Let $\phi$ be a quasilinear $p$-form over $F$ and let $d = h(\phi) - \qp{h}(\phi)$. Then $\mathfrak{i}(\phi) = (\witti{1}{\phi},\hdots,\witti{\qp{h}(\phi)}{\phi},p^d - p^{d-1},p^{d-1} - p^{d-2},\hdots,p^2 - p,p-1)$ and $\witti{\qp{h}(\phi)}{\phi} = \mathrm{dim}(\phi) - \wittj{\qp{h}(\phi) - 1}{\phi} - p^d < p^{d+1} - p^d$.
\begin{proof} The first statement is an immediate consequence of Proposition \ref{Propi1bound}. The point here is that $\phi_{\qp{h}(\phi)}$ is similar to a quasi-Pfister $p$-form of dimension $p^d$. By Remark \ref{RemsKnebusch} (2), we therefore have $\witti{\qp{h}(\phi)}{\phi} = \mydim{\phi_{\qp{h}(\phi)-1}} - \mydim{\phi_{\qp{h}(\phi)}} =  \mathrm{dim}(\phi) - \wittj{\qp{h}(\phi) - 1}{\phi} - p^d$. Finally, since $\phi_{\qp{h}(\phi)-1}$ is (by the definition of $\qp{h}(\phi)$) not similar to a quasi-Pfister $p$-form, it must have dimension $<p^{d+1}$, again by Proposition \ref{Propi1bound}. This proves the inequality in the second statement, and hence the lemma. \end{proof} \end{lemma}

The Knebusch splitting pattern of a quasilinear $p$-form $\phi$ is therefore determined by $h(\phi)$, $\qp{h}(\phi)$ and the truncated sequence $\big(\witti{1}{\phi},\hdots,\witti{\qp{h}(\phi)}{\phi}\big)$. With a view to studying the latter invariant, we now introduce new invariants of $\phi$ which will be of central interest in the sequel. More specifically, for each $1 \leq r \leq h(\phi)$, we define the \emph{r-th higher divisibility index} of $\phi$, denoted $\mathfrak{d}_r(\phi)$, as the integer $\mathfrak{d}_0(\phi_r)$ (see \S \ref{QPforms}). In other words, $\mathfrak{d}_r(\phi)$ is the largest integer $s$ such that $\phi_r$ is divisible by an $s$-fold quasi-Pfister $p$-form (over the corresponding field of the Knebusch splitting tower of $\phi$). The sequence of integers $\big(\mathfrak{d}_0(\phi),\hdots,\mathfrak{d}_{h(\phi)}(\phi)\big)$ will be denoted by $\mathfrak{d}(\phi)$. As per Lemma \ref{Lemtruncationsplitting} (and the proof of Proposition \ref{Propi1bound}), we have:

\begin{lemma} \label{Lemtruncationdivisibility} Let $\phi$ be a quasilinear $p$-form over $F$ and let $d = h(\phi) - \qp{h}(\phi)$. Then $\mathfrak{d}(\phi) = (\mathfrak{d}_0(\phi),\hdots,\mathfrak{d}_{\qp{h}(\phi)-1}(\phi),d,d-1,\hdots,1,0)$.
\end{lemma}

We also, however, have the following information concerning the ``nontrivial part'' of the sequence $\mathfrak{d}(\phi)$:

\begin{lemma} \label{Lemincreasingdivisibility} Let $\phi$ be a quasilinear $p$-form over $F$. Then $\mathfrak{d}_0(\phi) \leq \hdots \leq \mathfrak{d}_{\qp{h}(\phi)}$.
\begin{proof} We may assume that $\phi$ is anisotropic and not split. We need to show that if $\phi$ is not similar to a quasi-Pfister $p$-form, then $\mathfrak{d}_1(\phi) \geq \mathfrak{d}_0(\phi)$. By Lemma \ref{Lemdivisibilityindexextension} it will be sufficient to check that $\simform{\phi}$ remains anisotropic over $F(\phi)$. Suppose otherwise. Then, by Corollary \ref{Corisotropyff} (1), we have $N(\phi) \subseteq G(\phi)$. By Corollary \ref{Corsimilarityfield} it follows that $N(\phi) = G(\phi)$, or, equivalently, that $\normform{\phi} \simeq \simform{\phi}$ (see Lemma \ref{Lemexistenceofforms}). But, in view of Lemma \ref{Lempropertyofnormform} and Corollary \ref{Corpropertyofsimform}, this implies that $\phi$ is similar to a quasi-Pfister $p$-form, thus contradicting our assumption. The lemma follows. \end{proof} \end{lemma}

In particular, since $\witti{r}{\phi} = \mydim{\phi_r} - \mydim{\phi_{r-1}}$ for all $1 \leq r \leq h(\phi)$ (see Remark \ref{RemsKnebusch} (2)), we obtain the following result concerning the integers $\witti{r}{\phi}$:

\begin{corollary} \label{Cordivisorsofindices} Let $\phi$ be a quasilinear $p$-form over $F$. Then $\witti{r}{\phi} \equiv 0 \pmod{\mathfrak{d}_{r-1}(\phi)}$ for all $1 \leq r \leq \qp{h}(\phi)$. \end{corollary}

\subsection{Two basic examples} \label{Twoexamples} We now conclude this section with two basic computations which will be needed in the sequel, beginning with:

\begin{lemma} \label{Lemaddavariable} Let $\phi$ be an anisotropic quasilinear $p$-form over $F$ and let $\psi = \phi_{F(T)} \perp \form{T}$, where $T$ is an algebraically independent variable over $F$. Then $\mathfrak{i}(\psi) = \big(1,\witti{1}{\phi},\witti{2}{\phi},\hdots,\witti{h(\phi)}{\phi}\big)$ and $\mathfrak{d}(\psi) = \big(0,\mathfrak{d}_0(\phi),\mathfrak{d}_1(\phi),\hdots,\mathfrak{d}_{h(\phi)}(\phi)\big)$.
\begin{proof} The form $\psi$ is clearly anisotropic. Now, the field $F(T)(\psi)$ is $F$-isomorphic to a purely transcendental extension of $F$ (see the presentation of Remark \ref{Remsff} (1), for example). In particular, $\phi_{F(T)(\psi)}$ is anisotropic (Lemma \ref{Lemseparableextensions}), and so $\witti{1}{\psi} = 1$ and $\psi_1 \simeq \phi_{F(T)(\psi)}$ (see Remark \ref{Remanispart}). Since $F(T)(\psi)$ is purely transcendental over $F$, the first statement now follows immediately from Lemma \ref{Lemseparableextensions}. In a similar way, Lemma \ref{Lemdivisibilityindexstability} implies that $\mathfrak{d}_r(\psi) = \mathfrak{d}_{r-1}(\phi)$ for all $1 \leq r \leq h(\psi)$. Thus, to prove the second statement, it only remains to check that $\mathfrak{d}_0(\psi) = 0$. But, since $\witti{1}{\psi} = 1$, this is an immediate consequence of Corollary \ref{Cordivisorsofindices}. \end{proof} \end{lemma}

Given this result, we can give an example of a quasilinear $p$-form whose (full) splitting pattern is \emph{not} determined its Knebusch splitting pattern:

\begin{example}[{cf. \cite[Ex. 8.15]{HoffmannLaghribi1}}] \label{Exnongeneric} Let $T = (T_1,\hdots,T_{n+1})$ be a tuple of algebraically independent variables over a field $F_0$ of characteristic $p$, and let $F = F_0(T)$. Consider the form $\phi = \pfister{T_1,\hdots,T_n} \perp \form{T_{n+1}}$ over $F$. By Lemma \ref{Lemaddavariable} and Proposition \ref{Propi1bound}, we have $\mathfrak{i}(\phi) = (1,p^n - p^{n-1},p^{n-1} - p^{n-2},\hdots,p^2-p,p-1)$, so that $\wittj{r}{\phi} = p^n - p^{n-r+1} + 1$ for all $1 \leq r \leq n+1$. On the other hand, the full splitting pattern of $\phi$ also contains all the integers $p^n - p^{n-r+1}$ ($1 \leq r \leq n+1$). Indeed, if $(L_s)$ denotes the Knebusch splitting tower of $\pfister{a_1,\hdots,a_n}_L$, then we clearly have $\witti{0}{\phi_{L_{r-1}}} = p^n - p^{n-r+1}$ for all $1 \leq r \leq n+1$ (again, we are using Lemma \ref{Lemseparableextensions} and Proposition \ref{Propi1bound} here). \end{example}

Our second computation is the following:

\begin{lemma} \label{LemgenericQPmultiple} Let $\phi$ be a quasilinear $p$-form over $F$ and let $\psi = \pfister{T_1,\hdots,T_n} \otimes \phi_{F(T)}$, where $T=(T_1,\hdots,T_n)$ is a tuple of algebraically independent variables over $F$. Then $\mathfrak{i}(\psi) = (p^n\witti{1}{\phi},\hdots,p^n\witti{h(\phi)}{\phi},p^n - p^{n-1},p^{n-1} - p^{n-2},\hdots,p^2 - p,p-1)$ and $\mathfrak{d}(\psi) = (\mathfrak{d}_0(\phi) + n, \mathfrak{d}_1(\phi) + n,\hdots, \mathfrak{d}_{h(\phi)-1}(\phi) + n, n, n-1,\hdots,1,0)$.
\begin{proof} It is enough to treat the case where $n=1$. To simplify the notation, let us write $T$ for the variable $T_1$ and $L$ for the rational function field $F(T)$. Now, by construction, we have $\mathrm{ndeg}(\psi) = p\big(\mathrm{ndeg}(\phi)\big)$ (see Remark \ref{Remexplicitnormfield}). In view of Corollary \ref{Corheight=ndegree}, it follows that $h(\psi) = h(\phi) +1$. Let $(L_r)$ and $(F_r)$ denote the Knebusch splitting towers of $\psi$ and $\phi$ respectively. We claim that, for every $0 \leq r \leq h(\phi)$, $(L_r)_T$ is $F$-isomorphic to a purely transcendental extension of $F_r$. The case where $r = 0$ is evident. In general, we have $(L_r)_T = (L_T)_r$, where $((L_T)_r)$ denotes the Knebusch splitting tower of $\psi_{L_T}$. But, since $\pfister{T}_{L_T} \sim \form{1}$, and since $L_T$ is purely transcendental over $F$ (the $r=0$ case), we have $\anispart{(\psi_{L_T})} \simeq \phi_{L_T}$. By Remark \ref{Remanispart}, it follows that $(L_r)_T$ is $L$-isomorphic to a purely transcendental extension of the free composite $F_r \cdot L_T$. Again, since $L_T$ is purely transcendental over $F$, the claim follows. Given this, Proposition \ref{PropindexofmultiplesofQP} and Lemma \ref{Lemseparableextensions} together imply that
\begin{equation*} \wittj{r}{\psi} = \witti{0}{\psi_{L_r}} = p\witti{0}{\phi_{(L_r)_T}} = p\witti{0}{\phi_{F_r}} = p\wittj{r}{\phi} \end{equation*}
for all $0 \leq r \leq h(\phi)$, which proves the first statement of the lemma. Similarly, our claim, Proposition \ref{PropindexofmultiplesofQP} and Lemma \ref{Lemdivisibilityindexstability} together imply that
\begin{equation*} \mathfrak{d}_r(\psi) = \mathfrak{d}_0(\psi_{L_r}) = \mathfrak{d}_0(\phi_{(L_r)_T}) + 1 = \mathfrak{d}_0(\phi_{F_r}) + 1 = \mathfrak{d}_r(\phi) + 1 \end{equation*}
for all $0 \leq r \leq h(\phi)$, and so the second statement also holds.\end{proof} \end{lemma}

\section{An incompressibility theorem and related results} \label{Incompressibilityresults}

In this section, we collect some of the farther-reaching results on the isotropy behaviour of quasilinear $p$-forms over function fields of quasilinear $p$-hypersurfaces which have been obtained in recent years. These results will have an essential role to play in the sequel. We do not provide full details here, but the interested reader is referred to the original articles (\cite{HoffmannLaghribi1,Totaro1,Scully2}) for further information.

\subsection{The incompressibility theorem} \label{Incompressibilitythm} Let $\phi$ be an anisotropic quasilinear $p$-form of dimension $\geq 2$ over $F$ with associated quasilinear $p$-hypersurface $X_\phi$. As in \cite[\S 5]{Scully1}, we define the \emph{Izhboldin dimension} of $X_\phi$, denoted $\mathrm{dim}_{\mathrm{Izh}}(X_\phi)$, to be the integer $\mydim{X_\phi} - \witti{1}{\phi} + 1$. The following result was proved in \emph{loc. cit.}:

\begin{theorem}[{\cite[Thm. 5.12]{Scully1}}] \label{Thmincompressibility} Let $X$ be an anisotropic quasilinear $p$-hypersurface over $F$ and let $Y$ be an algebraic variety over $F$ such that $Y(F_{\mathrm{sep}}) = \emptyset$. If $\mydim{Y} < \mathrm{dim}_{\mathrm{Izh}}(X)$, then there cannot exist a rational map $X \dashrightarrow Y$. \end{theorem}

\begin{remark} In the case where $p=2$, Theorem \ref{Thmincompressibility} is due to Totaro (\cite[Thm. 5.1]{Totaro1}). In fact, if $X_\phi = \lbrace \phi = 0 \rbrace$ is an anisotropic projective quadric over a field $k$ of any characteristic, and if $Y$ is any complete $k$-variety possessing no closed points of odd degree, then it is known that the existence of a rational map $X \dashrightarrow Y$ necessarily implies the inequality $\mydim{Y} \geq \mathrm{dim}_{\mathrm{Izh}}(X_{\phi})$, where $\mathrm{dim}_{\mathrm{Izh}}(X_{\phi}) = \mydim{X_\phi} - \witti{1}{\phi} + 1$. This result was first proved by Karpenko and Merkurjev in the case where $X_\phi$ is smooth (\cite[Thm 4.1]{KarpenkoMerkurjev}, \cite[Thm. 76.5]{EKM}), and was later extended by Totaro (\emph{loc. cit.}) to the singular case. \end{remark}

In the remainder of this section, we will recall some of the main applications of Theorem \ref{Thmincompressibility} (and its proof). Here, we mention the following:

\begin{corollary}[{cf. \cite[Cor. 5.4]{Scully2}}] \label{Corsubformanisotropy} Let $\phi$ and $\psi$ be anisotropic quasilinear $p$-forms of dimension $\geq 2$ over $F$, and let $\sigma \subset \phi$ be a subform of dimension $\leq \mydim{\psi} - \witti{1}{\psi}$. Then $\sigma_{F(\psi)} \subset \anispart{(\phi_{F(\psi)})}$. In particular, $\sigma_{F(\psi)}$ is anisotropic.
\begin{proof} We trivially have $D(\sigma_{F(\psi)}) \subseteq D(\phi_{F(\psi)})$. In light of Proposition \ref{Propanisclassification}, it therefore suffices to check that $\sigma_{F(\psi)}$ is anisotropic, or, equivalently, that there does not exist a rational map $X_{\psi} \dashrightarrow X_\sigma$. But, since $X_{\sigma}(F_{\mathrm{sep}}) = \emptyset$ (see Lemma \ref{Lemseparableextensions}), this is an immediate consequence of Theorem \ref{Thmincompressibility}. \end{proof} \end{corollary}

In particular, we have the following fundamental observation:

\begin{corollary} \label{Corfirstkernel} Let $\phi$ be an anisotropic quasilinear $p$-form of dimension $\geq 2$ over $F$ and let $\psi \subset \phi$ be a subform of codimension $\witti{1}{\phi}$. Then $\phi_1 \simeq \psi_{F(\phi)}$. 
\begin{proof} By Corollary \ref{Corsubformanisotropy}, we have $\psi_{F(\phi)} \subset \phi_1$. Since both forms have the same dimension by hypothesis, the result follows. \end{proof} \end{corollary}

\subsection{Neighbours and near neighbours, I} \label{Neighbours1} Let $\psi$ and $\phi$ be anisotropic quasilinear $p$-forms of dimension $\geq 2$ over $F$. We will say that $\psi$ is a \emph{neighbour} (resp. \emph{near neighbour}) of $\phi$ if $\psi$ is similar to a subform of codimension $< \witti{1}{\phi}$ (resp. codimension $\witti{1}{\phi}$) of $\phi$. Our motivation here is the following extension of Corollary \ref{Corfirstkernel}:

\begin{lemma} \label{Lemnearneighbours} Let $\phi$ and $\psi$ be anisotropic quasilinear $p$-forms over $F$ such that $\phi$ is anisotropic of dimension $\geq 2$ and $\psi$ is similar to a subform of $\phi$. Then $\anispart{(\psi_{F(\phi)})}$ is similar to $\phi_1$ if and only if $\psi$ is a neighbour or near neighbour of $\phi$. 
\begin{proof} We may assume that $\psi \subset \phi$. Again, we trivially have $D(\psi_{F(\phi)}) \subseteq D(\phi_{F(\phi)}) = D(\phi_1)$. By Lemma \ref{Propanisclassification}, it therefore suffices to check that $\mydim{\anispart{(\psi_{F(\phi)})}} = \mydim{\phi} - \witti{1}{\phi}$ if and only if $\psi$ has codimension $\leq \witti{1}{\phi}$ in $\phi$. The left-to-right implication here is trivial. Conversely, if $\psi$ has codimension $\leq \witti{1}{\phi}$ in $\phi$, then we have $\mydim{\anispart{(\psi_{F(\phi)})}} \geq \mydim{\phi} - \witti{1}{\phi}$ by Theorem \ref{Thmincompressibility}. Since the reverse inequality holds here by obvious dimension reasons, the lemma is proved. \end{proof} \end{lemma}

Note here that while neighbours of $\phi$ become anisotropic over $F(\phi)$, its \emph{near} neighbours do not (Corollary \ref{Corsubformanisotropy}). This enables us to compute:

\begin{proposition}[{cf. \cite[Prop. 6.1]{Scully1}}] \label{Propneighbouri1} Let $\phi$ and $\psi$ be anisotropic quasilinear $p$-forms of dimension $\geq 2$ over $F$ such that $\psi$ is a codimension-$d$ neighbour of $\phi$. Then $\witti{1}{\psi} = \witti{1}{\phi} - d$. \end{proposition}

\subsection{The Ruledness theorem} \label{Ruledness} Another key application of Theorem \ref{Thmincompressibility} is the following extension of Proposition \ref{Propneighbouri1}, which shows in a precise way that anisotropic quasilinear $p$-hypersurfaces having first higher isotropy index larger than 1 are ruled.

\begin{theorem}[{cf. \cite[Thm. 7.6]{Scully1}}] \label{Thmruledness} Let $\phi$ and $\psi$ be anisotropic quasilinear $p$-forms of dimension $\geq 2$ over $F$ such that $\psi$ is a codimension-$d$ neighbour of $\phi$. Then $X_\phi$ is birationally isomorphic to $X_{\psi} \times_F \mathbb{P}^d$.
\begin{proof} By Proposition \ref{Propneighbouri1}, we have $\witti{1}{\psi} = \witti{1}{\phi} - d$. It is therefore enough to treat the case where $d = \witti{1}{\psi} - 1$, and this is covered by \cite[Thm. 7.6]{Scully1}. \end{proof} \end{theorem}

\begin{remark} Again, in the case where $p=2$, this result is due to Totaro (\cite[Thm. 6.4]{Totaro1}). Unlike Theorem \ref{Thmincompressibility}, however, the analogous assertion remains open for generically smooth quadrics (in any characteristic - see \cite{Totaro1,Totaro2}). \end{remark}

It is worth mentioning the following explicitly:

\begin{corollary} \label{Coraffinecodim1} Let $\phi$ and $\psi$ be an anisotropic quasilinear $p$-forms of dimension $\geq 2$ over $F$ such that $\witti{1}{\phi}>1$ and $\psi$ is similar to a codimension-1 subform of $\phi$. Then we have an $F$-isomorphism of fields $F(\phi) \simeq F[\psi]$. 
\begin{proof} By Theorem \ref{Thmruledness}, $F(\phi)$ is $F$-isomorphic to a degree-one purely transcendental extension of $F(\psi)$. In view of Remark \ref{Remsff} (3), the result follows. \end{proof} \end{corollary}

\subsection{Neighbours and near neighbours, II} \label{Neighbours2} Given Theorem \ref{Thmruledness}, one can extend Proposition \ref{Propneighbouri1} as follows:

\begin{proposition}[{cf. \cite[Prop. 6.2]{Scully2}}] \label{Propsspneighbours} Let $\phi$ and $\psi$ be anisotropic quasilinear $p$-forms of dimension $\geq 2$ over $F$ such that $\psi$ is a codimension-$d$ neighbour of $\phi$. Then $\mathfrak{i}(\psi) = \big(\witti{1}{\phi} - d, \witti{2}{\phi},\hdots,\witti{h(\phi)}{\phi}\big)$ and $\mathfrak{d}(\psi) = \big(\mathfrak{d}_0(\psi),\mathfrak{d}_1(\phi),\hdots,\mathfrak{d}_{h(\phi)}(\phi)\big)$.
\begin{proof} By Theorem \ref{Thmruledness}, $F(\phi)$ is $F$-isomorphic to a purely transcendental extension of $F(\psi)$. Thus, if $(F_r)$ and $\big(F(\phi)_r\big)$ denote the Knebusch splitting towers of $\psi$ and $\psi_{F(\phi)}$, respectively, then $F(\phi)_r$ is $F_r$-isomorphic to a purely transcendental extension of $F_{r+1}$ for every $0 \leq r \leq h(\psi_{F(\phi)})$. In particular, we have $\witti{r}{\psi_{F(\phi)}} = \witti{r+1}{\psi}$ and $\mathfrak{d}_r(\psi_{F(\phi)}) = \mathfrak{d}_{r+1}(\psi)$ for any such $r$ by Lemmas \ref{Lemseparableextensions} and \ref{Lemdivisibilityindexstability} , respectively. On the other hand, Lemma \ref{Lemnearneighbours} shows that $\anispart{(\psi_{F(\phi)})}$ is similar to $\phi_1$. Since $\witti{r}{\phi} = \witti{r-1}{\phi_1}$ and $\mathfrak{d}_r(\phi) = \mathfrak{d}_{r-1}(\phi_1)$ for all $1 \leq r \leq h(\phi)$, the proposition follows immediately. \end{proof} \end{proposition}

Let $\phi$ be a quasilinear $p$-form over $F$. We say that $\phi$ is a \emph{quasi-Pfister $p$-neighbour} (of $\pi$) if there exists a quasi-Pfister $p$-form $\pi$ over $F$ such that $\phi$ is similar to a subform of $\pi$ and $\mydim{\phi} > \frac{1}{p}\mydim{\pi}$. If $\phi$ is anisotropic, then it follows from Proposition \ref{Propi1bound} that $\phi$ is quasi-Pfister $p$-neighbour if and only if it is a neighbour of some quasi-Pfister $p$-form in the sense of \S \ref{Neighbours1}. Forms of this type are of special importance in the general theory of quasilinear $p$-forms. Putting Lemma \ref{Lempropertyofnormform}, Proposition \ref{Propi1bound}, Corollary \ref{Corheight=ndegree}, Lemma \ref{Lemnearneighbours} and Proposition \ref{Propsspneighbours} together, we obtain the following classification of anisotropic quasi-Pfister $p$-neighbours:

\begin{corollary}[{cf. \cite[Thm. 6.4]{Scully2}}] \label{CorClassificationQPN} Let $\phi$ be an anisotropic quasilinear $p$-form of dimension $\geq 2$ over $F$ and let $n$ be the smallest nonnegative integer such that $p^{n+1} \geq \mydim{\phi}$. Then the following are equivalent:
\begin{enumerate} \item $\phi$ is a quasi-Pfister $p$-neighbour.
\item $\phi$ is a neighbour of $\normform{\phi}$.
\item $\phi_{F(\normform{\phi})}$ is isotropic.
\item $\mathrm{ndeg}(\phi) = p^{n+1}$.
\item $h(\phi) = n+1$.
\item $\witti{1}{\phi} = \mydim{\phi} - p^n$ and $\witti{2}{\phi} = p^n - p^{n-1}$.
\item $\mathfrak{i}(\phi) = (\mydim{\phi} - p^n, p^n - p^{n-1},p^{n-1} - p^{n-2},\hdots,p^2-p,p-1)$.
\item $\phi_1$ is similar to a quasi-Pfister $p$-form (i.e., $\qp{h}(\phi) \leq 1$). \end{enumerate} \end{corollary}

\begin{remark} In the case where $p=2$, this result was proved earlier by Hoffmann and Laghribi (cf. \cite[Thm. 8.1]{HoffmannLaghribi1}) using different methods. \end{remark}

For arbitrary subforms, the situation is naturally more complicated, but we can nevertheless appeal to the following general result which was proved in \cite{Scully2} (and whose proof again makes essential use of Theorems \ref{Thmincompressibility} and \ref{Thmruledness}):

\begin{proposition}[{cf. \cite[Prop. 8.6]{Scully2}}] \label{Propsspsubform} Let $\phi$ and $\psi$ be anisotropic quasilinear $p$-forms of dimension $\geq 2$ over $F$ such that $\phi_{F(\psi)}$ is isotropic. Then, either
\begin{enumerate} \item $(\psi_r)_{F_r(\phi)}$ is anisotropic for all $0 \leq r < h(\psi)$ and $\mathfrak{i}(\psi_{F(\phi)}) = \mathfrak{i}(\psi)$, or
\item $\mathfrak{i}(\psi_{F(\phi)}) = \big(\witti{1}{\psi},\hdots,\witti{s-1}{\psi},\witti{s}{\psi} + \witti{s+1}{\psi}, \witti{s+2}{\psi},\hdots,\witti{h(\psi)}{\psi}\big)$, where $s<h(\psi)$ is the smallest nonnegative integer such that $(\psi_s)_{F_s(\phi)}$ is isotropic. \end{enumerate} \end{proposition}

Note here that in the special situation where $\psi$ is a neighbour of $\phi$, we are necessarily in case (2) with $s$ being equal to 0. Since $\anispart{(\psi_{F(\phi)})} \simeq \phi_1$ in this instance (Lemma \ref{Lemnearneighbours}), we thus recover the computation of Proposition \ref{Propsspneighbours}. By contrast, if $\psi$ is a \emph{near} neighbour of $\phi$, then we can be in either of cases (1) and (2) (see Remark \ref{Remnearneighbours} below). Nevertheless, we still have $\psi_{F(\phi)} \simeq \phi_1$ (Corollary \ref{Corfirstkernel}), and so we get:

\begin{corollary}[{cf. \cite[Cor. 6.10]{Scully2}}] \label{Corsspnearneighbours} Let $\phi$ be an anisotropic quasilinear $p$-form of dimension $\geq 2$ over $F$ and let $\psi$ be a near neighbour of $\phi$. Then, either
\begin{enumerate} \item $(\psi_r)_{F_r(\phi)}$ is anisotropic for all $0 \leq r < h(\psi)$ and $\mathfrak{i}(\psi) = \mathfrak{i}(\phi_1)$, or
\item $\mathfrak{i}(\psi) = \big(\witti{2}{\phi},\witti{3}{\phi},\hdots,\witti{s}{\phi}, \witti{s}{\psi}, \witti{s+1}{\phi} - \witti{s}{\psi}, \witti{s+2}{\phi}\hdots,\witti{h(\phi)}{\phi}\big)$, where $s<h(\psi)$ is the smallest positive integer such that $(\psi_s)_{F_s(\phi)}$ is isotropic. \end{enumerate} \end{corollary}

\begin{remark} \label{Remnearneighbours} As per the comments above, neither of cases (1) and (2) can be ruled out here. Indeed, case (1) describes the situation where $h(\psi) = h(\phi) - 1$, while case (2) describes that where $h(\psi) = h(\phi)$. As the reader will easily verify using Corollary \ref{Corheight=ndegree}, both these situations can arise in practice. \end{remark}

\subsection{A comparison result} \label{Comparisontheorem} We now conclude this section by recalling the following comparison result for isotropy indices of quasilinear quadratic forms which was obtained in \cite{Scully2} with the help of Theorem \ref{Thmincompressibility}:

\begin{proposition}[{cf. \cite[Thm. 7.13, Rem. 9.1]{Scully2}}] \label{Propcomparison} Assume that $p=2$. Let $\phi$ be an anisotropic quasilinear quadratic form of dimension $\geq 2$ over $F$ and let $L$ be a field extension of $F$ such that $\phi_L$ is not split. Then:
\begin{equation*} \witti{0}{\phi_{L(\phi)}} - \witti{1}{\phi} \geq \mathrm{min} \bigg\lbrace \witti{0}{\phi_L},\bigg[\frac{\mydim{\phi} - \witti{1}{\phi} + 1}{2}\bigg]\bigg \rbrace. \end{equation*} \end{proposition}

\begin{remark} A similar statement also holds for $p>2$ (see \cite[Thm. 7.13]{Scully2}), but this will not be needed below. \end{remark}

Finally, it will be convenient to record here the following application of this result which was also obtained in \cite{Scully2}:

\begin{theorem}[{\cite[Thm. 9.2]{Scully2}}] \label{Thmouterexcellentconnections} Let $\phi$ be an anisotropic quasilinear quadratic form of dimension $\geq 2$ over $F$ and write $\mydim{\phi} = 2^n + m$ for uniquely determined integers $n \geq 0$ and $1\leq m \leq 2^n$. Then, for any field extension $L$ of $F$, we either have $\witti{0}{\phi_L} \geq m$ or $\witti{0}{\phi_L} \leq m- \witti{1}{\phi}$. \end{theorem}

\begin{remark} Taking $L = F(\phi)$ here, we see that if $\witti{1}{\phi} < m$, then $\witti{1}{\phi} \leq \frac{m}{2}$. This result will now be subsumed in our Theorem \ref{ThmHoffmannsconjecture}. \end{remark}

\section{A motivational example} \label{Warmup}

As a warm-up for the proof of our main result, we will now prove Proposition \ref{PropMotivation}. Throughout this section, we assume that $p=2$. By a \emph{quasi-Pfister form}, we will mean a quasi-Pfister 2-form. Our assumption on the prime $p$ is imposed for simplicity, but also because we will make use of the following fact already mentioned in the introduction, the analogue of which is unknown when $p>3$ (see \cite[\S 4.1]{Scully2}):

\begin{lemma}[{cf. \cite[Cor. 7.22]{Hoffmann2}}] \label{Lemminimalityi1} Let $\phi$ be an anisotropic quasilinear quadratic form of dimension $\geq 2$ over $F$ and let $L$ be a field extension of $F$ such that $\phi_L$ is isotropic. Then $\witti{0}{\phi_L} \geq \witti{1}{\phi}$. \end{lemma}

Now, let $\phi$ be an anisotropic quasilinear quadratic form of dimension $\geq 2$ over $F$ and let $\psi \subset \phi$ be a subform of codimension $\witti{1}{\phi}$. By Corollary \ref{Corfirstkernel}, $\psi_{F(\phi)}$ is isomorphic to the first higher anisotropic kernel of $\phi$. In the next section, we will prove Theorem \ref{Maintheorem}, which asserts that the latter form is divisible by a quasi-Pfister form of dimension $\geq \witti{1}{\phi}$. Here, we will consider some special situations in which this divisibility property is already visible over the base field $F$. To this end, we will be interested in the following technical condition on the pair $(\phi, \psi)$:
\vspace{2mm}
\begin{enumerate} \item[$(\star)$] There exist elements $a \in D(\phi) \setminus \lbrace 0 \rbrace$ and $b \in D(\psi) \setminus \lbrace 0 \rbrace$ such that $a \neq c(bd + ef)$ for any $c,d,e,f \in D(\psi)$.\end{enumerate}

\begin{example} \label{Exstarcondition} If, in the above situation, we have $\mathrm{ndeg}(\psi) < \mathrm{ndeg}(\phi)$ (equivalently, if $h(\psi) < h(\phi)$ -- see Corollary \ref{Corheight=ndegree}), then $(\star)$ holds for the pair $(\phi,\psi)$. Indeed, in this case, $D(\phi)$ is (evidently) not contained in $cN(\psi)$ for any $c \in F$. Since $bd + ef \in N(\psi)$ for every $b,d,e,f \in D(\psi)$, the validity of $(\star)$ is immediately verified. \end{example}

In light of Example \ref{Exstarcondition}, Proposition \ref{PropMotivation} is subsumed in the following result:

\begin{proposition} \label{Propwarmup} Let $\phi$ be an anisotropic quasilinear quadratic form of dimension $\geq 2$ over $F$ and let $\psi \subset \phi$ be subform of codimension $\witti{1}{\phi}$ such that the pair $(\phi,\psi)$ satisfies $(\star)$. Then there exist a quasi-Pfister form $\pi$, a subform $\sigma \subset \pi$, an element $\lambda \in D(\phi)$ and a form $\tau$ over $F$ such that $\psi \simeq \pi\otimes \tau$ and $\phi \simeq \psi \perp \lambda\sigma$. 
\begin{proof} Let $b \in D(\psi) \setminus \lbrace 0 \rbrace$ be as in $(\star)$. Then $(\star)$ also holds for the pair $(b\phi,b\psi)$. Indeed, if this were not the case, then, for every $a \in D(\phi)$, we could find $c,d,e,f \in D(\psi)$ such that $ba = bc(b^3d + bebf)$, or, equivalently, such that $a = b^2c(bd + ef)$. But, since $D(\psi)$ is an $F^2$-linear subspace of $F$, we have $b^2c \in D(\psi)$, and so this would contradict the fact that $(\star)$ holds for the original pair $(\phi,\psi)$. Since the exchange $(\phi,\psi) \rightarrow (b\phi,b\psi)$ does not affect the statement of the proposition, we can therefore assume that $b = 1$. In other words. we can assume that $1 \in D(\psi)$ and that:

\begin{enumerate} \item[$(\star')$] There exists an element $a \in D(\phi) \setminus \lbrace 0 \rbrace$ such that $a \neq c(d + ef)$ for any $c,d,e,f \in D(\psi)$. \end{enumerate}

\noindent Let us fix $a \in D(\phi) \setminus \lbrace 0 \rbrace$ as in $(\star')$. We will now prove that the statement of the proposition holds with $\lambda = a$. First, we note that $(\star')$ implies:
\begin{enumerate}  \item $a \notin D(\psi)$.
\item $\psi_{F_{au}}$ is anisotropic for every $u \in D(\psi) \setminus \lbrace 0 \rbrace$. \end{enumerate}
Indeed, if $a$ were in $D(\psi)$, then we could contradict $(\star')$ by taking $c=1$, $d=a$ and $e=f=0$. Similarly, if $\psi_{F_{au}}$ were isotropic in (2), then (since $\psi$ is anisotropic) we could find nonzero elements $x,y \in D(\psi)$ such that $au = xy$ (see Lemma \ref{LemIsotropypurelyinseparable} (2)); taking $c = x^{-1}$, $d=0$, $e=y$ and $f = u$, this would again contradict $(\star')$.

Now, (1) implies that we have $\psi \perp \form{a} \subset \phi$. If $\witti{1}{\phi} = 1$, then the latter inclusion is an isomorphism, and the statement of the proposition holds with $\pi = \sigma = \form{1}$ and $\tau = \psi$. Assume now that $\witti{1}{\phi} >1$. Since $1,a \in D(\phi)$, $\phi_{F_a}$ is isotropic by Lemma \ref{LemIsotropypurelyinseparable} (2). By Lemma \ref{Lemminimalityi1}, it then follows that $\witti{0}{\phi_{F_a}} \geq \witti{1}{\phi}$. On the other hand, (2) (with $u=1 \in D(\psi)$) shows that $\psi_{F_a}$ is anisotropic, and, since $\mydim{\psi} = \mydim{\phi} - \witti{1}{\phi}$, we conclude that $\anispart{(\phi_{F_a})} \simeq \psi_{F_a}$. In other words, we have $D(\phi_{F_a}) = D(\psi_{F_a}) = D(\psi) + aD(\psi)$ (where the latter equality holds by Lemma \ref{LemIsotropypurelyinseparable} (1)). By Lemma \ref{Lemexistenceofforms}, it follows that we can write $\phi \simeq \psi \perp \form{a} \perp a\sigma'$ for some form $\sigma' \subset \psi$. Let $\sigma = \form{1} \perp \sigma$, so that $\phi \simeq \psi \perp a\sigma$. As $1 \in D(\psi)$, we have $\sigma \subset \psi$. Since $\sigma$ is anisotropic and represents 1, Lemma \ref{Lempropertyofnormform} shows that $\sigma \subset \normform{\sigma}$. Thus, in order to complete the proof, it will be enough to prove that $\psi$ is divisible by $\pi = \normform{\sigma}$. By Corollary \ref{CordivbyQP}, this amounts to showing that $D(\sigma) \subset G(\psi)$. Let $x \in D(\sigma) \setminus \lbrace 0 \rbrace$. In order to show that $x \in G(\psi)$, we must check that $xy \in D(\psi)$ for all $y \in D(\psi)$ (see Lemma \ref{Lemsimilarityfactor}). If $y = 0$, then there is nothing to prove. Suppose now that $y \neq 0$. Then, by Lemma \ref{LemIsotropypurelyinseparable} (2), $\phi_{F_{ay}}$ is isotropic. On the other hand, $\psi_{F_{ay}}$ is anisotropic by (2). Using Lemma \ref{Lemminimalityi1} in the same way as before, we see that $\anispart{(\phi_{F_{ay}})} \simeq \psi_{F_{ay}}$, or, equivalently, that $D(\phi_{F_{ay}}) = D(\psi_{F_{ay}}) = D(\psi) + ayD(\psi)$. But, letting $u = \sqrt{ay}$, we have $xy = (a^{-1}u)^2x \in D(\phi_{F_{ay}})$. In particular, we can find $u, v \in D(\psi)$ such that $xy = u + ayv$. To complete the proof, it now only remains to show that $v = 0$. But, if $v \neq 0$, then we have $a = v^{-1}(x + uy^{-1})$, and we obtain a contradiction to $(\star')$ by taking $c = v^{-1}$, $d = x$, $e = u$ and $f = y^{-1}$. The result follows. \end{proof} \end{proposition}

In general, it is \emph{not} always possible to find a subform $\psi \subset \phi$ of codimension $\witti{1}{\phi}$ such that the pair $(\phi, \psi)$ satisfies condition $(\star)$. Indeed, note that if $\witti{1}{\phi}=2$ in the situation of Proposition \ref{Propwarmup}, then $\phi$ is divisible by the binary (i.e., 2-dimensional) form $\sigma$. The following example, which is directly analogous to an example of Vishik from the characteristic $\neq 2$ theory of quadratic forms, shows that there exist anisotropic quasilinear quadratic forms which have first higher isotropy index equal to 2, but which are not divisible by a binary form: Let $a,b,c,d,e$ be algebraically independent variables over a field $F_0$ of characteristic 2, let $F = F_0(a,b,c,d,e)$ and consider the anisotropic $F$-form $\phi = \pfister{ab,ac,ad} \perp bcd\form{1,ab,ac,ad} \perp e\form{a,b,c,d}$. 

\begin{lemma}[{Vishik, cf. \cite[Lem. 7.1]{Totaro2}}] \label{LemVishiksexample} In the above situation, we have $\witti{1}{\phi} = 2$, but $\phi$ is not divisible by a binary form. 
\begin{proof} We will make use of some basic facts from the theory of symmetric bilinear forms over fields of characteristic 2. For the relevant notation and terminology, the reader is referred to the appendix below. Now, let $\mathfrak{b}$ denote the bilinear form $\pfister{ab,ac,ad}_b \perp bcd\form{1,ab,ac,ad}_b \perp e\form{a,b,c,d}_b$ over $F$, so that $\phi = \phi_{\mathfrak{b}}$ (i.e., $\phi$ is the diagonal part of $\mathfrak{b}$). As the reader will immediately verify, we have
\begin{equation*} \mathfrak{b} \sim \mathfrak{c} \coloneqq \pfister{a,b,c,d}_b \perp \pfister{e}_b \otimes \form{a,b,c,d}_b. \end{equation*}
Let $\pi = \pfister{a,b,c,d}_b$, and let $\pi'$ denote the pure subform of $\pi$. Since $\mathrm{ndeg}(\phi_{\pi}) = 16 < 32 = \mathrm{ndeg}(\phi)$, Corollary \ref{Corisotropyff} (2) implies that $\pi_{F(\phi)}$ is anisotropic. At the same time, it follows from Remark \ref{Remsff} (4) and the definition of $\phi$ that $e \in D(\pi'_{F(\phi)})$. Thus, by \cite[Lem. 6.1]{EKM}, there exists a 3-fold bilinear Pfister form $\eta$ over $F(\phi)$ such that $\pi_{F(\phi)} \simeq \pfister{e}_b \otimes \eta$. In particular, $\mathfrak{c}_{F(\phi)}$ is divisible by $\pfister{e}_b$, and so $\witti{0}{\mathfrak{c}_{F(\phi)}}$ is even (see \cite[Prop. 6.22]{EKM}). Since $\mydim{\mathfrak{c}} - \mydim{\mathfrak{b}} = 8 \equiv 0 \pmod{4}$, it follows that $\witti{0}{\mathfrak{b}_{F(\phi)}}$ is also even. In particular, we have $\witti{1}{\mathfrak{b}} \geq 2$, and, in order to prove that equality holds here, we just need to find a field extension $L$ of $F$ such that $\phi_L$ is isotropic, but $\witti{0}{\phi_L} \leq 2$ (see Lemma \ref{Lemminimalityi1}). We claim that $L = F_{cde}$ is such an extension. First, note that since $cde = (bcd)(b^{-1}e)$ is a product of two nonzero elements of $D(\phi)$, $\phi_L$ is isotropic by Lemma \ref{LemIsotropypurelyinseparable} (2). On the other hand, it is easy to see that the codimension-2 subform $\psi = \pfister{ab,ac,ad} \perp bcd\form{1,ab,ac,ad}  \perp e\form{c,d} \subset \phi$ remains anisotropic over $L$. Indeed, since
\begin{equation*} \psi_L \simeq \pfister{ab,ac,ad} \perp bcd \form{1,ab,ac,ad} \perp \form{d,c} \subset  \pfister{a,b,c,d}_L, \end{equation*}
it suffices to check that $\pfister{a,b,c,d}$ remains anisotropic over $L$. But, since $cde \notin F^2(a,b,c,d) = N(\pfister{a,b,c,d})$, this follows from Lemma \ref{LemIsotropypurelyinseparable} (4). Since the anisotropy of $\phi_L$ readily implies that $\witti{0}{\psi_L} \leq 2$, we have shown that $\witti{1}{\phi} = 2$.

It now remains to check that $\phi$ is not divisible by a binary form. For the sake of contradiction, suppose instead that $\phi$ is divisible by $\pfister{u}$ for some $u \in F \setminus F^2$. We claim that $u \in F^2(ab,ac,ad)$. Again, let us assume that this is not the case. Then the quasi-Pfister form $\tau = \pfister{ab,ac,ad}$ remains anisotropic over $F_u$ by Lemma \ref{LemIsotropypurelyinseparable} (4). Let $\sigma = \tau \perp \form{bcd} \subset \phi$ and $\eta = \tau \perp \form{ae} \subset \phi$. Since $\witti{0}{\phi_{F_u}} = \frac{1}{2}\mydim{\phi} = 8$ (Lemma \ref{LemIsotropypurelyinseparable} (6)), and since $\mydim{\sigma} = \mydim{\eta} = 9$, both $\sigma_{F_u}$ and $\eta_{F_u}$ are necessarily isotropic. Since $\tau_{F_u}$ is anisotropic, this means that we have $bcd,ae \in D(\tau_{F_u}) = F^2(ab,ac,ad,u)$. But this implies that $F^2(a,b,c,d,e) \subseteq F^2(ab,ac,ad,u)$, thus contradicting the fact that the elements $a,b,c,d,e$ are algebraically independent over $F$. This proves our claim, and so we can write $u = v+w$ for some $v \in D(\form{1,ab,ac,ad})$ and $w \in D(\form{bc,bd,cd,abcd})$. Now, Lemma \ref{LemIsotropypurelyinseparable} (6) implies that $u \in G(\phi)$. In particular, applying Lemma \ref{Lemsimilarityfactor} to the elements $ae, acd, abd \in D(\phi)$, we see that 
\begin{enumerate} \item $aeu \in D(\phi)$,
\item $acdu \in D(\phi)$, and 
\item $abdu \in D(\phi)$. \end{enumerate}
We can now complete the proof: First, note that since $ae\form{1,ab,ac,ad} \simeq e\form{a,b,c,d} \subset \phi$, we have $aev \in D(\phi)$. By (1), we this implies that $aew \in D(\phi)$. Note however that $ae\form{bc,bd,cd,abcd} \simeq \form{abce,abde,acde,bcde}$, and the latter form does not represent any nonzero element of $D(\phi)$. It follows that $w = 0$, and so $u \in D(\form{1,ab,ac,ad})$. Next, consider the form $\rho = acd\form{1,ab,ac,ad} \simeq \form{acd,bcd,c,d}$. By (2), we have $acdu \in D(\rho) \cap D(\phi)$. Since the elements $a,b,c,d,e$ are algebraically independent over $F$, direct inspection shows that the former intersection is equal to $acdD(\form{1,ab})$, and so $u \in D(\form{1,ab})$. Finally, we can use (3) in a similar way to show that $u \in F^2$, thus providing us with the needed contradiction. The lemma is proved.\end{proof} \end{lemma}

In fact, Vishik's example shows more: It is generally not possible to find a subform $\psi \subset \phi$ of codimension $\witti{1}{\phi}$ such that the pair $(\phi, \psi)$ satisfies condition $(\star)$, \emph{even after making arbitrary rational extensions of $F$} -- this stronger assertion follows from Lemma \ref{Lemdivisibilityindexstability}. This leads us to consider the possibility that decompositions of the kind suggested by Proposition \ref{Propwarmup} may be found by passing to suitable transcendental extensions of the base field, and, ultimately, to our main result. Nevertheless, it is still interesting to ask for conditions on $\phi$ (or, more specifically, on the Knebusch splitting pattern of $\phi$) which automatically ensure the existence of the needed subform $\psi$. To this end, it is natural to look to the extremities where $\phi$ is ``far from generic'' (e.g., where $\witti{1}{\phi}$ is ``large''). The basic example here is provided by Proposition \ref{Propi1bound}, which gives a characterisation of (scalar multiples of) anisotropic quasi-Pfister forms in terms of the first higher isotropy index. To give another example of this kind of equivalence, we state here the following conjecture, which is directly analogous to an open problem in the nonsingular theory of quadratic forms:

\begin{conjecture} \label{Conjh2good} Let $n$ be a positive integer, and let $\phi$ be an anisotropic quasilinear quadratic form of dimension $2^{n+1}$ over $F$ such that $\witti{1}{\phi} = 2^{n-1}$. Then $\phi$ is divisible by an $(n-1)$-fold quasi-Pfister form. \end{conjecture}

\begin{remark} Examples of such forms $\phi$ exist -- see Proposition \ref{Propexistenceofvalues} below. \end{remark}

Proposition \ref{Propwarmup} allows us to reformulate this problem as follows:

\begin{lemma} \label{Lemh2good} Let $n$ be a positive integer, and let $\phi$ be an anisotropic quasilinear quadratic form of dimension $2^{n+1}$ over $F$ such that $\witti{1}{\phi} = 2^{n-1}$. Then the following are equivalent:
\begin{enumerate} \item Conjecture \ref{Conjh2good} holds for the pair $(\phi,n)$.
\item There exists a subform $\psi \subset \phi$ of codimension $2^{n-1}$ such that $\mathrm{ndeg}(\psi) = 2^{n+1}$. \end{enumerate}
\begin{proof} Suppose first that Conjecture \ref{Conjh2good} holds for the pair $(\phi,n)$. In other words, we have $\phi \simeq \pi \otimes \tau$ for some $(n-1)$-fold quasi-Pfister form $\pi$ and some 4-dimensional form $\tau$. Then, if $\sigma$ is any codimension-1 subform of $\tau$ and $\psi = \pi \otimes \sigma$, then $\mathrm{ndeg}(\psi) \geq \mathrm{ndeg}(\pi)\cdot \mathrm{ndeg}(\sigma) = 2^{n-1}\cdot 4 = 2^{n+1}$. Since $\mydim{\psi} > 2^n$, the reverse inequality also holds here (see Lemma \ref{Lempropertyofnormform}), and so (1) implies (2). Conversely, if (2) holds, then the pair $(\phi, \psi)$ satisfies condition $(\star)$. Indeed, since $\witti{1}{\phi} = 2^{n-1}$, $\phi$ is not similar to a quasi-Pfister form (Proposition \ref{Propi1bound}), and so $\mathrm{ndeg}(\phi) > \mydim{\phi} = 2^{n+1}$ by Corollary \ref{CorClassificationQPN}. Our claim therefore follows from Example \ref{Exstarcondition}. In particular, we can apply Proposition \ref{Warmup} to the given pair. Let $\pi$ and $\sigma$ be as in the statement of the former result. Since $\mydim{\psi} = 2^n + 2^{n-1}$, and since $\witti{1}{\phi} = 2^{n-1}$, we see that $\pi$ must have dimension $2^{n-1} = \witti{1}{\phi}$, and so $\sigma \simeq \pi$. The statement of the proposition now says that $\phi$ is divisible by $\pi$, which shows that (1) holds. This proves the lemma. \end{proof}\end{lemma}

\begin{remark} Using Lemma \ref{Lemh2good}, it is not difficult to verify that Conjecture \ref{Conjh2good} holds for $n \leq 2$. \end{remark}

\section{Main theorem} \label{ProofofMainTheorem}

We are now ready to give the proof of Theorem \ref{Maintheorem}. In order to treat the case where $p>2$, the statement needs to be modified as follows:

\begin{theorem} \label{ThmMainthm} Let $\phi$ be an anisotropic quasilinear $p$-form of dimension $\geq 2$ over $F$ and let $s$ be the smallest nonnegative integer such that $p^s \geq \witti{1}{\phi}$. If $\phi$ is not a quasi-Pfister $p$-neighbour, then $\phi_1$ is divisible by an $s$-fold quasi-Pfister $p$-form. \end{theorem}

\begin{remark} Nothing is lost here by assuming that $\phi$ is not a quasi-Pfister $p$-neighbour. Indeed, if $\phi$ is a quasi-Pfister $p$-neighbour, and $n$ denotes the smallest nonnegative integer such that $p^{n+1} \geq \mydim{\phi}$, then $\phi_1$ is similar to an $n$-fold quasi-Pfister $p$-form by Corollary \ref{CorClassificationQPN}. If $p=2$, then we have $\witti{1}{\phi} \leq \frac{1}{2}\mydim{\phi} \leq 2^n$ by Proposition \ref{Propi1bound}, so that $n \geq s$, where $s$ is the integer defined in the statement of the theorem. Note, however, that if $p>2$, then $n$ may be strictly smaller than $s$ (again, see Corollary \ref{CorClassificationQPN}). This explains why the additional hypothesis is needed here, but not in the statement of Theorem \ref{Maintheorem} (i.e., the case where $p=2$). \end{remark}

\begin{proof} To simplify the notation, we will write $\mathfrak{i}_1$ instead of $\witti{1}{\phi}$ in what follows. If $\mathfrak{i}_1 = 1$, then the statement of the theorem holds trivially. We therefore assume henceforth that $\mathfrak{i}_1>1$. After multiplying $\phi$ by a nonzero scalar if necessary, we may also assume that $1 \in D(\phi)$. In particular, we can find $a_1,\hdots,a_n \in F$ such that $\phi \simeq \form{1,a_1,\hdots,a_n}$. For the remainder of the proof, we let $\phi' = \form{a_1,\hdots,a_n}$, and we write $\phi'(T)$ for the ``generic value'' of $\phi'$, i.e., $\phi'(T) = \sum_{i=1}^na_iT_i^p \in F[T]$, where $T = (T_1,\hdots,T_n)$ is a tuple of algebraically independent variables over $F$. By Corollary \ref{Coraffinecodim1}, the function field $F(\phi)$ is $F$-isomorphic to $F[\phi']$, and may therefore be identified with $\mathrm{Frac}\big(F[T]/(\phi'(T))\big)$ (see Remark \ref{Remsff} (1)). Fixing this identification henceforth, we will write $\overline{f}$ for the image of a polynomial $f \in F[T]$ under the canonical $F$-algebra homomorphism $F[T] \rightarrow F(\phi)$. We will also write $m(f)$ for the multiplicity $\mathrm{mult}_{\phi'(T)}(f)$ of $\phi'(T)$ in $f$, i.e., the largest integer $k$ such that $f = \phi'(T)^kh$ for some $h \in F[T]$. Note here that we have $\overline{f} \neq 0$ if and only if $m(f) = 0$.

Now, let $\psi \subset \phi$ be any subform of codimension $\witti{1}{\phi}$ such that $1 \in D(\psi)$. We then have the following lemma:

\begin{lemma} \label{LemMTL1} In the above situation, we can find elements $g_{i,j} \in D(\psi_{F[T]})\;(1 \leq i < \mathfrak{i}_1,\;1 \leq j <p)$ such that 
\begin{equation*} \phi_{F(T)} \simeq \psi_{F(T)} \oplus \phi'(T)\form{1,f_1,\hdots,f_{\mathfrak{i}_1 -1}}, \end{equation*} 
where, for each $1 \leq i < \mathfrak{i}_1$, $f_i = \sum_{j=1}^{p-1}g_{i,j}\phi'(T)^{j-1}$.
\begin{proof} First, let us note that $\phi'(T) \notin D(\psi_{F(T)})$. Indeed, if $\psi_{F(T)}$ were to represent $\phi'(T)$, then it would follow from Theorem \ref{ThmCasselsPfister} that $a_1,\hdots,a_n \in D(\psi)$. Since $1 \in D(\psi)$ by hypothesis, this would imply that $D(\phi) \subseteq D(\psi)$, or, equivalently, that $\phi \subset \psi$ (see Proposition \ref{Propanisclassification}), which is impossible for dimension reasons (recall here that $\mathfrak{i}_1>1$ by assumption). It follows that $\psi_{F(T)} \oplus \form{\phi'(T)}$ is anisotropic, and so $\psi_{F(T)} \oplus \form{\phi'(T)} \subset \phi_{F(T)}$ by Proposition \ref{Propanisclassification}. Now, the affine function field $F[\phi]$ may be identified (over $F$) with the field $K = F(T)_{\phi'(T)}$ (see Remark \ref{Remsff} (1)). Since $F[\phi]$ is $F$-isomorphic to a purely transcendental extension of $F(\phi)$ (Remark \ref{Remsff} (2)), Corollary \ref{Corfirstkernel} and Lemma \ref{Lemseparableextensions} together imply that $\psi_K \simeq \anispart{(\phi_K)}$. In particular, we have $D(\phi_{F(T)}) \subset D(\phi_K) = D(\psi_K) = \sum_{j=0}^{p-1}D(\psi_{F(T)})\phi'(T)^j$ (where the last equality holds by Lemma \ref{LemIsotropypurelyinseparable} (1)). Thus, by Lemma \ref{Lemexistenceofforms}, we can complete the subform inclusion $\psi_{F(T)} \oplus \form{\phi'(T)} \subset \phi_{F(T)}$ to an isomorphism $\phi_{F(T)} \simeq \psi_{F(T)} \oplus \form{\phi'(T)} \oplus \form{f_1',\hdots,f_{\mathfrak{i}_1-1}'}$, where, for each $i$, we have $f_i' = \sum_{j=0}^{p-1}g_{i,j}\phi'(T)^j$ for some $g_{i,j} \in D(\psi_{F(T)})$. Note, however, that every element of $D(\psi_{F(T)})$ is (trivially) the ratio of an element of $D(\psi_{F[T]})$ and a $p$-th power in $F[T]$. Since multiplying the $f_i'$ by $p$-th powers in $F[T]$ does not change the $F(T)$-form $\form{f_1',\hdots,f_{\mathfrak{i}_1-1}'}$ up to isomorphism (see Lemma \ref{Lemexistenceofforms}), we can arrange it so that the $g_{i,j}$ belong to $D(\psi_{F[T]})$. Similarly, since substracting elements of $D(\psi_{F(T)})$ from the $f_i'$ does not change the isomorphism class of $\psi_{F(T)} \oplus \form{f_1',\hdots,f_{\mathfrak{i}_1-1}'}$ (again, see Lemma \ref{Lemexistenceofforms}), we can also arrange it so that $g_{i,0} = 0$ for all $i$. The remaining $g_{i,j}$ then satisfy the statement of the lemma.  \end{proof} \end{lemma}

Let us now fix elements $g_{i,j} \in D(\psi_{F[T]})$ (and the associated polynomials $f_i$) satisfying the statement of Lemma \ref{LemMTL1}. We are searching here for a sufficiently large quasi-Pfister divisor of $\phi_1$, and we would like to try to build this quasi-Pfister $p$-form from the elements $g_{i,1}$. The basic point here is the following:

\begin{lemma} \label{LemMTL2} In the above situation, we have $\overline{g_{i,1}}b \in D(\phi_1)$ for all $b \in D(\psi)$ and all $1 \leq i < \mathfrak{i}_1$.
\begin{proof} If $b = 0$, then the statement is trivial. Let us now fix $b \in D(\psi) \setminus \lbrace 0 \rbrace$. We will need another lemma:

\begin{lemma} \label{LemMTL3} In the above situation, there exist elements $s_{i,j} \in D(\psi_{F[T]})$ and $t_{i,j} \in F[T]\setminus \lbrace 0 \rbrace\;(1 \leq i < \mathfrak{i}_1,\;0 \leq j < p)$ such that, for every $1 \leq i < \mathfrak{i}_1$, we have:
\begin{enumerate} \item $bf_i = \sum_{j=0}^{p-1}\frac{s_{i,j}}{t_{i,j}^p}\frac{\phi'(T)^j}{b^j}$ in $F(T)$.
\item For each $0 \leq j < p$, at least one of $\overline{s_{i,j}}$ and $\overline{t_{i,j}}$ is nonzero. \end{enumerate}
\begin{proof} Let $1 \leq i < \mathfrak{i}_1$, and consider the field $L = F(T)_{u}$, where $u = \frac{\phi'(T)}{b}$. In view of Remark \ref{Remsff} (1), $L$ is $F$-isomorphic to the affine function field $F[\eta]$, where $\eta$ denotes the $F$-form $\form{b} \perp \phi'$. Now, since $b \in D(\psi)$, and since $\psi \subset \phi$, we have $D(\phi') \subseteq D(\eta) \subseteq D(\phi)$. In particular, if $\eta \not\simeq \phi$, then it follows from Lemma \ref{Lemexistenceofforms} that $\anispart{\eta} \simeq \phi'$. Either way, we see that $L$ is $F$-isomorphic to a degree-1 purely transcendental extension of $F(\phi)$ -- in the first case, see Remark \ref{Remsff} (2); in the second, see Remark \ref{Remsff} (3) and Corollary \ref{Coraffinecodim1}. By Corollary \ref{Corfirstkernel} and Lemma \ref{Lemseparableextensions}, it follows that $\psi_L \simeq \anispart{(\phi_L)}$. In other words, we have $D(\phi_L) = D(\psi_L) = \sum_{j=0}^{p-1} D(\psi_{F(T)})u^j$ (again, see Lemma \ref{LemIsotropypurelyinseparable} (1) for the final equality). Now, since $u$ is a $p$-th power in $L$, we have $bf_i = \frac{\phi'(T)f_i}{u} \in D(\phi_L)$. We can therefore write $bf_i = \sum_{i=0}^{p-1} q_ju^j$ for some $q_j \in D(\psi_{F(T)})$. Since every element of $D(\psi_{F(T)})$ is the quotient of an element of $D(\psi_{F[T]})$ and a $p$-th power in $F[T]$, and since $u = \frac{\phi'(T)}{b}$, this shows that we can find elements $s_{i,j} \in D(\psi_{F[T]})$ and $t_{i,j} \in F[T]\setminus \lbrace 0 \rbrace$ such that (1) holds. Finally, since $\psi_{F(\phi)}$ is anisotropic, Proposition \ref{PropCPapp} implies that $m(s_{i,j}) \equiv 0 \pmod{p}$ for all $0 \leq j < p$. For each such $j$, Let $m_j = \mathrm{min}\big(m(s_{i,j}),pm(t_{i,j})\big)$, and put $s_{i,j}' = \frac{s_{i,j}}{\phi'(T)^{m_j}}$ and $t_{i,j}' = \frac{t_{i,j}}{\phi'(T)^{m_j/p}}$. Then, by Theorem \ref{ThmCasselsPfister}, we again have $s_{i,j}' \in D(\psi_{F[T]})$. Thus, replacing $s_{i,j}$ by $s_{i,j}'$ and $t_{i,j}$ by $t_{i,j}'$ (for each $j$), we arrive at the situation where, for any $j$, either $m(s_{i,j}) = 0$ or $m(t_{i,j}) = 0$. In other, words at least one of $\overline{s_{i,j}}$ and $\overline{t_{i,j}}$ is nonzero, as we wanted.
\end{proof} \end{lemma}

Returning now to the proof of Lemma \ref{LemMTL2}, let $s_{i,j} \in D(\psi_{F[T]})$ and $t_{i,j} \in F[T]\setminus \lbrace 0 \rbrace$ be as in Lemma \ref{LemMTL3}. In particular, we have the equation
\begin{equation*} \sum_{l=1}^{p-1}bg_{i,l}\phi'(T)^{l-1} =  bf_i = \sum_{j=0}^{p-1}\frac{s_{i,j}}{t_{i,j}^p}\frac{\phi'(T)^j}{b^j} \end{equation*}
in $F(T)$. Clearing denominators, we obtain
\begin{equation} \label{eq5.1}\prod_{k}t_{i,k}^p \sum_{l=1}^{p-1}bg_{i,l}\phi'(T)^{l-1} = \sum_{j=0}^{p-1}\prod_{k \neq j}t_{i,k}^p s_{i,j}\frac{\phi'(T)^j}{b^j} \end{equation}
Now, we claim that, for all $0 \leq j < p$, we have $\overline{t_{i,j}} \neq 0$, or, equivalently, $m(t_{i,j}) = 0$. To see this, let $m = \mathrm{min}\lbrace \sum_{k \neq j} m(t_{i,k})\;|\;0 \leq j < p \rbrace$. Then our claim amounts to the assertion that $m = \sum_{k=0}^{p-1}m(t_{i,k})$. Suppose that this is not the case, and let $0 \leq j < p$ be minimal so that $\sum_{k \neq j} m(t_{i,k}) = m$. Then, reducing both sides of \eqref{eq5.1} modulo $\phi'(T)^{pm+j+1}$, we see that $s_{i,j} \equiv 0 \pmod{\phi'(T)}$. In other words, we have $\overline{s_{i,j}} = 0$. By the choice of the $s_{i,j}$ and $t_{i,j}$, this implies that $\overline{t_{i,j}} \neq 0$, or, equivalently, that $m(t_{i,j}) = 0$. But then $m = \sum_{k \neq j}m(t_k) = \sum_{k=0}^{p-1}m(t_k)$, which contradicts our assumption. The claim is therefore proved, and so, reducing \eqref{eq5.1} modulo $\phi'(T)$ and dividing through by $\prod_k \overline{t_i}^p$, we obtain the equality $\overline{g_{i,1}}b = \overline{s_{i,0}}/\overline{t_{i,0}}^p$ in $F(\phi)$. As $s_{i,0} \in D(\psi_{F[T]})$, this shows that $\overline{g_{i,1}}b \in D(\psi_{F(\phi)})$. But since $\psi_{F(\phi)} \simeq \phi_1$, we have $D(\psi_{F(\phi)}) = D(\phi_1)$, and the lemma is therefore proved. \end{proof} \end{lemma}

Continuing with the proof of Theorem \ref{ThmMainthm}, let us now choose elements $g_{i,j} \in D(\psi_{F[T]})$ as in the statement of Lemma \ref{LemMTL1} so that the integer $\sum_{i=1}^{\mathfrak{i}_1-1}\mathrm{deg}_{T_1}(g_{i,1})$ is minimal (where for any $g \in F[T]$, $\mathrm{deg}_{T_1}(g)$ denotes the degree of $g$ viewed as an element of the ring $F(T_2,\hdots,T_n)[T_1]$, i.e., as a polynomial in the single variable $T_1$\footnote{With the added convention that $\mathrm{deg}_{T_1}(0) = 0$.}). Consider the form $\sigma = \form{1,\overline{g_{1,1}},\hdots,\overline{g_{\mathfrak{i}_1-1,1}}}$ over $F(\phi)$. The final step in the proof of the theorem will be to prove the following statement:

\begin{lemma} \label{LemMTL4} In the above situation, $\sigma$ is anisotropic. \end{lemma}

Before proving the lemma, let us explain how this concludes the proof of Theorem \ref{ThmMainthm}. First, we claim that $\phi_1$ is divisible by $\normform{\sigma}$. By Corollary \ref{CordivbyQP}, this amounts to checking that $D(\sigma) \subseteq G(\phi_1)$. Since $G(\phi_1)$ is a subfield of $F$ containing $F^p$ (Corollary \ref{Corsimilarityfield}), it suffices to show here that $\overline{g_{i,1}} \in G(\phi_1)$ for all $1 \leq i < \mathfrak{i}_1$. But by Lemma \ref{Lemsimilarityfactor}, this is equivalent to showing that, for all such $i$, we have $\overline{g_{i,1}}D(\phi_1) \subseteq D(\phi_1)$. Since $D(\phi_1) = D(\psi_{F(\phi)})$ is spanned as an $F(\phi)^p$-vector space by $D(\psi)$, this follows immediately from Lemma \ref{LemMTL2}. The claim is therefore proved, and to finish the proof of the theorem, it only remains to check that $\mydim{\normform{\sigma}} \geq p^s$. But, $\sigma$ is anisotropic by Lemma \ref{LemMTL4}, and so $\sigma \subset \normform{\sigma}$ by Lemma \ref{Lempropertyofnormform}. In particular, we have $\mydim{\normform{\sigma}} \geq \mydim{\sigma} = \mathfrak{i}_1$, which is precisely the assertion that $\mydim{\normform{\sigma}} \geq p^s$. Now, in order to prove Lemma \ref{LemMTL4}, we need another auxiliary statement:

\begin{lemma} \label{LemMTL5} If $\sigma$ is isotropic, then $p>2$, and there exist polynomials $g_{j} \in D(\psi_{F[T]})\;(1 \leq j <p)$ and an integer $2 \leq k < p$ such that 
\begin{enumerate} \item $\sum_{j=1}^{p-1}g_j\phi'(T)^j \in D(\phi_{F[T]})$.
\item $m(g_l) > 0$ for all $1 \leq l < k$.
\item $m(g_k) = 0$. \end{enumerate}
\begin{proof} If $\sigma$ is isotropic, then since $\phi'(T)$ is a Fermat-type polynomial of degree $p$, we can find an integer $1 \leq m \leq p$ and polynomials $h_0,\dots,h_{\mathfrak{i}_1-1},h \in F[T]$ such that:
\begin{enumerate} \item[(i)] $h_0^p + g_{1,1}h_1^p + \hdots + g_{\mathfrak{i}_1-1,1}h_{\mathfrak{i}_1-1}^p = \phi'(T)^mh$ in $F[T]$.
\item[(ii)] $\mathrm{deg}_{T_1}(h_i) < p$ for all $0 \leq i < \mathfrak{i}_1$. 
\item[(iii)] $\overline{h_i} \neq 0$ for some $1 \leq i < \mathfrak{i}_1$. \end{enumerate}
First, let us note that we have $\phi'(T)^m h \in D(\psi_{F[T]})$ by (i) and the definition of the elements $g_{i,1}$. Since $\psi_{F(\phi)}$ is anisotropic, it follows from Proposition \ref{PropCPapp} that $h=0$ or $m=p$. Either way, we can assume henceforth that $m=p$. Now, by (iii), there exists $1 \leq l < \mathfrak{i}_1$ such that $h_l \neq 0$. Among all such integers $l$, let us fix one so that $\mathrm{deg}_{T_1}(g_{l,1}h_l^p)$ is maximal. Consider now the polynomial $f = h_0^p + \sum_{k=1}^{\mathfrak{i}_1} f_kh_k^p \in F[T]$, where the $f_k$ are as in the statement of Lemma \ref{LemMTL1}. Since $h_l \neq 0$, Lemma \ref{Lemexistenceofforms} implies that $\form{1,f_1,\hdots,f_{\mathfrak{i}_1-1}} \simeq \form{1,f_1,\hdots,f_{l-1},f,f_{l+1},\hdots,f_{\mathfrak{i}_1-1}}$ as $F(T)$-forms. In particular, we have 
\begin{equation} \label{eq5.2} \phi_{F(T)} \simeq \psi_{F(T)} \oplus \phi'(T)\form{1,f_1,\hdots,f_{l-1},f,f_{l+1},\hdots,f_{\mathfrak{i}_1 -1}}. \end{equation}
Now, by definition, $f = \sum_{j=1}^{p-1}g_j'\phi'(T)^{j-1}$, where $g_1' = \phi'(T)^mh \in D(\psi_{F[T]})$ and $g_j' = \sum_{k=1}^{\mathfrak{i}_1-1}g_{k,j}h_k^p \in D(\psi_{F[T]})$ for all $2 \leq i < p$. Let $r = \mathrm{min}\lbrace m(g_j')\;|\;1 \leq j < p \rbrace$. Since $g_j' \in D(\psi_{F[T]})$ for all $j$, and since $\psi_{F(\phi)}$ is anisotropic, another application of Proposition \ref{PropCPapp} shows that $r \equiv 0 \pmod{p}$. In particular, for each $j \geq 1$, we have $g_j \coloneqq \frac{g_j'}{\phi'(T)^r} \in D(\psi_{F[T]})$. In view of \eqref{eq5.2}, it follows that the exchange $g_{l,j} \rightarrow g_j$ does not alter the statement of Lemma \ref{LemMTL1}. By our choice of the $g_{i,j}$, we therefore have
\begin{equation} \label{eq5.3} \mathrm{deg}_{T_1}(\phi'(T)^{p-r}h) = \mathrm{deg}_{T_1}(g_1) \geq \mathrm{deg}_{T_1}(g_{l,1}). \end{equation}
Now, we claim that the elements $g_j$ (together with an appropriate integer $k$) satisfy the conditions of the lemma. We have already seen here the validity of (1). At the same time, we have $m(g_j) = 0$ for some $j \geq 1$ by construction. Thus, in order to prove the existence of an integer $k$ such that (2) and (3) are satisfied, we just need to check that $m(g_1) > 0$. Recall again that we have $g_1 = \phi'(T)^{p-r}h$. If $h=0$, then there is nothing to prove. Suppose now that $h \neq 0$. By (i) and the choice of the integer $l$, we have $\mathrm{deg}_{T_1}(g_{l,1}h_l^p) \geq \mathrm{deg}_{T_1}(\phi'(T)^ph) \geq p^2 + \mathrm{deg}_{T_1}(h)$. Since $\mathrm{deg}_{T_1}(h_l) < p$ (by (ii)), it follows that $\mathrm{deg}_{T_1}(g_{l,1}) > \mathrm{deg}_{T_1}(h)$. In view of \eqref{eq5.3}, we see that $r = 0$ in this case. In particular, we have $g_1 = \phi'(T)^ph$, and so $m(g_1) \geq p >0$, as we wanted.\end{proof} \end{lemma}

We are now ready to prove Lemma \ref{LemMTL4} and thus complete the proof of Theorem \ref{ThmMainthm}. If $p=2$, then the statement was already proved in Lemma \ref{LemMTL5}. Suppose now that $p>2$, and assume for the sake of contradiction that $\sigma$ is isotropic. Let $g_j$ ($1 \leq j < p)$ and $k$ be as in the statement of Lemma \ref{LemMTL5}. By condition (2) of the lemma, we can, for each $l < k$, write $g_l = \phi'(T)^{m_l}h_l$ for some positive integer $m_l$ and some polynomial $h_l$. By a now familiar application of Proposition \ref{PropCPapp}, we have $m_l \equiv 0 \pmod{p}$ for every such $l$. In particular, the $m_l$ are all strictly larger than $k$. Now, by condition (1) of the lemma, the element
\begin{equation*} \phi'(T)^k\big(g_k + g_{k-1}\phi'(T) + \hdots + g_{\mathfrak{i}_1-1}\phi'(T)^{p-k-1} + h_1\phi'(T)^{m_1 - k} + \hdots + h_{k-1}\phi'(T)^{m_{k-1} - k}\big) \end{equation*}
lies in $D(\phi_{F[T]})$. Using the very same argument as that used to prove Lemma \ref{LemMTL2} above (and the fact that the integers $m_l - k$ ($l<k$) are all positive), one readily shows that $b^k\overline{g_k} \in D(\phi_1)$ for every $b \in D(\psi)$. Note, however, that $\overline{g_k} \neq 0$ by condition (3) of Lemma \ref{LemMTL5}. Since $2 \leq k < p$, and since $\phi_1 \simeq \psi_{F(\phi)}$, it follows from Lemma \ref{LemQPcriterion} that $\phi_1$ is a quasi-Pfister $p$-form. But, by Corollary \ref{CorClassificationQPN}, this in turn implies that $\phi$ is a quasi-Pfister $p$-neighbour, thus contradicting our original hypothesis. The lemma and theorem are therefore proved. \end{proof}

\section{Applications of the main theorem} \label{Applications}

We now give the basic applications of Theorem \ref{Maintheorem}.

\subsection{Possible values of the Knebusch splitting pattern} \label{Possiblevalues} Let $\phi$ be a quasilinear $p$-form of dimension $\geq 2$ over $F$. In the previous section we have shown that the first higher anisotropic kernel $\phi_1$ of $\phi$ is divisible by a quasi-Pfister $p$-form of dimension $\geq \witti{1}{\phi}$, provided that $\anispart{\phi}$ is not a quasi-Pfister $p$-neighbour. In the terminology of \S \ref{Higherdivindices}, this amounts to the assertion that if $\qp{h}(\phi) \geq 2$, then $\mathfrak{d}_1(\phi) \geq \mathrm{log}_p\big(\witti{1}{\phi}\big)$ (here we are also making use Corollary \ref{CorClassificationQPN}). By virtue of the inductive nature of the Knebusch splitting tower construction, we also obtain analogous restrictions on the higher isotropy indices $\witti{r}{\phi}$ ($2 \leq r < \qp{h}(\phi)$) in terms of the corresponding higher divisibility indices $\mathfrak{d}_r(\phi)$. Taking the observations of \S \ref{Higherdivindices} into account, our results may be summarised as follows:

\begin{theorem} \label{Thmallvalues} Let $\phi$ be a quasilinear $p$-form over $F$ and let $d=h(\phi) - \qp{h}(\phi)$. Then:
\begin{enumerate} \item $\mathfrak{i}(\phi) = (\witti{1}{\phi},\hdots,\witti{\qp{h}(\phi)}{\phi},p^d -p^{d-1},p^{d-1}-p^{d-2},\hdots,p^2-p,p-1)$.
\item $\witti{\qp{h}}{\phi} = \mydim{\phi} - \wittj{\qp{h}(\phi)-1}{\phi} - p^d < p^{d+1}-p^d$.
\item $\mathfrak{d}(\phi) = (\mathfrak{d}_0(\phi),\hdots,\mathfrak{d}_{\qp{h}(\phi)-1}(\phi),d,d - 1,\hdots,1,0)$.
\item $\mathfrak{d}_0(\phi) \leq \mathfrak{d}_1(\phi) \leq \hdots \leq \mathfrak{d}_{\qp{h}(\phi)}=d$.
\item $\witti{r}{\phi} \equiv 0 \pmod{p^{\mathfrak{d}_{r-1}(\phi)}}$ for all $1 \leq r < \qp{h}(\phi)$. 
\item $\mathfrak{d}_r(\phi) \geq \mathrm{log}_p\big(\witti{r}{\phi}\big)$ for all $1 \leq r < \qp{h}(\phi)$.
\item For every $1 \leq r < \qp{h}(\phi)$, $\witti{r}{\phi} - 1$ is the remainder of $\mydim{\phi} - \wittj{r-1}{\phi} - 1$ modulo $p^{\mathfrak{d}_r(\phi)}$. \end{enumerate}
\begin{proof} Parts (1), (2), (3), (4) and (5) are the statements comprising Lemmas \ref{Lemtruncationsplitting}, \ref{Lemtruncationdivisibility}, \ref{Lemincreasingdivisibility} and Corollary \ref{Cordivisorsofindices}. Since $\witti{r}{\phi} = \witti{1}{\phi_{r-1}}$, $\mathfrak{d}_r(\phi) = \mathfrak{d}_1(\phi_{r-1})$ and $\mydim{\phi_{r-1}} = \mydim{\phi} - \wittj{r-1}{\phi}$ for all $1 \leq r \leq h(\phi)$, parts (6) and (7) follow immediately from Theorem \ref{ThmMainthm} and Corollary \ref{CorClassificationQPN}. \end{proof} \end{theorem}

In order to highlight the general shape of the Knebusch splitting pattern exposed by Theorem \ref{Thmallvalues}, it is worth writing down the following result explicitly (here the notation $a \mid b$ means that $a$ divides $b$):

\begin{corollary} \label{Corincreasing} Let $\phi$ be a quasilinear $p$-form over $F$. Then $\witti{1}{\phi} \leq p^{\mathfrak{d}_1(\phi)} \mid \witti{2}{\phi} \leq p^{\mathfrak{d}_2(\phi)} \mid \hdots \mid \witti{\qp{h}(\phi)-1}{\phi} \leq p^{\mathfrak{d}_{\qp{h}(\phi)-1}(\phi)} \mid \witti{\qp{h}(\phi)}{\phi}$. \end{corollary}

\begin{remark} In the special case where $p=2$, the chain of inequalities $\witti{1}{\phi} \leq \witti{2}{\phi} \leq \hdots \leq \witti{\qp{h}(\phi)}{\phi}$ was previously obtained in \cite[Thm. 9.5]{Scully2} using Proposition \ref{Propcomparison}. Here, we have given a more precise and natural explanation of this phenomenon. \end{remark}

We now show that, as far as the Knebusch splitting pattern is concerned, one cannot do any better than Theorem \ref{Thmallvalues} in general:

\begin{proposition} \label{Propexistenceofvalues} Let $n$ be any positive integer. Suppose that we are given a nonnegative integer $k \leq n$ and two sequences $(\mathfrak{d}_0,\mathfrak{d}_1,\hdots,\mathfrak{d}_k=d)$ and $(\mathfrak{i}_1,\hdots,\mathfrak{i}_k)$ of $k+1$ and $k$ nonnegative integers, respectively, such that the following conditions hold:
\begin{enumerate} \item[$\mathrm{(i)}$] $\mathfrak{i}_k = n - \sum_{j=1}^{k-1}\mathfrak{i}_j - p^{d} < p^{d + 1} - p^{d}$.
\item[$\mathrm{(ii)}$] $\mathfrak{d}_0 \leq \mathfrak{d}_1 \leq \hdots \leq \mathfrak{d}_k=d$.
\item[$\mathrm{(iii)}$] $1 \leq \mathfrak{i}_r \equiv 0 \pmod{p^{\mathfrak{d}_{r-1}}}$ for all $1 \leq r < k$.
\item[$\mathrm{(iv)}$] $\mathfrak{d}_r \geq \mathrm{log}_p(\mathfrak{i}_r)$ for all $1 \leq r < k$. 
\item[$\mathrm{(v)}$] For every $1 \leq r < k$, $\mathfrak{i}_r - 1$ is the remainder of $n - (\sum_{j=1}^{r-1} \mathfrak{i}_j) - 1$ modulo $p^{\mathfrak{d}_r}$. \end{enumerate}
Then there exists a (purely transcendental) field extension $L$ of $F$ and an anisotropic quasilinear $p$-form $\phi$ of dimension $n$ over $L$ such that:
\begin{enumerate} \item $\qp{h}(\phi) = k$.
\item $h(\phi) = k+d$.
\item $\mathfrak{d}(\phi) = (\mathfrak{d}_0(\phi),\hdots,\mathfrak{d}_{k-1}(\phi),d,d-1,\hdots,1,0)$.
\item $\mathfrak{d}_r(\phi) \geq \mathfrak{d}_r$ for all $0 \leq r < k$.
\item $\mathfrak{i}(\phi) = (\mathfrak{i}_1,\hdots,\mathfrak{i}_k,p^d - p^{d-1},p^{d-1} - p^{d-2},\hdots,p^2 - p,p-1)$. \end{enumerate}
\begin{proof} We argue by induction on $k$. If $k=0$, then Proposition \ref{Propi1bound} shows that we can take $L = F(T)$ and $\phi = \pfister{T_1,\hdots,T_d}$, where $T = (T_1,\hdots,T_{d})$ is a $d$-tuple of algebraically independent variables over $F$. Suppose now that $k>0$, and let $n' = \frac{n - \mathfrak{i}_1}{p^{\mathfrak{d}_1}}$ and $\mathfrak{i}_r' = \frac{\mathfrak{i}_{r+1}}{p^{\mathfrak{d}_1}}$ for all $1 \leq r < k$. By our hypotheses, these ratios are, in fact, positive integers. Setting $\mathfrak{d}_r' = \mathfrak{d}_{r+1} - \mathfrak{d}_1$ for all $0 \leq r < k$, and putting $d' = \mathfrak{d}_{k-1}'$, conditions (i)-(v) then imply the following:
\begin{enumerate} \item[$\mathrm{(i')}$] $\mathfrak{i}'_{k-1} = n' - \sum_{j=1}^{k-2}\mathfrak{i}_j' - p^{d'} < p^{d'+1} - p^{d'}$.
\item[$\mathrm{(ii')}$] $\mathfrak{d}_0' \leq \mathfrak{d}_1' \leq \hdots \leq \mathfrak{d}_{k-1}' = d'$.
\item[$\mathrm{(iii')}$] $1 \leq \mathfrak{i}_r' \equiv 0 \pmod{p^{\mathfrak{d}_{r-1}'}}$ for all $1 \leq r < k-1$.
\item[$\mathrm{(iv')}$] $\mathfrak{d}_r' \geq \mathrm{log}_p(\mathfrak{i}_r')$ for all $1 \leq r < k-1$.
\item[$\mathrm{(v')}$] For every $1 \leq r < k-1$, $\mathfrak{i}_r' - 1$ is the remainder of $n' - (\sum_{j=1}^{r-1} \mathfrak{i}_j') - 1$ modulo $p^{\mathfrak{d}_r'}$. \end{enumerate}
By the induction hypothesis, there exists a (purely transcendental) field extension $L_0$ of $F$ and an anisotropic quasilinear $p$-form $\psi$ of dimension $n'$ over $L_0$ such that
\begin{enumerate} \item[$\mathrm{(1')}$] $\qp{h}(\psi) = k-1$.
\item[$\mathrm{(2')}$] $h(\psi) = k - 1 +d'$.
\item[$\mathrm{(3')}$] $\mathfrak{d}(\psi) = (\mathfrak{d}_0(\psi),\mathfrak{d}_1(\psi),\hdots,\mathfrak{d}_{k-2}(\psi),d',d' - 1,\hdots,1,0)$.
\item[$\mathrm{(4')}$] $\mathfrak{d}_r(\psi) \geq \mathfrak{d}_r'$ for all $0 \leq r <k-1$.
\item[$\mathrm{(5')}$] $\mathfrak{i}(\psi) = (\mathfrak{i}_1',\hdots,\mathfrak{i}_{k-1}',p^{d'} - p^{d'-1},p^{d'-1} - p^{d'-2},\hdots,p^2-p,p-1)$. \end{enumerate}
Consider now the form $\sigma = \psi_{L_1} \perp \form{T_0}$ over the rational function field $L_1 = L_0(T_0)$. By $(3')$, $(5')$ and Lemma \ref{Lemaddavariable}, we have
\begin{enumerate} \item[$\mathrm{(a)}$] $\mathfrak{d}(\sigma) = (0,\mathfrak{d}_0(\psi),\mathfrak{d}_1(\psi),\hdots,\mathfrak{d}_{k-2}(\psi),d',d'-1,\hdots,1,0)$.
\item[$\mathrm{(b)}$] $\mathfrak{i}(\sigma) = (1,\mathfrak{i}_1',\hdots,\mathfrak{i}'_{k-1},p^{d'} - p^{d'-1},p^{d'-1} - p^{d'-2},\hdots,p^2-p,p-1)$. \end{enumerate}
We would like to modify this further. Consider next the product $\tau = \pfister{T_1,\hdots,T_{\mathfrak{d}_1}} \otimes \sigma_L$ over $L = L_1(T)$, where $T = (T_1,\hdots,T_{\mathfrak{d}_1})$ is a $\mathfrak{d}_1$-tuple of algebraically independent variables over $L_1$. Then, by (a), (b) and Lemma \ref{LemgenericQPmultiple}, we have
\begin{enumerate} \item[$\mathrm{(c)}$] $\mathfrak{d}(\tau) = (\mathfrak{d}_1, \mathfrak{d}_1(\psi) + \mathfrak{d}_1,\mathfrak{d}_2(\psi) + \mathfrak{d}_1,\hdots,\mathfrak{d}_{k-2}(\psi) + d_1,d,d -1,\hdots,1,0)$.
\item[$\mathrm{(d)}$] $\mathfrak{i}(\tau) = (p^{\mathfrak{d}_1},\mathfrak{i}_2,\hdots,\mathfrak{i}_k,p^d - p^{d-1},p^{d-1} - p^{d-2},\hdots,p^2-p,p-1)$. \end{enumerate}
Now, by (iv), we have $\mathfrak{i}_1 = p^{\mathfrak{d}_1} - s$ for some $0 \leq s < p^{\mathfrak{d}_1}$. By (ii) and (iii), $s$ is divisible by $p^{\mathfrak{d}_0}$. Let $\phi$ be any codimension-$s$ subform of $\tau$ which is divisible by $\pfister{T_1,\hdots,T_{\mathfrak{d}_0}}$. Clearly $\phi$ is anisotropic, and by (c), (d) and Proposition \ref{Propsspneighbours}, we have
\begin{enumerate} \item[$\mathrm{(e)}$] $\mathfrak{d}(\phi) = (\mathfrak{d}_0(\phi),\mathfrak{d}_1(\psi) + \mathfrak{d}_1,\hdots,\mathfrak{d}_{k-2}(\psi) + \mathfrak{d}_1,d,d -1,\hdots,1,0)$.
\item[$\mathrm{(f)}$] $\mathfrak{i}(\phi) = (\mathfrak{i}_1,\mathfrak{i}_2,\hdots,\mathfrak{i}_k,p^d-p^{d-1},p^{d-1} - p^{d-2},\hdots,p^2-p,p-1)$. \end{enumerate}
The second statement shows that $\phi$ satisfies conditions (2) and (5). At the same time, since $\mathfrak{d}_0(\phi) \geq \mathfrak{d}_0$ by construction, and since $\mathfrak{d}_2(\psi) + \mathfrak{d}_1 \geq \mathfrak{d}_r' + \mathfrak{d}_1 = \mathfrak{d}_{r+1}$ for all $0 \leq r < k-1$ by $(4')$, (e) shows that (3) and (4) are also satisfied. Finally, since $\mathfrak{i}_k < p^{d+1}-p^d$, Proposition \ref{Propi1bound} shows that $\qp{h}(\phi) = k$, i.e., that (1) holds for $\phi$. The pair $(L,\phi)$ therefore has all the desired properties. \end{proof} \end{proposition}

\begin{remark} In general, it is not possible to arrange it so that $\mathfrak{d}_r(\phi) = \mathfrak{d}_r$ for all $0 \leq r < k$ in the statement of Proposition \ref{Propexistenceofvalues}. For example, suppose that $p=2$, and take $n = 2^{s+1} - 2$ for some $s \geq 3$, $k = 1$, $\mathfrak{d}_0 = 0$, $\mathfrak{d}_1 = s$, $\mathfrak{i}_0 = 0$, $\mathfrak{i}_1 = 2^s - 2$. As the reader will readily verify, these integers satisfy conditions (i)-(v) of the Proposition. On the other hand, let $(L,\phi)$ be any pair consisting of a field extension $L$ of $F$ and an anisotropic form $\phi$ of dimension $2^{s+1} - 2$ over $L$ such that $\witti{1}{\phi} = 2^s -2$. By Theorem \ref{Thmmaxsplitting} below (see also \cite[Thm. 9.6]{Scully2}), $\phi$ is necessarily a quasi-Pfister 2-neighbour, and therefore satisfies conditions (1)-(5) of the proposition (see Corollary \ref{CorClassificationQPN}). We claim, however, that $\mathfrak{d}_0(\phi) > 0$, i.e., that $\phi$ is divisible by a binary form. Indeed, note that, since $\phi$ is a quasi-Pfister 2-neighbour, there exists an anisotropic $(s+1)$-fold bilinear Pfister form $\mathfrak{c}$ over $L$ and a codimension-2 subform $\mathfrak{b} \subset \mathfrak{c}$ such that $\phi \simeq \phi_{\mathfrak{b}}$. Since $\mathrm{det}(\mathfrak{c}) = 1$, we have $\mathfrak{c} \simeq \mathfrak{b} \perp \lambda \pfister{\mathrm{det}(\mathfrak{b})}_b$ for some $\lambda \in L^*$. In particular, we have $\mathfrak{b}_{L_{\mathrm{det}(\mathfrak{b})}} \sim 0$. By Corollary \ref{CorBilindex}, it follows that $\witti{0}{\phi_{L_{\mathrm{det}(\mathfrak{b})}}} \geq \frac{1}{2}\mydim{\phi}$, and so $\phi$ is divisible by $\pfister{\mathrm{det}(\mathfrak{b})}$ (see Lemma \ref{LemIsotropypurelyinseparable} (6)). \end{remark}

Theorem \ref{Thmallvalues} and Proposition \ref{Propexistenceofvalues} thus give a complete solution to the problem of determining the possible values of the Knebusch splitting pattern for quasilinear $p$-forms. In particular, we have an answer to Question \ref{Quevalues} in the totally singular case. As noted in Example \ref{Exnongeneric}, the Knebusch and full splitting patterns need not agree in general for quasilinear $p$-forms. In \S \ref{Excellentconnections} below, we will consider the problem of determining the possible values of the full splitting pattern in the case where $p=2$.

\subsection{Quasilinear $p$-forms with maximal splitting} Let $\phi$ be an anisotropic quasilinear $p$-form of dimension $\geq 2$ over $F$ and write $\mydim{\phi} = p^n + m$ for uniquely determined integers $n \geq 0$ and $1 \leq m \leq p^{n+1} - p^n$. By Theorem \ref{Thmallvalues} (see also \cite[Cor. 6.8]{Scully1}), we have $\witti{1}{\phi} \leq m$. If equality holds here, then we say that $\phi$ has \emph{maximal splitting}. The basic examples of forms having this property are given by anisotropic quasi-Pfister $p$-neighbours (see Corollary \ref{CorClassificationQPN}). It is interesting to ask here to what extent this property characterises quasi-Pfister $p$-neighbours. Given Theorem \ref{ThmMainthm}, we can now prove the following general result:

\begin{theorem} \label{Thmmaxsplitting} Let $\phi$ be an anisotropic quasilinear $p$-form of dimension $\geq 2$ over $F$, and $n$ be the smallest nonnegative integer such that $p^{n+1} \geq \mydim{\phi}$. If $\phi$ has maximal splitting, and if either
\begin{enumerate} \item $p>2$ and $\mydim{\phi}> p^n + p^{n-1}$, or
\item $p=2$ and $\mydim{\phi} > 2^n + 2^{n-2}$, \end{enumerate}
then $\phi$ is a quasi-Pfister $p$-neighbour.
\begin{proof} By Theorem \ref{ThmMainthm}, we may assume that $\mathfrak{d}_1(\phi) \geq \witti{1}{\phi}$. Suppose first that $p>2$. Since $\phi$ has maximal splitting, we have $\witti{1}{\phi} > p^{n-1}$ by (1), and so $\mathfrak{d}_1(\phi) \geq p^n$. On the other hand, we have $\mydim{\phi_1} = \mydim{\phi} - \witti{1}{\phi} = p^n$ (see Remark \ref{RemsKnebusch} (2)). It follows that $\phi_1$ is similar to an $n$-fold quasi-Pfister $p$-form, and so $\phi$ is a quasi-Pfister $p$-neighbour by Corollary \ref{CorClassificationQPN}. If $p=2$, the same argument (and (2)) shows that $\phi_1$ is a form of dimension $2^n$ which is divisible by an $(n-1)$-fold quasi-Pfister 2-form. Since every binary form is similar to a quasi-Pfister 2-form in this case, $\phi_1$ is, in fact, similar to an $n$-fold quasi-Pfister 2-form, and we now conclude as before.   \end{proof} \end{theorem}

\begin{remark} The statement of Theorem \ref{Thmmaxsplitting} was originally conjectured by Hoffmann in \cite[Rem. 7.32]{Hoffmann2} (see also \cite[Que. 7.6]{Scully2}). The result is the best possible, in the sense that one has examples of anisotropic quasilinear $p$-forms with maximal splitting which are not quasi-Pfister $p$-neighbours in every dimension omitted in the statement of the theorem (see \cite[Ex. 7.31]{Hoffmann2}). In the special case where $p=2$, Theorem \ref{Thmmaxsplitting} was previously established in \cite[Thm. 9.6]{Scully2} using Proposition \ref{Propcomparison}. Here, we obtain a more natural explanation of this phenomenon by way of Theorem \ref{ThmMainthm}. Note that the $p=2$ case of the theorem is a direct analogue of a conjecture of Hoffmann in the theory of nonsingular quadratic forms which remains open, even over fields of characteristic different from 2 (see \cite[\S 4]{Hoffmann1}, \cite[Conj. 1.6]{IzhboldinVishik}). \end{remark}

\section{Some remarks on the splitting of symmetric bilinear forms in characteristic 2} \label{Bilinearforms}

We now provide a brief and relatively informal discussion on some implications of Theorem \ref{ThmMainthm} for the Knebusch splitting theory of symmetric bilinear forms in characteristic 2. The basic details of this theory, which was initiated by Laghribi in \cite{Laghribi6}, are outlined in \S \ref{SSPbilinear} below, and we refer there (and to \S \ref{BilBasicnotions}) for the relevant notation and terminology. Throughout this section, we assume that $p=2$. By a \emph{quasi-Pfister form} (resp. \emph{neighbour}), we will mean a quasi-Pfister 2-form (resp. 2-neighbour).

Let $\mathfrak{b}$ be an anisotropic bilinear form over $F$. By Lemma \ref{Lembilheight}, we have $h(\mathfrak{b}) \leq h(\phi_{\mathfrak{b}})$. Moreover, if $h(\mathfrak{b}) > 0$, then Lemma \ref{Lembilineari1} shows that $\frac{1}{2}\witti{1}{\phi_{\mathfrak{b}}} \leq \witti{1}{\mathfrak{b}} \leq \witti{1}{\phi_{\mathfrak{b}}}$, where $\phi_{\mathfrak{b}}$ denotes the quasilinear quadratic form on $V_\mathfrak{b}$ defined by the assignment $v \mapsto \mathfrak{b}(v,v)$. As per Remark \ref{Remunclear}, the precise nature of the relationship between $\witti{1}{\mathfrak{b}}$ and $\witti{1}{\phi_{\mathfrak{b}}}$ seems to be unknown, but one can easily produce examples which show that the given upper and lower bound on $\witti{1}{\mathfrak{b}}$ can both be met in practice. This suggests that prior knowledge of the Knebusch splitting pattern of $\phi_{\mathfrak{b}}$ provides only limited information concerning that of $\mathfrak{b}$. With the results of the previous section in hand, however, we can now clarify the situation a little further:

\begin{proposition} \label{Propbilineari1i2comparison} Let $\mathfrak{b}$ be an anisotropic bilinear form over $F$ such that $h(\mathfrak{b}) >1$ and write $\mydim{\mathfrak{b}} = 2^n + m$ for uniquely determined integers $n \geq 1$ and $1 \leq m \leq 2^n$. Then $h(\phi_{\mathfrak{b}}) > 1$, and exactly one of the following holds:
\begin{enumerate} \item $2\witti{1}{\mathfrak{b}} < \wittj{2}{\phi_{\mathfrak{b}}}$.
\item $\witti{1}{\mathfrak{b}} = \witti{1}{\phi_{\mathfrak{b}}} = \witti{2}{\phi_{\mathfrak{b}}} = 2^{\mathfrak{d}_1(\phi_\mathfrak{b})}$.
\item $\phi_{\mathfrak{b}}$ is a quasi-Pfister neighbour, $m > 2^{n-1}$ and $2^{n-2} + \frac{m}{2} \leq \witti{1}{\mathfrak{b}} \leq m$. \end{enumerate}
\begin{proof} As mentioned above, the inequality $h(\phi_{\mathfrak{b}}) > 1$ holds by Lemma \ref{Lembilheight}. Suppose now that $2\witti{1}{\mathfrak{b}} \geq \wittj{2}{\phi_{\mathfrak{b}}}$. If $\witti{2}{\phi_{\mathfrak{b}}} \geq \witti{1}{\phi_{\mathfrak{b}}}$, then it follows from Lemma \ref{Lembilineari1} that $\witti{1}{\mathfrak{b}} = \witti{1}{\phi_{\mathfrak{b}}} = \witti{2}{\phi_{\mathfrak{b}}}$. By Corollary \ref{Corincreasing}, we are therefore in case (2). On the other hand, if $\witti{2}{\phi_{\mathfrak{b}}} < \witti{1}{\phi_{\mathfrak{b}}}$, then the same corollary implies that $\qp{h}(\phi_{\mathfrak{b}}) \leq 1$. By Corollary \ref{CorClassificationQPN}, this implies that $\phi_{\mathfrak{b}}$ is a quasi-Pfister neighbour and that $\witti{1}{\phi_{\mathfrak{b}}} = m$ and $\witti{2}{\phi_{\mathfrak{b}}} = 2^{n-1}$. We are therefore in case (3), and so the result follows. \end{proof} \end{proposition}

This yields the following:

\begin{corollary} \label{Corbilinearpattern} Let $\mathfrak{b}$ be an anisotropic bilinear form over $F$ such that $h(\mathfrak{b}) > 1$ and write $\mydim{\mathfrak{b}} = 2^n + m$ for uniquely determined integers $n \geq 1$ and $1 \leq m \leq 2^n$. Then $h(\phi_{\mathfrak{b}}) > 1$, and exactly one of the following holds:
\begin{enumerate} \item $\phi_{\mathfrak{b}_1}$ is a neighbour of $(\phi_{\mathfrak{b}})_1$.
\item $\witti{1}{\mathfrak{b}} = \witti{1}{\phi_{\mathfrak{b}}} = \witti{2}{\phi_{\mathfrak{b}}} = 2^{\mathfrak{d}_1(\phi_\mathfrak{b})}$ and $\phi_{\mathfrak{b}_1}$ is a near neighbour of $(\phi_{\mathfrak{b}})_1$.
\item $\phi_\mathfrak{b}$ is a quasi-Pfister neighbour, $m > 2^{n-1}$ and $2^{n-2} + \frac{m}{2} \leq \witti{1}{\mathfrak{b}} \leq m$.  \end{enumerate}
\begin{proof} We apply Proposition \ref{Propbilineari1i2comparison} to the form $\mathfrak{b}$. By Proposition \ref{PropBildecomposition} (2), $\phi_{\mathfrak{b}_1}$ is a codimension-$d$ subform of $(\phi_{\mathfrak{b}})_1$, where $d = 2\witti{1}{\mathfrak{b}} - \witti{1}{\phi_{\mathfrak{b}}}$. In particular, if $d < \witti{2}{\phi_{\mathfrak{b}}}$, then $\phi_{\mathfrak{b}_1}$ is a neighbour of $(\phi_{\mathfrak{b}})_1$. In other words, if $2\witti{1}{\mathfrak{b}} < \wittj{2}{\phi_{\mathfrak{b}}}$, then (1) holds. On the other hand, if $2\witti{1}{\mathfrak{b}} \geq \wittj{2}{\phi_{\mathfrak{b}}}$, then Proposition \ref{Propbilineari1i2comparison} implies that either (3) holds, or else $\witti{1}{\mathfrak{b}} = \witti{1}{\phi_{\mathfrak{b}}} = \witti{2}{\phi_{\mathfrak{b}}} = 2^{d_1(\phi_\mathfrak{b})}$. In the latter situation, we have $\mydim{\mathfrak{b}_1} = \mydim{\mathfrak{b}} - 2\witti{1}{\mathfrak{b}} = \mydim{\phi_{\mathfrak{b}}} - \wittj{2}{\phi_{\mathfrak{b}}} = \mydim{(\phi_{\mathfrak{b}})_1} - \witti{1}{(\phi_{\mathfrak{b}})_1}$, so that $\phi_{\mathfrak{b}_1}$ is a near neighbour of $(\phi_{\mathfrak{b}})_1$. We are therefore in case (2), whence the result. \end{proof} \end{corollary}

In the statement of Corollary \ref{Corbilinearpattern}, (3) may be regarded as an extreme case. The situation is thus as follows: We know that $\frac{1}{2}\witti{1}{\phi_{\mathfrak{b}}} \leq \witti{1}{\mathfrak{b}} \leq \witti{1}{\phi_{\mathfrak{b}}}$. If $h(\mathfrak{b})>1$ and if we are \emph{not} in case (3), then $\mathfrak{b}_1$ is either a neighbour or a near neighbour of $(\phi_{\mathfrak{b}})_1$, the latter case occuring only when $\witti{1}{\mathfrak{b}} = \witti{1}{\phi_{\mathfrak{b}}} = \witti{2}{\phi_{\mathfrak{b}}} = 2^{\mathfrak{d}_1(\phi_\mathfrak{b})}$. Appealing to the relevant statement among those found in Proposition \ref{Propsspneighbours} and Corollary \ref{Corsspnearneighbours}, and applying Lemma \ref{Lembilineari1} to the form $\mathfrak{b}_1$, we then obtain restrictions on the invariant $\witti{2}{\mathfrak{b}}$ in terms of the sequence $\mathfrak{i}((\phi_{\mathfrak{b}})_1)$. Thus, repeating this process as many times as is possible, we can, in principle, obtain non-trivial information concerning the entire Knebusch splitting pattern of $\mathfrak{b}$ in terms of $\mathfrak{i}(\phi_{\mathfrak{b}})$. Of course, these ideas are mostly of interest in situations which are distant from ``generic''. We do not attempt to make this precise, but we give here an explicit example which serves to illustrate the broader philosophy. More specifically, we will examine some implications of the preceding discussion for the problem of classifying bilinear forms of height 2 in characteristic 2. In spite of the simple description of anisotropic bilinear forms of height 1 (Proposition \ref{Propheight1bilinear}), this problem remains wide open at present. In fact, the analogous problem in characteristic $\neq 2$ is also open. There, the expected classification in even dimensions is as follows (recall here that the terms ``quadratic form'' and ``symmetric bilinear form'' are synonymous in characteristic $\neq 2$):

\begin{conjecture}[{cf. \cite[Conj. 7]{Kahn1}}] \label{Conjheight2} Let $\phi$ be an anisotropic quadratic form of even dimension over a field of characteristic $\neq 2$ with $h(\phi) = 2$. If $\mathrm{deg}(\phi) = d$, then exactly one of the following holds:
\begin{enumerate} \item $\phi \simeq \pi \otimes \psi$ for some $(d-1)$-fold Pfister form $\pi$ and 4-dimensional form $\psi$ of nontrivial discriminant.
\item $\phi \simeq \pi \otimes \psi$ for some $(d-2)$-fold Pfister form $\pi$ and 6-dimensional form $\psi$ of trivial discriminant.
\item $\phi$ is an excellent form\footnote{That is, similar to $\pi \otimes \sigma'$, where $\pi$ is a $d$-fold Pfister form and $\sigma'$ is the pure part of some $(n-d)$-fold Pfister form $\sigma$; equivalently, $\phi_1$ is defined over the base field.} of dimension $2^{n+1} - 2^{d}$ for some $n > d$.\end{enumerate} \end{conjecture}

While the dimension-theoretic part of Conjecture \ref{Conjheight2} has been verified by Vishik (see \cite[Thm. 3.1]{Vishik3}), the general statement is only known in some special cases. Returning to characteristic 2, Corollary \ref{Corbilinearpattern} now yields the following:

\begin{corollary} \label{Corheight2even} Let $\mathfrak{b}$ be an anisotropic bilinear form of even dimension over $F$ such that $h(\mathfrak{b}) = 2$. If $\mathrm{deg}(b) = d$, then exactly one of the following holds:
\begin{enumerate} \item $\mydim{\mathfrak{b}} = 2^{d + 1}$, $\qp{h}(\phi_{\mathfrak{b}}) = 2$ and $\witti{1}{\phi_{\mathfrak{b}}} = \witti{2}{\phi_{\mathfrak{b}}} = 2^{d-1}$.
\item $d >1$, $\mydim{\mathfrak{b}} = 2^{d} + 2^{d-1} $, $\qp{h}(\phi_{\mathfrak{b}}) = 2$ and $\witti{1}{\phi_{\mathfrak{b}}} = \witti{2}{\phi_{\mathfrak{b}}} = 2^{d -2}$.
\item $\phi_\mathfrak{b}$ is a quasi-Pfister neighbour and $\mydim{\mathfrak{b}} \in \lbrace 2^{d+1} \rbrace \cup \lbrace 2^{d+1} -2^i\;|\;1 \leq i \leq 2^{d-1}\rbrace$.
\item $\phi_\mathfrak{b}$ is a quasi-Pfister neighbour and $2^{n+1} -2^d \leq \mydim{\mathfrak{b}} \leq 2^{n+1}$ for some $n > d$.
\end{enumerate}
\begin{proof} Since $h(\mathfrak{b}) = 2$, $\mathfrak{b}_1$ is a Pfister form of dimension $2^d$. By Theorem \ref{ThmVishikholes}, we either have $\mydim{\mathfrak{b}} \geq 2^{d+1}$ or $\mydim{\mathfrak{b}} = 2^{d+1}-2^i$ for some $1 \leq i \leq 2^{d-1}$. Now, write $\mydim{\mathfrak{b}} = 2^n + m$ for uniquely determined integers $n \geq 1$ and $1 \leq m \leq 2^n$. By Corollary \ref{Corbilinearpattern}, one of the following holds:
\begin{enumerate} \item[$\mathrm{(i)}$] $\phi_{\mathfrak{b}_1}$ is a neighbour of $(\phi_{\mathfrak{b}})_1$.
\item[$\mathrm{(ii)}$] $\witti{1}{\mathfrak{b}} = \witti{1}{\phi_{\mathfrak{b}}} = \witti{2}{\phi_{\mathfrak{b}}} = 2^{\mathfrak{d}_1(\phi_\mathfrak{b})}$ and $\phi_{\mathfrak{b}_1}$ is a near neighbour of $(\phi_{\mathfrak{b}})_1$.
\item[$\mathrm{(iii)}$] $\phi_\mathfrak{b}$ is a quasi-Pfister neighbour, $m > 2^{n-1}$ and $2^{n-2} + \frac{m}{2} \leq \witti{1}{\mathfrak{b}} \leq m$.\end{enumerate} 
If (i) holds, then we have $\mydim{(\phi_{\mathfrak{b}})_1} - \witti{1}{(\phi_{\mathfrak{b}})_1} < \mydim{\phi_{\mathfrak{b}_1}} \leq \mydim{(\phi_{\mathfrak{b}})_1}$. But, since $\mydim{\phi_{\mathfrak{b}_1}} = 2^{d}$, Theorem \ref{ThmHoffmannsconjecture} (more precisely, Hoffmann and Laghribi's upper bound on $\mathfrak{i}_1$ applied to the form $(\phi_{\mathfrak{b}})_1$) then implies that $d = n$ and $\phi_{\mathfrak{b}_1} \simeq (\phi_{\mathfrak{b}})_1$. In view of Corollary \ref{CorClassificationQPN}, we are therefore in case (3). If (ii) holds, then there are three possibilities: Either we are in case (1), case (2), or we have $\mydim{\mathfrak{b}} = 2^{n} + 2^{n-1}$, $d = n-1$ and $\witti{1}{\mathfrak{b}} = \witti{1}{\phi_{\mathfrak{b}}} = \witti{2}{\phi_{\mathfrak{b}}} = 2^{n-1}$. In the latter scenario, Corollary \ref{CorClassificationQPN} implies that $\phi_\mathfrak{b}$ is a quasi-Pfister neighbour, and so we are in case (4). Finally, if (iii) holds, then $\phi_\mathfrak{b}$ is a quasi-Pfister neighbour, $m>2^{n-1}$ and, since $2\witti{1}{\mathfrak{b}} = 2^n + m - 2^{d}$, we have $2^{n-1} \leq 2^n - 2^{d} \leq m$. We are thus again in case (4), and the result follows. \end{proof} \end{corollary}

The reader will immediately observe basic parallels between Corollary \ref{Corheight2even} and the expected classification of anisotropic quadratic forms of height 2 in characteristic $\neq 2$. For instance, case (1) should correspond here to case (1) of Conjecture \ref{Conjheight2}:

\begin{conjecture} \label{Conjgoodformsheight2} Let $\mathfrak{b}$ be an anisotropic bilinear form over $F$ such that $h(\mathfrak{b}) = \qp{h}(\phi_{\mathfrak{b}}) = 2$ and $\mydim{\mathfrak{b}} = 2^{d + 1}$, where $\mathrm{deg}(\mathfrak{b}) = d$. Then $\mathfrak{b} = \pi \otimes \mathfrak{c}$ for some $(d-1)$-fold Pfister form $\pi$ and some 4-dimensional form $\mathfrak{c}$ of nontrivial determinant. \end{conjecture}

This conjecture is also directly analogous to Conjecture \ref{Conjh2good}. In fact, we have:

\begin{lemma} Let $\mathfrak{b}$ be an anisotropic bilinear form over $F$ such that $h(\mathfrak{b}) = \qp{h}(\phi_{\mathfrak{b}}) = 2$ and $\mydim{\mathfrak{b}} = 2^{d + 1}$, where $\mathrm{deg}(\mathfrak{b}) = d$. Then the following are equivalent:
\begin{enumerate} \item Conjecture \ref{Conjgoodformsheight2} holds for $\mathfrak{b}$.
\item Conjecture \ref{Conjh2good} holds for the pair $(\phi_{\mathfrak{b}},d)$.
\item There exists a subform $\psi \subset \phi_{\mathfrak{b}}$ of codimension $2^{d-1}$ such that $\mathrm{ndeg}(\psi) = 2^{d-1}$. \end{enumerate}
\begin{proof} The equivalence of (2) and (3) has already been shown in Lemma \ref{Lemh2good}, while the implication $(1) \Rightarrow (2)$ is trivial. We now sketch the proof of $(2) \Rightarrow (1)$, making free use of known results from the theory of symmetric bilinear forms in characteristic 2. Without loss of generality, we may assume that $1 \in D(\phi_{\mathfrak{b}})$. Suppose now that (2) holds, i.e., that $\phi_\mathfrak{b} \simeq \sigma \otimes \tau$ for some $(d-1)$-fold quasi-Pfister form $\sigma$ and some form $\tau$ of dimension 4. We claim that $\mathfrak{b}_{F(\sigma)} \sim 0$. Since $h(\mathfrak{b}) = 2$, the Arason-Pfister Hauptsatz (which is a special case of Theorem \ref{ThmVishikholes}) implies that it will be sufficient to check that $\witti{0}{\mathfrak{b}_{F(\sigma)}} > \witti{1}{\mathfrak{b}}$. By Corollary \ref{CorBilindex}, we have $\witti{0}{\mathfrak{b}_{F(\sigma)}} \geq \frac{1}{2}\witti{0}{(\phi_{\mathfrak{b}})_{F(\sigma)}} = \frac{1}{2}(\frac{1}{2}\mydim{\phi_{\mathfrak{b}}}) = 2^{d-1} = \witti{1}{\mathfrak{b}}$. If equality holds here, then we necessarily have that:
\begin{enumerate} \item $\anispart{(\mathfrak{b}_{F(\sigma)})}$ is a $d$-fold Pfister form.
\item $\phi_{\anispart{(\mathfrak{b}_{F(\sigma)})}} \simeq \anispart{((\phi_{\mathfrak{b}})_{F(\sigma)})}$. \end{enumerate} 
But this implies that $\mathrm{ndeg}((\phi_{\mathfrak{b}})_{F(\sigma)})) = 2^d = \frac{1}{4}\mathrm{ndeg}(\phi_{\mathfrak{b}})$, which contradicts Lemma \ref{Lemisotropyff} (3). Our claim is therefore proved. Now, by \cite[Cor. 3.3]{AravireBaeza1}, it follows that there exists an anisotropic $d$-fold Pfister form $\eta$ over $F$ such that $\mathrm{deg}(\mathfrak{b} \perp \eta) \geq d+1$. But, since $1 \in D(\phi_{\mathfrak{b}}) \cap D(\phi_{\eta})$, $\mathfrak{b} \perp \eta$ is isotropic, and so $\mydim{\anispart{(\mathfrak{b} \perp \eta)}} < 2^{d+1} + 2^d$. By Theorem \ref{ThmVishikholes}, it follows that $\anispart{(\mathfrak{b} \perp \eta)}$ is similar to an $(d+1)$-fold Pfister form. By a standard linkage result (see \cite[Prop. 6.21]{EKM}), we deduce that $\mathfrak{b} \simeq \pi \otimes \mathfrak{c}$ for some $(d-1)$-fold Pfister form $\pi$ and 4-dimensional form $\mathfrak{c}$ over $F$. As $\mathfrak{b}$ is not similar to a Pfister form, $\mathfrak{c}$ necessarily has nontrivial determinant. \end{proof} \end{lemma}

Similarly, one would expect case (2) of Corollary \ref{Corheight2even} to correspond to case (2) of Conjecture \ref{Conjheight2}. Note here, however, that if $\mathfrak{b}$ is either (i) the product of a $(d-1)$-fold Pfister form and a 4-dimensional form of nontrivial determinant, or (ii) a generalised Albert form of dimension $2^d + 2^{d-1}$, then $\phi_{\mathfrak{b}}$ can be a quasi-Pfister neighbour, and so $\mathfrak{b}$ may, in fact, belong to the class of forms described in Corollary \ref{Corheight2even} (3). Furthermore, it is appropriate to separate the latter case from cases (1) and (2). Indeed, for any $1 \leq i \leq 2^{d-1}$, one can easily construct forms of dimension $2^{d+1} - 2^i$, degree $d$ and height 2 as the difference of two $d$-fold Pfister forms having $(i-1)$ common slots. Case (3) of Corollary \ref{Corheight2even} can therefore occur in every permitted dimension. We thus have an entirely new class of height 2 forms of even dimension in characteristic 2. Moreover, as pointed out by Laghribi (see \cite[Thm. 1.1]{Laghribi7}), new forms of height 2 and degree $d \geq 1$ also appear in case (4) of Corollary \ref{Corheight2even}. These forms are of dimension $2^{n+1}$ ($n>d$) and are divisible by suitable Pfister forms of degree ${d-1}$. Of course, we also have here the direct analogues of those forms (of dimension $2^{n+1} - 2^d$) appearing in case (3) of Conjecture \ref{Conjheight2}. It seems to be unknown, however, if any other dimensions can occur in the situation of Corollary \ref{Corheight2even} (4). As for the classification of odd-dimensional anisotropic bilinear forms of height 2, we have: 

\begin{corollary} \label{Corheight2odd} Let $\mathfrak{b}$ be an anisotropic bilinear form of odd dimension over $F$ such that $h(\mathfrak{b}) = 2$. Then $\mydim{\mathfrak{b}_1} = 2^d-1$ for some $d\geq 2$, and exactly one of the following holds:
\begin{enumerate} \item $\mydim{\mathfrak{b}} = 5$, $\qp{h}(\phi_{\mathfrak{b}}) = 3$ and $\witti{1}{\phi_{\mathfrak{b}}} = \witti{2}{\phi_{\mathfrak{b}}} = 1$.
\item $\phi_\mathfrak{b}$ is a quasi-Pfister neighbour and $\mydim{\mathfrak{b}} = 2^{d+1} - 2^i \pm 1$ for some $1 \leq i \leq d$.
\item $\phi_\mathfrak{b}$ is a quasi-Pfister neighbour and $2^n -2^d - 1\leq \mydim{\mathfrak{b}} \leq 2^{n+1}-1$ for some $n>d$.\end{enumerate}
\begin{proof} Since $h(\mathfrak{b}) = 2$, $\mathfrak{b}_1$ is similar to a codimension-1 subform of a Pfister form by Proposition \ref{Propheight1bilinear}. As such, we have $\mydim{\mathfrak{b}_1} = 2^d - 1$ for some $d \geq 2$. Let $\mathfrak{c} = \anispart{(\mathfrak{b} \perp \form{\mathrm{det}(\mathfrak{b})}_b)}$. Note that, for any extension $L$ of $F$, we have $\mydim{\anispart{(\mathfrak{b}_L)}} = \mydim{\anispart{(\mathfrak{c}_L)}} \pm 1$. In particular, if $(F_r)$ denotes the Knebusch splitting tower of $\mathfrak{c}$, then $\mydim{\anispart{(\mathfrak{b}_{F_{h(\mathfrak{c})-1}})}} = 2^{\mathrm{deg}(\mathfrak{c})} \pm 1$. By the Arason-Pfister Hauptsatz (see Theorem \ref{ThmVishikholes}), it follows that $\mathrm{deg}(\mathfrak{c}) \geq d$. Now, write $\mydim{\mathfrak{b}} = 2^n + m$ for uniquely determined integers $n \geq d$ and $1 \leq m \leq 2^n$. By Corollary \ref{Corbilinearpattern}, one of the following holds:
\begin{enumerate} \item[$\mathrm{(i)}$] $\phi_{\mathfrak{b}_1}$ is a neighbour of $(\phi_{\mathfrak{b}})_1$.
\item[$\mathrm{(ii)}$] $\witti{1}{\mathfrak{b}} = \witti{1}{\phi_{\mathfrak{b}}} = \witti{2}{\phi_{\mathfrak{b}}} = 2^{\mathfrak{d}_1(\phi_\mathfrak{b})}$ and $\phi_{\mathfrak{b}_1}$ is a near neighbour of $(\phi_{\mathfrak{b}})_1$.
\item[$\mathrm{(iii)}$] $\phi_\mathfrak{b}$ is a quasi-Pfister neighbour, $m > 2^{n-1}$ and $2^{n-2} + \frac{m}{2} \leq \witti{1}{\mathfrak{b}} \leq m$.\end{enumerate}
If (i) holds, then we have $\mydim{(\phi_{\mathfrak{b}})_1} - \witti{1}{(\phi_{\mathfrak{b}})_1} < \mydim{\phi_{\mathfrak{b}_1}} \leq \mydim{(\phi_{\mathfrak{b}})_1}$. Since $\mydim{\phi_{\mathfrak{b}_1}} = 2^d -1$, Theorem \ref{ThmHoffmannsconjecture} (more precisely, Hoffmann and Laghribi's upper bound on $\mathfrak{i}_1$ applied to the form $(\phi_\mathfrak{b})_1$) then implies that we necessarily have $\mydim{(\phi_{\mathfrak{b}})_1} = 2^d$. Since $\phi_{\mathfrak{b}_1}$ is both a quasi-Pfister neighbour and a codimension-1 neighbour of $(\phi_{\mathfrak{b}})_1$, it follows that the latter form is similar to a quasi-Pfister form. By Corollary \ref{CorClassificationQPN}, this means that $\phi_\mathfrak{b}$ is a quasi-Pfister neighbour. Furthermore, we must have that $n = d$. Since $\mydim{\mathfrak{b}} = \mydim{\mathfrak{c}} \pm 1$, and since $\mathrm{deg}(\mathfrak{c}) \geq d$, it follows from Theorem \ref{ThmVishikholes} that $\mydim{\mathfrak{b}} = 2^{d+1} - 2^i \pm 1$ for some $1 \leq i \leq d$, and we are therefore in case (2). Now, if (ii) holds, then another application of Theorem \ref{ThmHoffmannsconjecture} shows that $\mydim{\mathfrak{b}} = 2^{d} +1$ and $\witti{1}{\mathfrak{b}} = \witti{1}{\phi_{\mathfrak{b}}} = \witti{2}{\phi_{\mathfrak{b}}} = 1$. If $d>2$, then Theorem \ref{ThmVishikholes} implies that $\mydim{\mathfrak{c}} = 2^d$, and so $\mathfrak{c}$ is a Pfister form. But this implies that $\phi_\mathfrak{b}$ is a quasi-Pfister neighbour, which contradicts the fact that $\witti{2}{\phi_{\mathfrak{b}}} = 1$ (see Corollary \ref{CorClassificationQPN}). We thus conclude that $d=2$, and so we are in case (1). Finally, as in the proof of Corollary \ref{Corheight2even}, the reader will immediately verify that we are in case (3) whenever (iii) holds. 
\end{proof} \end{corollary}

In characteristic $\neq 2$, it is expected that any odd-dimensional anisotropic quadratic form of height 2 is excellent of dimension $2^{n+1} - 2^d +1$ for some integers $n\geq d>1$, with one exceptional case occuring in dimension 5. Again, although we see obvious parallels between this conjectural statement and Corollary \ref{Corheight2odd}, we remark here that new values can appear among the dimensions of odd-dimensional forms of height 2 in characteristic 2. For example, it is easy to check that any codimension-1 subform of a generalised Albert form belonging to class (3) of Corollary \ref{Corheight2even} has height 2. Finally, let us note that, in the situation of Corollary \ref{Corheight2even} (resp. Corollary \ref{Corheight2odd}), we have $\qp{h}(\phi_{\mathfrak{b}}) \leq 2$ (resp. $\qp{h}(\phi_{\mathfrak{b}}) \leq 3$) in all cases. More generally, we have:

\begin{corollary} \label{Corqphbound} Let $\mathfrak{b}$ be an anisotropic bilinear form over $F$ such that $h(\mathfrak{b}) >1$.
\begin{enumerate} \item If $\mydim{\mathfrak{b}}$ is even, then $\qp{h}(\phi_{\mathfrak{b}}) \leq 2h(\mathfrak{b}) - 2$.
\item If $\mydim{\mathfrak{b}}$ is odd, then $\qp{h}(\phi_{\mathfrak{b}}) \leq 2h(\mathfrak{b}) - 1$. \end{enumerate}
\begin{proof} We proceed by induction on $h(\mathfrak{b})$. The case where $h(\mathfrak{b}) = 1$ follows from Proposition \ref{Propheight1bilinear} and Corollary \ref{CorClassificationQPN}. Suppose now that $h(\mathfrak{b}) > 1$. If $\qp{h}(\phi_{\mathfrak{b}}) \leq 1$, then there is nothing to prove. We can therefore assume that $\qp{h}(\phi_{\mathfrak{b}}) > 1$, or, equivalently, that $\phi_{\mathfrak{b}}$ is not a quasi-Pfister neighbour (see Corollary \ref{CorClassificationQPN}). By Corollary \ref{Corbilinearpattern}, it follows that $\phi_{\mathfrak{b}_1}$ is either a neighbour or a near neighbour of $(\phi_{\mathfrak{b}})_1$. Thus, by Proposition \ref{Propsspneighbours} and Corollaries \ref{Corsspnearneighbours} and \ref{CorClassificationQPN}, we have that $\qp{h}(\phi_{\mathfrak{b}_1}) \geq \qp{h}\big((\phi_{\mathfrak{b}})_1\big) - 1$. By the induction hypothesis, we therefore have
\begin{equation*} \qp{h}(\phi_{\mathfrak{b}}) = \qp{h}\big((\phi_{\mathfrak{b}})_1\big) + 1 \leq \qp{h}(\phi_{\mathfrak{b}_1}) + 2 \leq 2h(\mathfrak{b}_1) + 2 - \epsilon = 2h(\mathfrak{b}) - \epsilon, \end{equation*}
where $\epsilon = 2$ (resp. $\epsilon = 1$) if $\mydim{\mathfrak{b}}$ is even (resp. odd). This proves the result. \end{proof} \end{corollary}

\begin{remark} Standard consideration of ``generic'' forms shows that the bounds of Corollary \ref{Corqphbound} cannot be improved in general.  \end{remark}

\section{Full splitting of quasilinear quadratic forms} \label{Excellentconnections} 

Having determined all possible \emph{standard} splitting patterns of quasilinear $p$-forms (Theorem \ref{Thmallvalues}, Proposition \ref{Propexistenceofvalues}), we now turn our attention towards the problem of obtaining a similar result for the \emph{full} splitting pattern. Here, we restrict our considerations to the case of quasilinear \emph{quadratic} forms, where we can take direct inspiration from the following theorem of Vishik (which may be deduced from the existence of ``excellent connections'' in the integral Chow motives of anisotropic quadrics over fields of characteristic $\neq 2$ -- see \cite[Thm. 1.3]{Vishik5}):

\begin{theorem}[Vishik, \cite{Vishik5}] \label{ThmVishikstheorem} Let $\phi$ be an anisotropic quadratic form of dimension $\geq 2$ over a field $k$ of characteristic $\neq 2$, and write $\mydim{\phi} - \witti{1}{\phi} = 2^{r_1} - 2^{r_2} + \hdots + (-1)^{s-1}2^{r_s}$ for uniquely determined integers $r_1>r_2>\hdots>r_{s-1} > r_s + 1 >1$. Let $1 \leq l \leq s$, and put $D_l = \sum_{i=1}^{l-1}(-1)^{i-1}2^{r_i - 1} + \epsilon(l) \sum_{j=l}^s(-1)^{j-1}2^{r_j}$, where $\epsilon(l) = 1$ (resp. $\epsilon(l) = 0$) if $l$ is even (resp. odd). Then, for any field extension $L$ of $k$, we either have $\witti{0}{\phi_L} \geq D_l + \witti{1}{\phi}$ or $\witti{0}{\phi_L} \leq D_l$. 
\begin{proof} This is nothing else but a restatement of \cite[Prop. 2.6]{Vishik5} in terms of Witt indices. To see how it may be derived from \cite[Thm 2.1]{Vishik5} in further detail, we refer the reader to \cite[Proof of Thm. 1.2]{Scully2}. \end{proof}\end{theorem}

\begin{examples} \label{ExsExcellentconnections} Let $\phi$ be an anisotropic quadratic form of dimension $\geq 2$ over a field $k$ of characteristic $\neq 2$, and write $\mydim{\phi} = 2^n + m$ for uniquely determined integers $n \geq 0$ and $1 \leq m \leq 2^n$.
\begin{enumerate} \item For $l=1$, Theorem \ref{ThmVishikstheorem} asserts that $\witti{0}{\phi_L} \geq \witti{1}{\phi}$ whenever $\phi_L$ is isotropic.
\item For $l=2$, Theorem \ref{ThmVishikstheorem} asserts that, for any field extension $L$ of $k$, we either have $\witti{0}{\phi_L} \geq m$ or $\witti{0}{\phi_L} \leq m - \witti{1}{\phi_L}$. 
\item For $l=s$, Theorem \ref{ThmVishikstheorem} asserts that, for any field extension $L$ of $k$, we either have $\witti{0}{\phi_L} \geq \frac{1}{2}(\mydim{\phi}+\witti{1}{\phi}-2^{r_s})$ or $\witti{0}{\phi_L} \leq \frac{1}{2}(\mydim{\phi}-\witti{1}{\phi}-2^{r_s})$. Taking $L = \overline{k}$ (so that $\witti{0}{\phi_L} = \big[\frac{1}{2}\mydim{\phi}\big]$), we see that $\witti{1}{\phi}\leq 2^{r_s}$, which gives a proof of Hoffmann's Conjecture \ref{ConjHoffmanns} in this setting (see \cite[Thm 2.5.]{Vishik5}) \end{enumerate} \end{examples}

We expect, in fact, that Theorem \ref{ThmVishikstheorem} extends verbatim to our setting. Henceforth, let us assume that $p=2$. We state the following:

\begin{conjecture} \label{ConjExcellentconnections} Let $\phi$ be an anisotropic quasilinear quadratic form of dimension $\geq 2$ over $F$, and write $\mydim{\phi} - \witti{1}{\phi} = 2^{r_1} - 2^{r_2} + \hdots + (-1)^{s-1}2^{r_s}$ for uniquely determined integers $r_1>r_2>\hdots>r_{s-1} > r_s + 1 >1$. Let $1 \leq l \leq s$, and put $D_l = \sum_{i=1}^{l-1}(-1)^{i-1}2^{r_i - 1} + \epsilon(l) \sum_{j=l}^s(-1)^{j-1}2^{r_j}$, where $\epsilon(l) = 1$ (resp. $\epsilon(l) = 0$) if $l$ is even (resp. odd). If $L$ is any field extension of $F$, then we either have $\witti{0}{\phi_L} \geq D_l + \witti{1}{\phi}$ or $\witti{0}{\phi_L} \leq D_l$. \end{conjecture}

This expectation is partly justified by:

\begin{proposition} Conjecture \ref{ConjExcellentconnections} holds for $l \leq 2$.
\begin{proof} As in Example \ref{ExsExcellentconnections} (1) (resp. (2)), the $l=1$ (resp. $l=2$) case is nothing else but Lemma \ref{Lemminimalityi1} (resp. Theorem \ref{Thmouterexcellentconnections}). \end{proof} \end{proposition}

At present, we do not have a general approach to the $l>2$ case of Conjecture \ref{ConjExcellentconnections}. Using Theorem \ref{ThmHoffmannsconjecture}, however, we can provide further evidence for the $l=s$ case. First, it is worth stating here the following lemma:

\begin{lemma} \label{LemExcellentconnections} In order to prove Conjecture \ref{ConjExcellentconnections}, we may assume that $\witti{1}{\phi_L} \geq \witti{1}{\phi}$.
\begin{proof} Suppose that the statement of the conjecture fails to hold. In other words, suppose that $\witti{1}{\phi} > 1$ and $\witti{0}{\phi_L} = D_l + t$ for some $1 \leq t < \witti{1}{\phi}$. Note that we necessarily have $r_s>0$ by Theorem \ref{ThmHoffmannsconjecture}. Now, let $\sigma = \anispart{(\phi_L)}$. Since $D(\sigma)$ is spanned by elements of $D(\phi)$, there exists a subform $\rho \subset \phi$ such that $\sigma \simeq \rho_L$ (see Lemma \ref{Lemexistenceofforms}). Let $\psi$ be any codimension-($t-1$) subform of $\phi$ containing $\rho$. By construction, we have $D(\psi_L) \subseteq D(\phi_L) = D(\sigma) = D(\rho_L) \subseteq D(\psi_L)$, whence $D(\psi_L) = D(\sigma)$, or, equivalently, $\anispart{(\psi_L)} \simeq \sigma$. In particular, we have $\witti{0}{\psi_L} = D_l + 1$. On the other hand, since $t<\witti{1}{\phi}$, $\psi$ is a neighbour of $\phi$, and so $\mydim{\psi} - \witti{1}{\psi} = \mydim{\phi} - \witti{1}{\phi} = 2^{r_1} - 2^{r_2} + \hdots + (-1)^{s-1}2^{r_s}$ (Proposition \ref{Propsspneighbours}). We therefore conclude that the statement of the conjecture also fails for the triple $(\psi,l,L)$. Since we are looking to produce a contradiction, we can replace $\phi$ by $\psi$ in order to arrive at the case where $t=1$. In this case, we claim that $\witti{1}{\phi_L} \geq \witti{1}{\phi}$. To see this, note first that $L(\phi)$ is $L$-isomorphic to a purely transcendental extension of $L(\sigma)$ (Remark \ref{Remsff} (3)). In view of Lemma \ref{Lemseparableextensions}, it follows that $\witti{1}{\phi_L} = \witti{1}{\sigma_{L(\phi)}}$. In particular, we have $\witti{0}{\phi_{L(\phi)}} = \witti{0}{\phi_L} + \witti{0}{\sigma_{L(\phi)}} = \witti{0}{\phi_L} + \witti{1}{\phi_L}$. The statement of Proposition \ref{Propcomparison} may therefore be rewritten here as 
\begin{equation} \label{eq8.1} \witti{1}{\phi_L} + \witti{0}{\phi_L} - \witti{1}{\phi} \geq \mathrm{min} \bigg\lbrace \witti{0}{\phi_L},\bigg[\frac{\mydim{\phi} - \witti{1}{\phi} + 1}{2}\bigg]\bigg \rbrace. \end{equation}
But, since $r_s>0$, we have $\witti{0}{\phi_L} = D_l + 1 \leq 2^{r_1-1} - 2^{r_2 -1} + \hdots + (-1)^{s-1}2^{r_s-1} = \frac{1}{2}\big(\mydim{\phi} - \witti{1}{\phi}\big)$. Inequality \eqref{eq8.1} therefore yields the desired assertion, and so the lemma is proved. \end{proof} \end{lemma}

We can now prove:

\begin{proposition} \label{Propmaximalec} Conjecture \ref{ConjExcellentconnections} holds when $l=s$ and $L=F(Q)$ is the function field of any integral (affine or projective) quadric $Q$ over $F$.
\begin{proof} In view of Lemma \ref{Lemseparableextensions}, the statement of the conjecture is stable under replacing $F$ by any separable extension of itself. We may therefore assume that $L$ is a purely inseparable quadratic extension of $F$. By Lemma \ref{LemExcellentconnections}, we may also assume that $\witti{1}{\phi_L} \geq \witti{1}{\phi}$. Suppose now that the statement fails to hold, so that $\witti{1}{\phi} > 1$ and $\witti{0}{\phi_L} = D_s + t$ for some $1 \leq t < \witti{1}{\phi}$. We then have
\begin{equation*} \mydim{\anispart{(\phi_L)}} = \begin{cases} 2^{r_1 - 1} - 2^{r_2-1} + \hdots - 2^{r_{s-2}-1} + 2^{r_{s-1} - 1} + \witti{1}{\phi} - t & \text{if $s$ is even}\\   
2^{r_1 - 1} - 2^{r_2-1} + \hdots - 2^{r_{s-1} - 1} + 2^{r_s} + \witti{1}{\phi} - t & \text{if $s$ is odd.} \end{cases} \end{equation*}
Since $\witti{1}{\phi_L} \geq \witti{1}{\phi}$, and since $1 \leq \witti{1}{\phi} - t <\witti{1}{\phi}$, Theorem \ref{ThmHoffmannsconjecture} implies that 
\begin{equation*} \witti{1}{\phi_L} \geq \begin{cases} 2^{r_{s-1} - 1} + \witti{1}{\phi} - t & \text{if $s$ is even}\\   
2^{r_s} + \witti{1}{\phi} - t & \text{if $s$ is odd}, \end{cases} \end{equation*}
from which we conclude that
\begin{equation*} \witti{0}{\phi_L} + \witti{1}{\phi_L} \geq \begin{cases} 2^{r_1-1} - 2^{r_2-1} + \hdots + 2^{r_{s-1}-1} + \witti{1}{\phi} & \text{if $s$ is even}\\   
2^{r_1-1} - 2^{r_2-1} + \hdots + 2^{r_{s-2}-1} + \witti{1}{\phi} & \text{if $s$ is odd}. \end{cases} \end{equation*}
But $\witti{0}{\phi_L} + \witti{1}{\phi_L} = \witti{0}{\phi_{L(\phi)}} = \witti{1}{\phi} + \witti{0}{(\phi_1)_{L(\phi)}}$ (see the proof of Lemma \ref{LemExcellentconnections}), and so we have
\begin{equation*} \witti{0}{(\phi_1)_{L(\phi)}} \geq \begin{cases} 2^{r_1-1} - 2^{r_2-1} + \hdots + 2^{r_{s-1}-1} & \text{if $s$ is even}\\   
2^{r_1-1} - 2^{r_2-1} + \hdots + 2^{r_{s-2}-1} & \text{if $s$ is odd}. \end{cases} \end{equation*}
Either way, we see that $\witti{0}{(\phi_1)_{L(\phi)}} > 2^{r_1-1} - 2^{r_2-1} + \hdots + (-1)^{s-1}2^{r_s-1} = \frac{1}{2}\big(\mydim{\phi} - \witti{1}{\phi}\big) = \frac{1}{2}\mydim{\phi_1}$. Since $L(\phi)$ is a purely inseparable quadratic extension of $F(\phi)$, this is impossible by Lemma \ref{LemIsotropypurelyinseparable} (5), and so the result follows. \end{proof} \end{proposition}

In general, it suffices to prove Conjecture \ref{ConjExcellentconnections} in the case where $L$ is a finite purely inseparable extension of $F$. Proposition \ref{Propmaximalec} shows that the $l=s$ case of the conjecture holds for degree 2 purely inseparable extensions. More generally, the proposition covers the case where $\mathrm{ndeg}(\phi_L) = \frac{1}{2}\mathrm{ndeg}(\phi_L)$. Indeed, using Lemma \ref{Lemseparableextensions} and \ref{LemIsotropypurelyinseparable} (3), one can easily reduce this case to that of a quadratic extension.

\appendix

\section{Symmetric bilinear forms in characteristic 2} \label{AppBilinearforms} We collect here some some basic facts from the theory of symmetric bilinear forms over fields of characteristic 2 which are needed in \S \ref{Bilinearforms} above. Throughout this section, we assume that $p=2$.

\subsection{Basic notions} \label{BilBasicnotions} Let $\mathfrak{b}$ be a symmetric bilinear form on a finite-dimensional $F$-vector space $V$. We will say that $\mathfrak{b}$ is \emph{nondegenerate} if the $F$-linear map $l \colon V \rightarrow V^\vee$ which sends $v \in V$ to $l(v) \colon w \mapsto \mathfrak{b}(v,w)$ is bijective, and that $\mathfrak{b}$ is \emph{alternating} if $\mathrm{dim}_kV \geq 1$ and $\mathfrak{b}(v,v) = 0$ for all $v \in V$. In what follows, a \emph{symmetric bilinear form over $F$} will always mean a \emph{nondegenerate} symmetric bilinear form on some finite-dimensional $F$-vector space.

Let $\mathfrak{b}$ be a symmetric bilinear form over $F$. The underlying $F$-vector space of $\mathfrak{b}$ will be denoted by $V_{\mathfrak{b}}$. Its will be called the \emph{dimension} of $\mathfrak{b}$ and will be denoted by $\mydim{\mathfrak{b}}$. Given a field extension $L$ of $F$, we will write $\mathfrak{b}_L$ for the unique symmetric bilinear form on the $L$-vector space $V_{\mathfrak{b}}\otimes_F L$ such that $\mathfrak{b}_L(v\otimes 1, w\otimes 1) = \mathfrak{b}(v,w)$ for all $(v,w) \in V_{\mathfrak{b}} \times V_{\mathfrak{b}}$. The \emph{determinant} of $\mathfrak{b}$, denoted $\mathrm{det}(\mathfrak{b})$, is defined as the image of $\mathrm{det}(l)$ in the square class group $F^*/(F^*)^2$ (where $l \colon V \xrightarrow{\sim} V^\vee$ is as above).

Let $\mathfrak{c}$ be another symmetric bilinear form over $F$. If there exists a bijective $F$-linear map $f \colon V_{\mathfrak{b}} \rightarrow V_{\mathfrak{c}}$ such that $\mathfrak{c}(f(v),f(w)) = \mathfrak{b}(v,w)$ for all $(v,w) \in V_{\mathfrak{b}} \times V_{\mathfrak{b}}$, then we say that $\mathfrak{b}$ and $\mathfrak{c}$ are \emph{isomorphic} and write $\mathfrak{b} \simeq \mathfrak{c}$. Note that in this case, we necessarily have $\mathrm{det}(\mathfrak{b}) = \mathrm{det}(\mathfrak{c})$. If $\mathfrak{b} \simeq a\mathfrak{c}$ for some $a \in F^*$ (where $a\mathfrak{c}$ denotes the symmetric bilinear form on $V_{\mathfrak{c}}$ given by $(v,w) \mapsto a\mathfrak{c}(v,w)$), then we say that $\mathfrak{b}$ and $\mathfrak{c}$ are similar. The \emph{sum} $\mathfrak{b} \perp \mathfrak{c}$ (resp. the \emph{product} $\mathfrak{b} \otimes \mathfrak{c}$) is defined as the unique symmetric bilinear form on $V_{\mathfrak{b}} \oplus V_{\mathfrak{c}}$ (resp. $V_{\mathfrak{b}} \otimes_F V_{\mathfrak{c}}$) such that $(\mathfrak{b} \perp \mathfrak{c})\big((v_1,v_2),(w_1,w_2)\big) = \mathfrak{b}(v_1,w_1) + \mathfrak{c}(v_2,w_2)$ (resp. $(\mathfrak{b}\otimes \mathfrak{c})(v_1\otimes v_2,w_1 \otimes w_2) = \mathfrak{b}(v_1,w_1)\mathfrak{c}(v_2,w_2)$) for all $(v_1,w_1) \in V_{\mathfrak{b}} \times V_{\mathfrak{b}}$ and all $(v_2,w_2) \in V_{\mathfrak{c}} \times V_{\mathfrak{c}}$. Given a positive integer $n$, $n \cdot \mathfrak{b}$ will denote the orthogonal sum of $n$ copies of $\mathfrak{b}$ (this is not the same as $n\mathfrak{b}$). If there exists a symmetric bilinear form $\mathfrak{d}$ over $F$ such that $\mathfrak{b} = \mathfrak{c} \perp \mathfrak{d}$ (resp. $\mathfrak{b} = \mathfrak{c} \otimes \mathfrak{d}$), then we will say that $\mathfrak{c}$ is a \emph{subform} of $\mathfrak{b}$ (resp. that $\mathfrak{b}$ is \emph{divisible} by $\mathfrak{c}$).

Given elements $a_1,\hdots,a_n \in F^*$, we write $\form{a_1,\hdots,a_n}_b$ for the symmetric bilinear form $\big((\mu_1,\hdots,\mu_n),(\lambda_1,\hdots,\lambda_n)\big) \mapsto \sum_{i=1}^na_i\mu_i\lambda_i$ on the $F$-vector space $F^{\oplus n}$. For any $a\in F^*$, the 2-dimensional symmetric bilinear form $\form{a,a}_b$ will simply be denoted by $\mathbb{M}_a$. Forms of this type are called \emph{metabolic planes}. The \emph{hyperbolic plane}, denoted $\mathbb{H}$, is defined as the 2-dimensional symmetric bilinear form $\big((v_1,v_2),(w_1,w_2)\big) \mapsto v_1w_2 + v_2w_1$ on the $F$-vector space $F \oplus F$.

The assignment $v \mapsto \mathfrak{b}(v,v)$ defines a quadratic form $\phi_{\mathfrak{b}}$ on $V_{\mathfrak{b}}$. If $\mydim{\mathfrak{b}} \geq 1$ and $\mathfrak{b}$ is \emph{not} alternating, then $\phi_{\mathfrak{b}}$ is quasilinear in the sense of \S \ref{Basicnotions}. Conversely, it is clear that every quasilinear quadratic form over $F$ is isomorphic the diagonal part of some non-alternating symmetric bilinear form defined on its underlying space. Henceforth, by a \emph{bilinear form over $F$}, we will mean a symmetric bilinear form over $F$ which is \emph{not alternating}. Note that this class includes all forms of the shape $\form{a_1,\hdots,a_n}_b$, but does not include the hyperbolic plane $\mathbb{H}$.

Let $\mathfrak{b}$ be a bilinear form over $F$. We say that $\mathfrak{b}$ is \emph{isotropic} if $\mydim{\mathfrak{b}} \geq 1$ and $\phi_{\mathfrak{b}}$ isotropic, and that $\mathfrak{b}$ is \emph{anisotropic} otherwise. We can now state the following refined version of the classical Witt decomposition theorem:

\begin{proposition}[{cf. \cite{Milnor}, \cite[Prop. 5.15]{Laghribi6}}] \label{PropBildecomposition} Let $\mathfrak{b}$ be a bilinear form over $F$. Then there exist a unique pair of nonnegative integers $(s,t)$, elements $a_1,\hdots,a_t \in F^*$ and, up to isomorphism, a unique anisotropic bilinear form $\anispart{\mathfrak{b}}$ such that:
\begin{enumerate} \item $\mathfrak{b} \simeq \anispart{\mathfrak{b}} \perp (s\cdot \mathbb{H}) \perp \mathbb{M}_{a_1} \perp \hdots \perp \mathbb{M}_{a_t}$.
\item $\anispart{(\phi_{\mathfrak{b}})} \simeq \phi_{\anispart{\mathfrak{b}}} \perp \form{a_1,\hdots,a_t}$. \end{enumerate} \end{proposition}

The form $\anispart{b}$ is called the \emph{anisotropic kernel} of $\mathfrak{b}$. If $\mydim{\anispart{\mathfrak{b}}} \leq 1$, then we will say that $\mathfrak{b}$ is \emph{split}. Given another bilinear form $\mathfrak{c}$ over $F$, we will write $\mathfrak{b} \sim \mathfrak{c}$ whenever $\anispart{\mathfrak{b}} \simeq \anispart{\mathfrak{c}}$. The integer $s+t$ appearing in the statement of Proposition \ref{PropBildecomposition} is called the \emph{Witt index} of $\mathfrak{b}$, and is denoted $\witti{0}{\mathfrak{b}}$. Clearly $\witti{0}{\mathfrak{b}}$ is bounded from above by $\mydim{\mathfrak{b}}/2$. At the same time, the proposition immediately yields:

\begin{corollary} \label{CorBilindex} Let $\mathfrak{b}$ be a bilinear form over $F$. Then $\frac{1}{2}\witti{0}{\phi_{\mathfrak{b}}} \leq  \witti{0}{\mathfrak{b}} \leq \witti{0}{\phi_{\mathfrak{b}}}$. \end{corollary}

The most important examples of bilinear forms are the \emph{Pfister forms}. Here, we say that $\mathfrak{b}$ is an \emph{$n$-fold} (resp. \emph{$1$-fold}) \emph{Pfister form} if $\mathfrak{b} \simeq \pfister{a_1,\hdots,a_n}_b \coloneqq \form{1,a_1}_b \otimes \hdots \otimes \form{1,a_n}_b$ for some $a_1,\hdots,a_n \in F^*$ (resp. $\mathfrak{b} \simeq \form{1}_b$). Note that if $\mathfrak{b}$ is a Pfister form, then we can write $\mathfrak{b} \simeq \form{1}_b \perp \mathfrak{b}'$ for some bilinear form $\mathfrak{b}'$ over $F$. If $\mathfrak{b}$ is additionally \emph{anisotropic}, then it follows from the Witt cancellation theorem (see \cite[Cor. 1.28]{EKM}) that $\mathfrak{b}'$ is uniquely determined up to isomorphism. In this case $\mathfrak{b'}$ is called the \emph{pure subform} of $\mathfrak{b}$. Finally, note that if $\mathfrak{b} \simeq \pfister{a_1,\hdots,a_n}_b$ (resp. $\mathfrak{b} \simeq \form{1}_b$), then the associated quasilinear quadratic form $\phi_{\mathfrak{b}}$ is quasi-Pfister (see \S \ref{QPforms}), being isomorphic to $\pfister{a_1,\hdots,a_n}$ (resp. $\form{1}$).

\subsection{Knebusch splitting of symmetric bilinear forms in characteristic 2} \label{SSPbilinear} In \cite{Laghribi6}, Laghribi has initiated a splitting theory of bilinear forms in characteristic 2 by way of analogy with the Knebusch splitting theory of quadratic forms. More explicitly, one proceeds here as follows:

Let $\mathfrak{b}$ be a bilinear form over $F$. Set $F_0 = F$, $\mathfrak{b}_0 = \anispart{\mathfrak{b}}$, and inductively define $F_r = F_{r-1}(\phi_{\mathfrak{b}_{r-1}})$, $\mathfrak{b}_r = \anispart{(\mathfrak{b}_{F_r})}$, with the understanding that this (finite) process stops when we reach the first non-negative integer $h(\mathfrak{b})$ such that $\mathfrak{b}_{h(\mathfrak{b})}$ is split. The integer $h(\mathfrak{b})$ and the tower of fields $F = F_0 \subset F_1 \subset \hdots \subset F_{h(\mathfrak{b})}$ will be called the \emph{height} and \emph{Knebusch splitting tower} of $\mathfrak{b}$, respectively. For $1 \leq r \leq h(\mathfrak{b})$, the anisotropic form $\mathfrak{b}_r$ is called the \emph{r-th higher anisotropic kernel} of $\phi$. The \emph{r-th higher Witt index of $\mathfrak{b}$}, denoted $\witti{r}{\mathfrak{b}}$, is defined as the difference $\witti{0}{\mathfrak{b}_{k_r}} - \witti{0}{\mathfrak{b}_{k_{r-1}}}$. The sequence $\mathfrak{i}(\mathfrak{b}) = \big(\witti{1}{\mathfrak{b}},\hdots,\witti{h(\mathfrak{b})}{\mathfrak{b}}\big)$ will be called the \emph{Knebusch splitting pattern} of $\mathfrak{b}$\footnote{See Footnote \ref{Footnote3}.}. Note that, by the inductive nature of the construction, we have $\witti{r}{\mathfrak{b}} = \witti{1}{\mathfrak{b}_{r-1}}$ for every $1 \leq r \leq h(\mathfrak{b})$. Since $F_1 = F(\phi_{\anispart{\mathfrak{b}}})$, Corollary \ref{CorBilindex} immediately implies:

\begin{lemma} \label{Lembilineari1} Let $\mathfrak{b}$ be an anisotropic bilinear form of dimension $\geq 2$ over $F$. Then $\frac{1}{2}\witti{1}{\phi_{\mathfrak{b}}} \leq \witti{1}{\mathfrak{b}} \leq \witti{1}{\phi_{\mathfrak{b}}}$. \end{lemma}

\begin{remark} \label{Remunclear} To the author's knowledge, there are no known results indicating the precise nature of the relationship between $\witti{1}{\mathfrak{b}}$ and $\witti{1}{\phi_{\mathfrak{b}}}$.  \end{remark}

By similar reasoning, we also have:

\begin{lemma} \label{Lembilheight} Let $\mathfrak{b}$ be a bilinear form over $F$. Then $h(\mathfrak{b}) \leq h(\phi_{\mathfrak{b}})$.
\begin{proof} We proceed by induction on $\mydim{\mathfrak{b}}$. The case where $\mydim{\mathfrak{b}} = 1$ is trivial. Suppose now that $\mydim{\mathfrak{b}} >1$, and let $L$ be any extension of $F$ such that $\mathfrak{b}_L$ is isotropic. Then, by the induction hypothesis, we have $h(\mathfrak{b}_L) = h\big(\anispart{(\mathfrak{b}_L)}\big) \leq h(\phi_{\anispart{(\mathfrak{b}_L)}})$. On the other hand, Proposition \ref{PropBildecomposition} (2) implies that $\phi_{\anispart{(\mathfrak{b}_L)}}$ is a subform of $\anispart{((\phi_{\mathfrak{b}})_L)}$, and so $h(\phi_{\anispart{(\mathfrak{b}_L)}}) \leq h\big(\anispart{((\phi_{\mathfrak{b}})_L)}\big)$ by Corollary \ref{Corheight=ndegree}. We therefore conclude that $h(\mathfrak{b}_L) \leq h\big(\anispart{((\phi_{\mathfrak{b}})_L)}\big)$. Now, if $\mathfrak{b}$ is split, then applying the above to $L=F$ yields the desired inequality $h(\mathfrak{b}) \leq h(\anispart{(\phi_{\mathfrak{b}})}) = h(\phi_{\mathfrak{b}})$. If $\mathfrak{b}$ is not split, then taking $L = F(\phi_{\anispart{\mathfrak{b}}})$, we compute
\begin{equation*} h(\mathfrak{b}) = h(\mathfrak{b}_1) + 1 = h(\mathfrak{b}_L) + 1 \leq h\big(\anispart{((\phi_{\mathfrak{b}})_L)}\big) + 1 = h((\phi_{\mathfrak{b}})_1) + 1 = h(\phi_{\mathfrak{b}}), \end{equation*}
and so the result also holds in this case. \end{proof} \end{lemma}

The classification of bilinear forms of height 1 in characteristic 2 has been established in by Laghribi in \cite{Laghribi6}. This result is directly analogous to the well-known classification of nonsingular quadratic forms of height 1 (see \cite[Prop 25.6, Ex. 28.2]{EKM}), and, in the even-dimensional case, to Proposition \ref{Propi1bound} above:

\begin{proposition}[{\cite[Thm. 4.1]{Laghribi6}}] \label{Propheight1bilinear} Let $\mathfrak{b}$ be an anisotropic bilinear form of even (rep. odd) dimension $\geq 2$ over $F$. Then $h(\mathfrak{b}) = 1$ if and only if $\mathfrak{b}$ is similar to a Pfister form (resp. the pure subform of a Pfister form). \end{proposition}

In particular, if $\mathfrak{c}$ is an even-dimensional bilinear form of height $>0$ over $F$, then $\mathfrak{c}_{h(\mathfrak{c})-1}$ is similar to a Pfister form, and therefore has dimension equal to a power of 2. This allows one to define the \emph{degree} $\mathrm{deg}(\mathfrak{b})$ of an arbitrary bilinear form $\mathfrak{b}$ over $F$ as follows:

\begin{enumerate} \item If $\mydim{\mathfrak{b}}$ is odd, then $\mathrm{deg}(\mathfrak{b}) = 0$.
\item If $\mydim{\mathfrak{b}}$ is even and $h(\mathfrak{b}) = 0$, then $\mathrm{deg}(\mathfrak{b}) = \infty$
\item If $\mydim{\mathfrak{b}}$ is even $h(\mathfrak{b}) > 0 $, then $\mathrm{deg}(\mathfrak{b}) = \mathrm{log}_2(\mydim{\mathfrak{b}_{h(\mathfrak{b})-1}})$. \end{enumerate}

We  conclude this appendix by mentioning the following deep result concerning the degree invariant, which may be deduced from its characteristic-0 analogue due to Karpenko (see \cite{Karpenko2}):

\begin{theorem}[{see \cite[Thm. 4.5, Prop. 5.7]{Laghribi6}}] \label{ThmVishikholes} Let $\mathfrak{b}$ be an anisotropic bilinear form of even dimension over $F$ such that $h(\mathfrak{b}) > 0$. Then either:
\begin{enumerate} \item $\mydim{\mathfrak{b}} \geq 2^{\mathrm{deg}(\mathfrak{b}) + 1}$, or
\item $\mydim{\mathfrak{b}} = 2^{\mathrm{deg}(\mathfrak{b})+1} - 2^i$ for some $1 \leq i \leq \mathrm{deg}(\mathfrak{b})$. \end{enumerate} \end{theorem}

{\bf Acknowledgements.} This work was carried out while I was a visiting postdoctoral fellow at Max-Planck-Institut f\"{u}r Mathematik in Bonn. I would like to thank this institution for its support and hospitality throughout the duration of my stay. \\

\bibliographystyle{alphaurl}

\end{document}